\documentclass[12pt,psamsfonts]{amsart}
\usepackage[margin=1.5in]{geometry}
\usepackage{wrapfig}
\usepackage{amssymb}
\usepackage{geometry}

\usepackage{float}

\usepackage{amsmath,amsfonts,amsthm,amstext}
\usepackage[latin1]{inputenc}
\usepackage{fancybox}
\usepackage{niceframe}
\usepackage{array}
\usepackage{newlfont}
\usepackage{verbatim}
\usepackage[dvips]{psfrag}
\usepackage{color}
\usepackage{float}
\usepackage{indentfirst}
\usepackage[normalem]{ulem}
\usepackage{pst-text}
\usepackage{graphicx}
\usepackage{graphics}
\usepackage{overpic}
\usepackage{hyperref}
\usepackage{pst-node}
\usepackage{tikz-cd} 
\usepackage{enumerate}

\graphicspath{{./Figures/}}
\usepackage{wrapfig}
\newcommand{\R}{\ensuremath{\mathbb{R}}}

\newcommand{\N}{\ensuremath{\mathbb{N}}}

\newcommand{\s}{\Sigma}

\newcommand{\e}{\varepsilon}

\newcommand{\V}{\mathcal{V}}

\newcommand{\Cr}{\mathcal{C}^{r}}
\newcommand{\Xr}{\chi^{r}}
\newcommand{\Or}{\Omega^{r}}
\newcommand{\rn}[1]{\mathbb{R}^{#1}}

\newcommand{\er}{\mathcal{O}}
\newcommand{\p}{\varphi}
\newcommand{\ag}{\alpha}
\newcommand{\bg}{\beta}
\newcommand{\cg}{\gamma}
\newcommand{\dg}{\delta}
\newcommand{\sgn}{\textrm{sgn}}

\newtheorem {prop}  {Proposition}
\newtheorem {corollary}  {Corollary}
\newtheorem {lemma}  {Lemma}
\newtheorem {definition}  {Definition}
\newtheorem {remark}  {Remark}
\newtheorem {example} {Example}

\makeatletter
\DeclareFontFamily{U}{tipa}{}
\DeclareFontShape{U}{tipa}{m}{n}{<->tipa10}{}
\newcommand{\arc@char}{{\usefont{U}{tipa}{m}{n}\symbol{62}}}%

\newcommand{\arc}[1]{\mathpalette\arc@arc{#1}}

\newcommand{\arc@arc}[2]{%
	\sbox0{$\m@th#1#2$}%
	\vbox{
		\hbox{\resizebox{\wd0}{\height}{\arc@char}}
		\nointerlineskip
		\box0
	}%
}
\makeatother

\definecolor{verde}{rgb}{0.0,0.5,0.0}
\definecolor{azul}{rgb}{0,0,128}
\definecolor{roxo}{rgb}{0.44,0.16,0.39}
\definecolor{vinho}{rgb}{0.5,0.0,0.13}
\definecolor{lilas1}{rgb}{0.6,0.33,0.73}
\definecolor{rosa}{rgb}{0.84,0.04,0.33}
\definecolor{mostarda}{rgb}{0.91,0.41,0.17}
\definecolor{mostarda2}{rgb}{1.0,0.66,0.07}

\newtheorem {mtheorem} {Theorem} 

\begin{document}

\title[Homoclinic-like Loops in 3D Filippov Systems]{On Typical Homoclinic-like loops in 3D Filippov Systems}

\author[O. M. L. Gomide]{Ot\'avio M. L. Gomide}
\address[OMLG]{Department of Mathematics, UFG, IME\\ Goi\^ania-GO, 74690-900, Brazil/ Department of Mathematics, Unicamp, IMECC\\ Campinas-SP, 13083-970, Brazil}
\email{otaviomleandro@gmail.com}

\author[M. A. Teixeira]{Marco A. Teixeira}
\address[MAT]{Department of Mathematics, Unicamp, IMECC\\ Campinas-SP, 13083-970, Brazil}
\email{teixeira@ime.unicamp.br}

%\author[M. Guardia]{Marcel Guardia}
%\address[MG]{ Departament de Matemàtica Aplicada I, Universitat Politècnica de Catalunya, Diagonal 647, 08028 Barcelona, Spain}
%\email{marcel.guardia@upc.edu}
%
%\author[T. M. Seara]{Tere M. Seara}
%\address[TS]{Departament de Matemàtica Aplicada I, Universitat Politècnica de Catalunya, Diagonal 647, 08028 Barcelona, Spain }
%\email{tere.m-seara@upc.edu }

\maketitle

\begin{abstract}
In this work a homoclinic-like loop of a piecewise smooth vector field passing through a typical singularity is analyzed. We have shown that such a loop is robust in one-parameter families of Filippov systems. The basin of attraction of this connection is computed as well as its bifurcation diagram. It is worthwhile to mention that this phenomenon has no counterpart in the smooth world and the techniques used in this analysis differ from the usual ones. 
\end{abstract}

\section{Introduction}

The study of global connections in smooth systems is a challenging problem which has been extensively studied throughout the last decades. In fact, the interest of the community in the detection of invariant minimal sets, as limit cycles and homoclinic/heteroclinic connections, has provided several tools which have contributed to the development of the Theory of Dynamical Systems. Nevertheless, despite considerable advances, there are still a lot of open problems concerning global phenomena, due to the richness and complexity of the dynamics associated to such objects. 

In the nonsmooth context, global phenomena can be responsible for  further complications in the dynamics of a piecewise smooth dynamical system. Roughly speaking, Filippov systems present new kinds of typical singular elements, such as the so-called $\s$-singularities (distiguished points of the switching manifold), which give rise to a rich extensive class of global connections having no counterparts in the smooth framework. The study of such objects present countless challenges and, maybe for such a reason, they are not frequently considered in the literature. It is worth mentioning that the comprehension of non-local connections between generic $\s$-singularities provides applications in several research lines in Piecewise Smooth Dynamical Systems as structural stability and generic bifurcation theory.

%
%
%
%plays a crucial role in the attempt to characterize the structurally stable $3D$ Filippov systems from a global point of view.

%
% due to its importance in the understanding of the dynamics of a smooth vector field. Roughly speaking, once the singular elements (singularities, limit cycles, etc.) of the system are detected, the dynamics of the system inside a region is generally determined by the existence or not of global connections between them.  
%
%In the nonsmooth context, one finds typical singular elements, such as the so-called $\s$-singularities, and thus, it gives rise to a rich extensive class of global connections which has no counterparts in the smooth framework. Also, the comprehension of non-local connections between generic $\s$-singularities plays a crucial role in the attempt to characterize the structurally stable $3D$ Filippov systems from a global point of view.

\subsection{Historical Facts}

Global connections to $\s$-singularities of planar Filippov systems have been studied in
\cite{BS16,AGN19,FPT15,LH13,PYZ13}.  More specifically, in  \cite{FPT15}, the bifurcation diagram of a loop at a generic $\s$-singularity, named fold-regular singularity (to be defined later), was described and, in \cite{BS16}, a study on a smoothing process of such an object was provided. It is worth mentioning that such loops also appeared in the  unfolding of degenerate phenomena (see \cite{BCT12}, \cite{NTZ18}). In \cite{AGN19}, loops passing through degenerate $\s$-singularities were considered and a method to deal with certain non-local structures in a general scenario was developed. 

As far as we know, homoclinic-like loops through a $\s$-singularity have not been treated for $3D$ Filippov systems. So, in light of the recent development in the planar case, we were encouraged to analyze analogous phenomena in dimension $3$, which usually present much more complexity than planar phenomena.
%
%
%planar systems in a general scenario is developed. As far as we know, this topic has not been treated for $3D$ Filippov systems in the literature. So, with the recent development of planar phenomena, it is natural to extend these studies to dimension $3$.
%
%
%In dimension $2$, there are plenty of works dealing with global connections to $\s$-singularities of Filippov systems. In fact, homoclinic-like loops at a generic $\s$-singularity have been studied in \cite{BS16,AGN19,FPT15,LH13,PYZ13}. More specifically, in  \cite{FPT15}, the bifurcation diagram of a loop at a fold-regular singularity (a kind of $\s$-singularity) is described and, in \cite{BS16}, a study on the regularization of such an object was provided. It is worth mentioning that such loops appear in the  unfolding of degenerate phenomena (see \cite{BCT12}, \cite{NTZ18}). In \cite{AGN19}, a method to deal with polycycles in planar system in a general scenario is developed. As far as we know, this topic has not been treated for $3D$ Filippov systems in the literature. So, with the recent development of planar phenomena, it is natural to extend these studies to dimension $3$.

%and they used such a mechanism to study the generic bifurcation of some codimension two global connections to $\s$-singularities. 

%As far as we know, the generic bifurcation of global connections to $\s$-singularities in $3D$ Filippov systems is not treated in the literature. So, with the recent development of planar phenomena, it is natural to extend these studies to dimension $3$.

\subsection{Description of the Main Results}

Now, a rough description of the results of this work is provided.  We consider Filippov systems of the form 
$$Z(p)=F(p)+\sgn(f(p)) G(p),\ p\in\rn{3},$$
where $F=\frac{X+Y}{2}$, $G=\frac{X-Y}{2}$, $f:\rn{3}\rightarrow\rn{}$ is a $\Cr$ function having $0$ as a regular value and $X,Y$ are $\Cr$ vector fields. In this case, we denote $Z=(X,Y)$ and $\s=h^{-1}(0)$  is the switching manifold of $Z$. For our purposes, it is sufficient to assume $r\geq 2$.

In this work, we study Filippov systems $Z=(X,Y)$ having a homoclinic-like loop $\Gamma$ at a \textbf{fold-regular singularity} $p$. Generally speaking, such a singularity happens when $X$ (resp. $Y$) has a quadratic contact with the switching manifold $\s$ at $p$ and $Y$ (resp. $X$) is transverse to $\s$ at $p$.  It is worthwhile to mention that, in dimension $3$, a fold-regular singularity $p$ is contained in a curve $\cg$ of fold-regular singularities of $Z$ in $\s$, named \textbf{fold curve}. Since $Z$ has a loop $\Gamma$ at $p$, the fold curve $\cg$ is mapped onto a curve $\zeta\subset\s$ through orbits of $X$ and $Y$ and they intersect at $p$ (see Figure \ref{generic_cycle_fig}). In this paper we focus on the case where $\gamma\cap\zeta$ transversally at $p$, ensuring that such a loop occurs in a robust scenario.

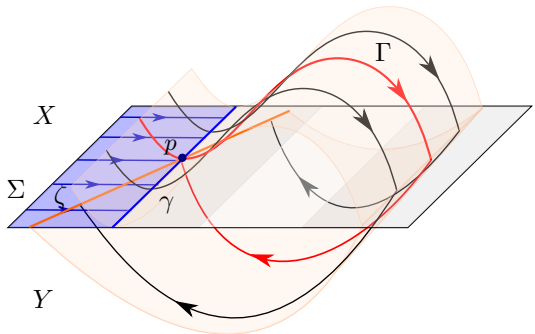
\begin{figure}[h!]
	\centering
	\bigskip
	\begin{overpic}[width=7cm]{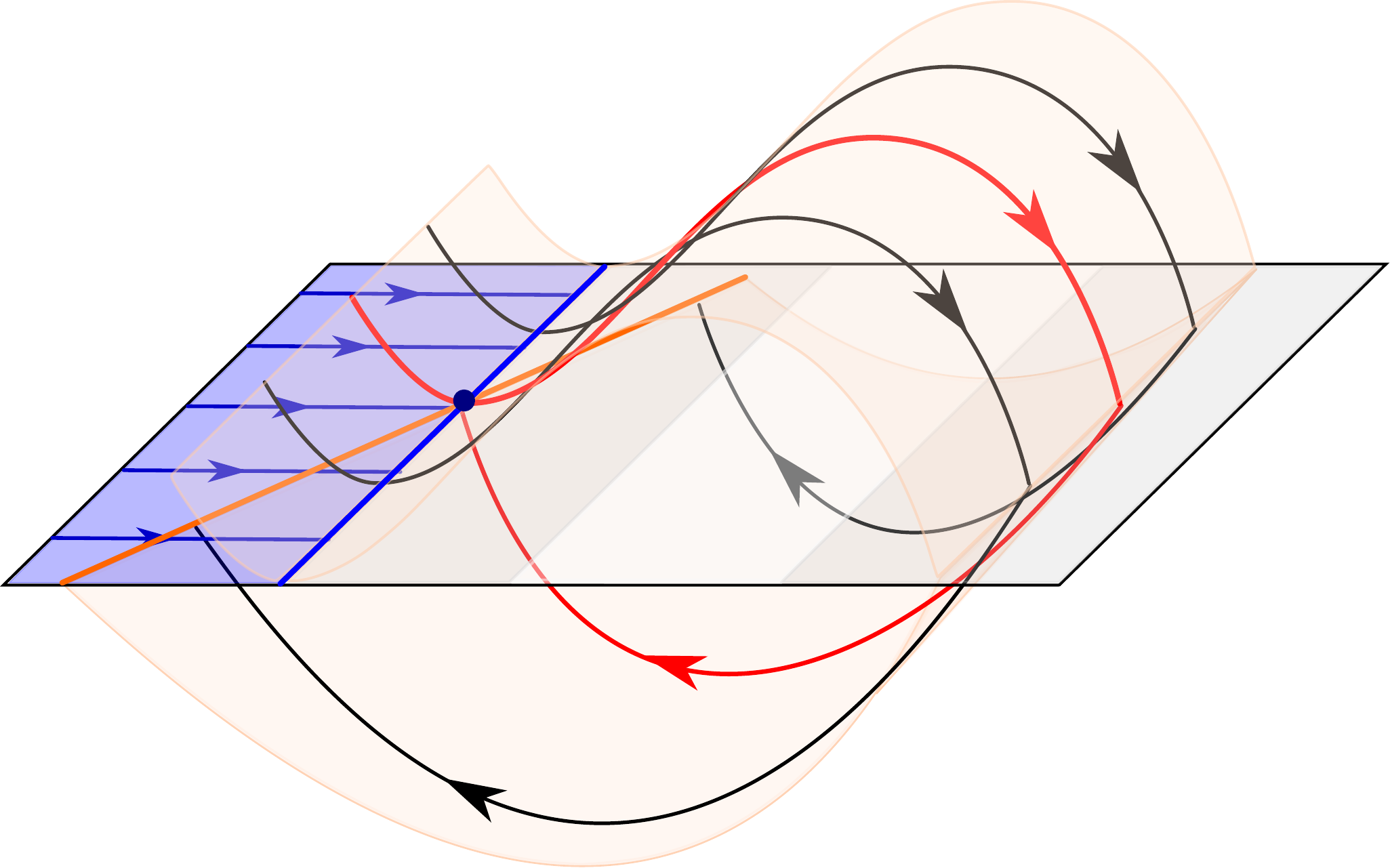}
		%							\begin{overpic}[grid,tics=5,width=13cm]{Figures/cycles.pdf}		
		\put(70,52){{\footnotesize $\Gamma$}}		
		\put(30,35){{\scriptsize $p$}}		
		%		\put(25,-3){{\footnotesize $(a)$}}		
		%		\put(75,-3){{\footnotesize $(b)$}}		
		%		\put(16.2,10.5){{\scriptsize $p_0$}}	
		%		\put(38,10.5){{\scriptsize $q_0$}}			
		%		\put(0,10){{\footnotesize $\s$}}	
		\put(0,26){{\footnotesize $\s$}}	
		\put(29,24){{\footnotesize $\cg$}}	
		\put(9,25){{\footnotesize $\zeta$}}	
		%		\put(26,22){{\footnotesize $\Gamma_0$}}			
		\put(5,5){{\footnotesize $Y$}}	
		\put(5,40){{\footnotesize $X$}}							
		
	\end{overpic}
	\bigskip
	\caption{A homoclinic-like loop $\Gamma$ of $Z$ at a fold-regular singularity $p$.}	\label{generic_cycle_fig}
\end{figure} 

Let $\Lambda_1$ be the class of Filippov systems $Z$ having a homoclinic-like loop $\Gamma$ at a fold-regular singularity $p$  satisfying the robustness scheme above (plus some technical conditions). We show that $Z\in\Lambda_1$ is generic in one-parameter families. We mean $\Lambda_1$ is a codimension one submanifold of the space of all $3D$ Filippov systems $\Or$.

The bifurcation diagram of $Z\in\Lambda_1$  around $\Gamma$ is exhibited  and the basin of attraction of $\Gamma$ is computed. It is worth mentioning that the use of sliding features of $Z$ is crucial for obtaining such results.

Finally, a notion of weak equivalence in $\Lambda_1$ is introduced and aspects of modulus of stability are discussed. Moreover, we conclude that there are infinitely many distinct topological types in $\Lambda_1$ under the weak equivalence relation.
%We prove that $Z_0\in\Lambda_1$ is generic in one-parameter families.% It means that, given a $\Cr$ family of Filippov systems $\mathcal{Z}(\lambda)$, $\lambda\in[-\e_0,\e_0]$,  such that $\mathcal{Z}(0)=Z_0\in\Lambda_1$,  then any one-parameter family $\widetilde{\mathcal{Z}}$ sufficiently near to $\mathcal{Z}$ (in the $\Cr$ topology) has a point $\lambda_0\in[-\e_0,\e_0]$ such that $\widetilde{\mathcal{Z}}(\lambda_0)\in\Lambda_1$.

%We prove that $Z_0\in\Lambda_1$ is generic in one-parameter families.We provide the bifurcation diagram of $Z_0\in\Lambda_1$ around $\Gamma_0$ under certain generic conditions. Also, we compute the basin of attraction of $\Gamma_0$. It is worth mentioning that, the use of sliding features of $Z_0$ is crucial to obtain these results.

%
%
%
%which is generic in $\Cr$ one-parameter families. We also provide the bifurcation diagram of such connection and we obtain a modulus of stability in $\Lambda_1$.
%
%
%We consider $3D$ Filippov systems having a global 
%connection 
%
%
%
%
%
%- Conexoes globais: importancia/intro
%
%- Ciclo 2D passando pela Fold-Regular: Ponce, Tere/Carles, Douglas/Iris, Buzzi/Carvalho
%
%- Kristian/ Hogan : conexao em R3

%In the nonsmooth context, the global connections between $\s$-singularities.... Notice that these aspects have no counterparts in smooth systems.

This paper is organized as follows. Section \ref{premsec} is devoted to present some basic concepts on Filippov systems. In Section \ref{globalconsec} we discuss some scenarios where a $3D$ Filippov system admits a global connection involving a fold-regular singularity. In Section \ref{main_sec} we present the Filippov systems approached in this work and we state our main results. Section \ref{struct_sec} is devoted to present the necessary tools to prove our results. Finally, in Section \ref{proofsqg} we prove the main results stated in Section \ref{main_sec}.% In Section \ref{thmBsec} we illustrate the results given in Section \ref{main_sec} by means of some examples. Finally, in Section \ref{proofsqg} we prove the main results stated in Section \ref{main_sec}.

\section{Preliminaries} \label{premsec}

Let $M$ be an open bounded  connected set of $\rn{3}$ and let $f:M\rightarrow \rn{}$ be a smooth function having $0$ as a regular value. Therefore, $\s=f^{-1}(0)$ is an embedded codimension one submanifold of $M$ which splits it in the sets $M^{\pm}=\{p\in M; \pm f(p)>0\}$.

A \textbf{germ of vector field} of class $\Cr$ at a compact set $\Lambda\subset M$ is an equivalence class $\widetilde{X}$ of $\Cr$ vector fields defined in a neighborhood of $\Lambda$. More specifically, two $\Cr$ vector fields $X_{1}$ and $X_{2}$ are in the same equivalence class if:
\begin{itemize}
	\item $X_{1}$ and $X_{2}$ are defined in neighborhoods $U_{1}$ and $U_{2}$ of $\Lambda$ in $M$, respectively;
	\item there exists a neighborhood $U_3$ of $\Lambda$ in $M$ such that $U_{3}\subset U_{1}\cap U_{2}$;
	\item $X_{1}|_{U_{3}}=X_{2}|_{U_{3}}$.
\end{itemize}
In this case, if $X$ is an element of the equivalence class $\widetilde{X}$, then $X$ is said to be a representative of $\widetilde{X}$.  
The set of germs of vector fields of class $\Cr$ at $\Lambda$ will be denoted by $\Xr(\Lambda)$, or simply $\Xr$. We endow $\Xr$ with the $\Cr$ topology. For the sake of simplicity, a germ of vector field $\widetilde{X}$ will be referred simply by its representative $X$.

Analogously, a \textbf{germ of piecewise smooth vector field} of class $\Cr$ at a compact set $\Lambda\subset M$ is an equivalence class  $\widetilde{Z}=(\widetilde{X},\widetilde{Y})$ of pairwise $\Cr$ vector fields defined as follows: $Z_{1}=(X_{1},Y_{1})$ and $Z_{2}=(X_{2},Y_{2})$ are in the same equivalence class if, and only if,
\begin{itemize}
	\item $X_{i}$ and $Y_{i}$ are defined in neighborhoods $U_{i}$ and $V_{i}$ of $\Lambda$ in $M$, respectively, $i=1,2$;
	\item there exist neighborhoods $U_{3}$ and $V_{3}$ of $\Lambda$ in $M$ such that $U_{3}\subset U_{1}\cap U_{2}$ and $V_{3}\subset V_{1}\cap V_{2}$;
	\item $X_{1}|_{U_{3}\cap \overline{M^{+}}}=X_{2}|_{U_{3}\cap\overline{M^{+}}}$ and $Y_{1}|_{V_{3}\cap \overline{M^{-}}}=Y_{2}|_{V_{3}\cap\overline{M^{-}}}$.
\end{itemize}
In this case,  if $Z=(X,Y)$ is an element of the equivalence class $\widetilde{Z}$, then $Z$ is said to be a representative of $\widetilde{Z}$.
The set of germs of piecewise smooth vector fields of class $\Cr$ at $\Lambda$ will be denoted by $\Or(\Lambda)$, or simply $\Or$. Also, $\Or=\Xr\times\Xr$ is endowed with the product topology.

If $Z=(X,Y)\in\Or$ then a \textbf{piecewise smooth vector field} is defined in some neighborhood $V$ of $\Lambda$ in $M$ as
\begin{equation}\label{filippov}
Z(p)=F(p)+\sgn(f(p)) G(p),
\end{equation}
where $F(p)=\frac{X(p)+Y(p)}{2}$ and $G(p)=\frac{X(p)-Y(p)}{2}$.

The \textbf{Lie derivative} $Xf(p)$ of $h$ in the direction of the vector field $X\in \Xr$ at $p\in\Sigma$ is defined as $Xf(p)=\langle X(p), \nabla f(p)\rangle$. Accordingly, the \textbf{tangency set} between $X$ and $\s$ is given by $S_{X}=\{p\in\Sigma;\ Xf(p)=0\}$.

\begin{remark}
	Notice that the Lie derivative is well-defined for a germ $\widetilde{X}\in\Xr$ since all the elements in this class coincide in $\s$.
\end{remark}

For $X_{1},\cdots, X_{k}\in \Xr$, the higher order Lie derivatives of $h$ are defined recurrently as
$$X_{k}\cdots X_{1}f(p)=X_{k}(X_{k-1}\cdots X_{1}f)(p),$$
i.e. $X_{k}\cdots X_{1}f(p)$ is the Lie derivative of the smooth function $X_{k-1}\cdots X_{1}f$ in the direction of the vector field $X_{k}$ at $p$. In particular, $X^{k}f(p)$ denotes $X_{k}\cdots X_{1}f(p)$, where $X_{i}=X$, for $i=1,\cdots,k$.

For a piecewise smooth vector field $Z=(X,Y)$ the switching manifold $\Sigma$ is generically the closure of the union of the following three distinct open regions.
\begin{itemize}
	\item Crossing Region: $\Sigma^{c}=\{p\in \Sigma;\ Xf(p)Yf(p)>0\}.$
	\item Stable Sliding Region: $\Sigma^{ss}=\{p\in \Sigma;\ Xf(p)<0,\ Yf(p)>0\}.$
	\item Unstable Sliding Region: $\Sigma^{us}=\{p\in \Sigma;\ Xf(p)>0,\ Yf(p)<0\}.$			
\end{itemize}

The tangency set of $Z$ will be referred as $S_{Z}=S_{X}\cup S_{Y}$. Notice that $\s$ is the disjoint union $\s^{c}\cup \s^{ss}\cup \s^{us}\cup S_{Z}$. Herein, $\s^{s}=\s^{ss}\cup \s^{us}$ is called \textbf{sliding region} of $Z$. See Figure \ref{dicfig}.

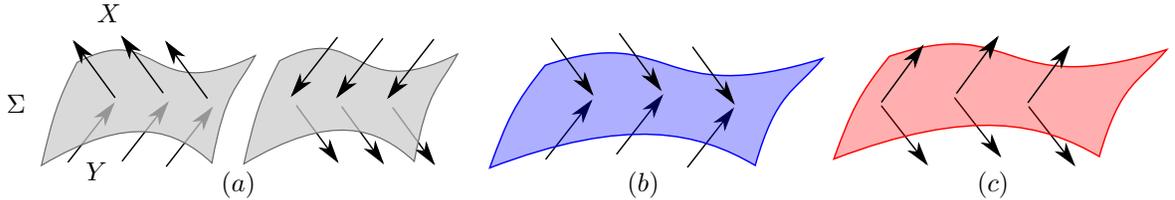
\begin{figure}[h!]
	\centering
	\bigskip
	\begin{overpic}[width=15cm]{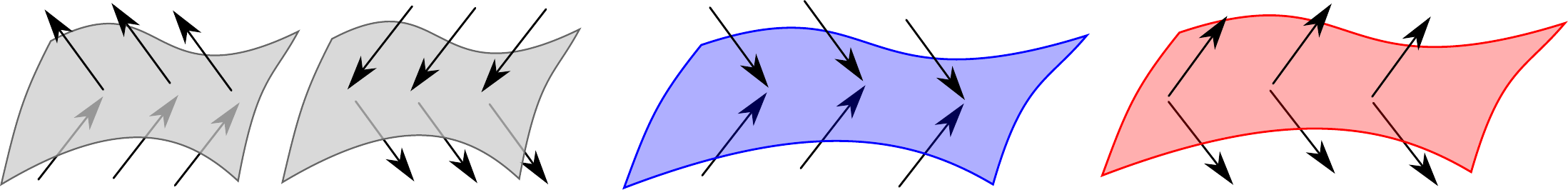}
		%					\begin{overpic}[grid,tics=5,width=11cm]{Figures/discregions.pdf}			
		\put(16,-2){{\footnotesize $(a)$}}		
		\put(52,-2){{\footnotesize $(b)$}}		
		\put(83,-2){{\footnotesize $(c)$}}	
		\put(-3,5){{\footnotesize $\s$}}	
		\put(5,13){{\footnotesize $X$}}	
		\put(4,-1){{\footnotesize $Y$}}							
		
	\end{overpic}
	\bigskip
	\caption{Regions in $\s$: $\s^{c}$ in $(a)$, $\s^{ss}$ in $(b)$ and $\s^{us}$ in $(c)$.   }	\label{dicfig}
\end{figure}

The concept of solution of $Z$ follows the Filippov's convention (see, for instance, \cite{F,GTS,T8}). The local solution of $Z=(X,Y)\in \Or$ at $p\in \Sigma^{s}$ is given by the \textbf{sliding vector field}
\begin{equation}
F_{Z}(p)=\frac{1}{Yf(p)-Xf(p)}\left(Yf(p)X(p)-Xf(p)Y(p)\right).
\end{equation}
Notice that $F_{Z}$ is a $\Cr$ vector field tangent to $\s^{s}$. %The critical points of $F_{Z}$ in $\s^{s}$ are called \textbf{pseudo-equilibria} of $Z$.

%\begin{definition}\label{NSVF}
%	We defined the \textbf{normalized sliding vector field} $F_{Z}^{N}$ of $Z$ by
%	\begin{equation}
%	F_{Z}^{N}(p)=Y f(p)X (p)-X f(p)Y(p),
%	\end{equation}	
%	for every $p\in\s^{s}$. 
%\end{definition}
%
%Notice that $F_Z^N$ is also a $\Cr$ vector field  tangent to $\s^{s}$.
%
%\begin{remark}
%	The normalized sliding vector field can be $\Cr$ extended beyond the boundary of $\s^s$. In addition, if $R$ is a connected component of $\s^{ss}$, then $F_{Z}^{N}$ is a re-parameterization of $F_{Z}$ in $R$, and so the phase portraits of both coincide. If $R$ is a connected component of $\s^{us}$, then $F_{Z}^{N}$ is a (negative) re-parameterization of $F_{Z}$ in $R$, then they have the same phase portrait, but the orbits are oriented in opposite direction.
%\end{remark}
%\begin{remark}
%	If $R$ is a connected component of $\s^{ss}$, then $F_{Z}^{N}$ is a re-parameterization of $F_{Z}$ in $R$, and they have exactly the same phase portrait. If $R$ is a connected component of $\s^{us}$, then $F_{Z}^{N}$ is a (negative) re-parameterization of $F_{Z}$ in $R$, then they have the same phase portrait, but the orbits are oriented in opposite direction.
%\end{remark}

If $p\in\s^c$, then the orbit of $Z=(X,Y)\in\Or$ at $p$ is defined as the concatenation of the orbits of $X$ and $Y$ at $p$. Nevertheless, if $p\in\s\setminus\s^c$, then it may occur a lack of uniqueness of solutions. In this case,  the flow of $Z$ is multivalued and any possible trajectory passing through $p$ originated by the orbits of $X$, $Y$ and $F_Z$ is considered as a solution of $Z$. More details can be found in \cite{F,GTS}.

In the following definition, we introduce the so-called $\s$-singularities of a Filippov system.

\begin{definition}\label{sigmasing}
	Let $Z=(X,Y)\in\Or$, a point $p\in \s$ is said to be:
	\begin{enumerate}[i)]
		\item a \textbf{tangential singularity} of $Z$ provided that $Xf(p)Yf(p)=0$ and $X(p), Y(p)\neq 0$;
		\item a \textbf{$\s$-singularity} of $Z$ provided that $p$ is either a tangential singularity, an equilibrium of $X$ or $Y$, or a pseudo-equilibrium of $Z$. 
	\end{enumerate}
\end{definition}

\begin{remark}
	A point $p\in\s$ which is not a $\s$-singularity of $Z$ is also referred as a \textbf{regular-regular} point of $Z$.
\end{remark}

We say that $\gamma$ is a \textbf{regular orbit} of $Z=(X,Y)$ if it is a piecewise smooth curve such that $\gamma\cap M^{+}$ and $\gamma\cap M^{-}$ are unions of regular orbits of $X$ and $Y$, respectively, and  $\gamma\cap\s\subset\s^{c}$.

\begin{definition}\label{elementarydef}
	Let $Z=(X,Y)\in\Or$. A tangential singularity $p\in \s$ is said to be a \textbf{fold-regular singularity} if either  one of the following conditions hold.
	\begin{enumerate}[i)]
		\item $Xf(p)=0$, $X^2f(p)\neq 0$ and $Yf(p)\neq 0$. In this case, if $X^2f(p)>0$, then we say that $p$ is \textbf{visible}, otherwise it is said to be \textbf{invisible}.
	\item $Yf(p)=0$, $Y^2f(p)\neq 0$ and $Xf(p)\neq 0$.	In this case, if $Y^2f(p)<0$, then we say that $p$ is \textbf{visible}, otherwise it is said to be \textbf{invisible}.
	\end{enumerate}
\end{definition}

\begin{figure}[h!]
	\centering
	\bigskip
	\begin{overpic}[width=10cm]{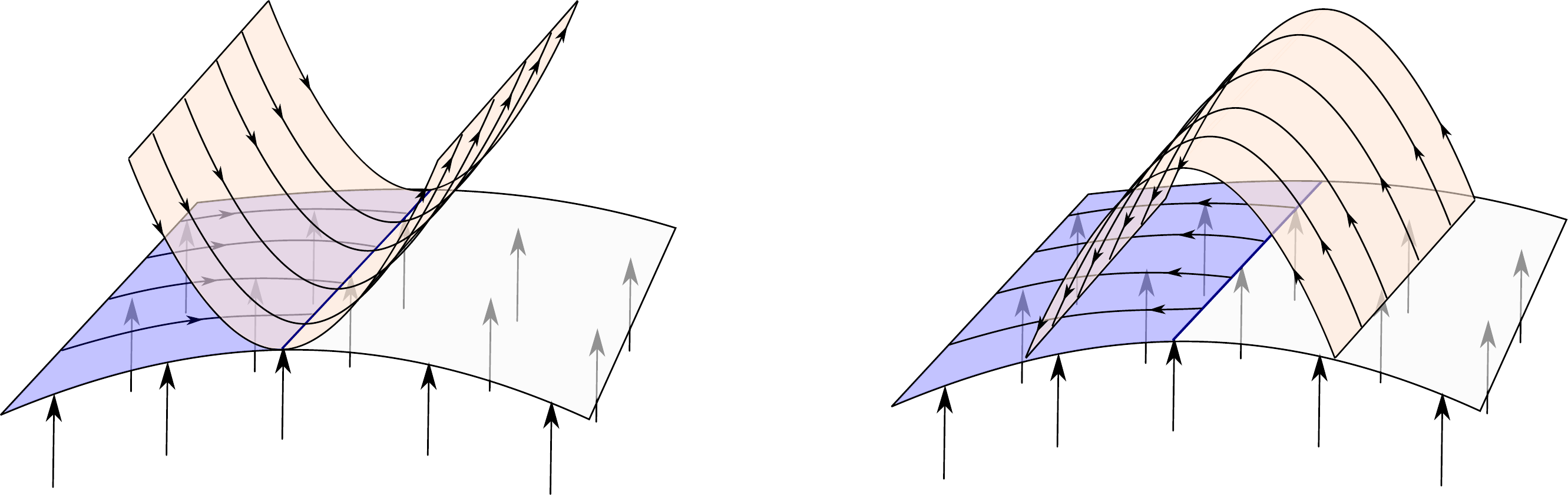}
		%			\begin{overpic}[grid,tics=5,width=10cm]{types.pdf}			
		\put(15,-3){{\footnotesize (a)}}						
		\put(75,-3){{\footnotesize (b)}}
		\put(-2,7){{\footnotesize $\s$}}		
		\put(0,17){{\footnotesize $X$}}	
		\put(0,0){{\footnotesize $Y$}}																					
	\end{overpic}
	\bigskip
	\caption{A fold-regular singularity $p$ with $Xf(p)=0$ and $(a)$ $X^2f(p)>0$ or  $(b)$ $X^2f(p)<0$. }	\label{types_fig}
\end{figure}
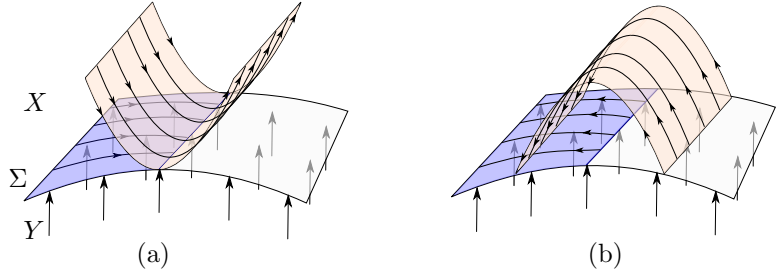 

\begin{remark}
	If condition $(i)$ (resp. $(ii)$) is satisfied in Definition \ref{elementarydef}, then we say that $X$ (resp. $Y$) has a fold point at $p$.
\end{remark}
The following proposition is proved in \cite{V}.

\begin{prop}[Vishik's Normal Form] \label{Vishik}
	Let $X\in\Xr$. If $p\in\s$ satisfies $Xf(p)=0$ and $X^2f(p)\neq 0$, then there exist a neighborhood $U_0$ of $p$ in $M$ and a system of coordinates $(x,y,z)$ at $p$ defined in $U_0$ ($x(p)=y(p)=z(p)=0$) such that $X$ is given by
	$$X(x,y,z)=(0,1,y), $$
and $\s$ is given by the equation $z=0$ in $U_0$.
\end{prop}

\section{A discussion on some global connections}\label{globalconsec}

Let $Z_0=(X_0,Y_0)\in\Or$ be a Filippov system having a visible fold-regular singularity at $p_0\in\s$ (see Definition \ref{elementarydef}). Denote the flows of $X_0$ and $Y_0$ by $\p_{X_0}(t;p)$ and $\p_{Y_0}(t;p)$, respectively. Assume that $Z_0$ satisfies the following set of global hypotheses \textbf{(G)}:

\begin{enumerate}
	\item[($G_1$)] There exists $T_+>0$ such that $p_0^+=\p_{X_0}(T_+; p_0)\in \s$.
	\item[($G_2$)] $\Gamma_0^+=\{\p_{X_0}(t; p_0);\ t\in(0,T_+)\}\subset M^+$ and $X_0$ is transverse  to $\s$ at $p_0^+$.
	\item[($G_3$)] There exist a point $q_0\in \s$ and a regular orbit $\Gamma_0^R$ of $Z_0$ connecting $p_0^+$ and $q_0$.
\end{enumerate} 

Without loss of generality, assume that $\Gamma_0^R$ in condition $(G_3)$ is a regular orbit of $Y_0$ contained in $M^-$. Using properties of a fold-regular singularity (see \cite{V}) and the transversality condition $(G_2)$, we define the germs $\mathcal{P}_0^+:(\s,p_0)\rightarrow (\s, p_0^+)$ and $\mathcal{P}_0^-:(\s,p_0^+)\rightarrow (\s, q_0)$ induced by the flows of $X_0$ and $Y_0$, respectively (see Figure \ref{first_fig}). Thus, consider
\begin{equation}
\label{first}
\mathcal{P}_0=\mathcal{P}_0^-\circ \mathcal{P}_0^+,
\end{equation}
and notice that the restriction of $\mathcal{P}_0$ to $\overline{\s^c}$ is a first return map of $Z_0$ in $\s$.

\begin{figure}[h!]
	\centering
	\bigskip
	\begin{overpic}[width=10cm]{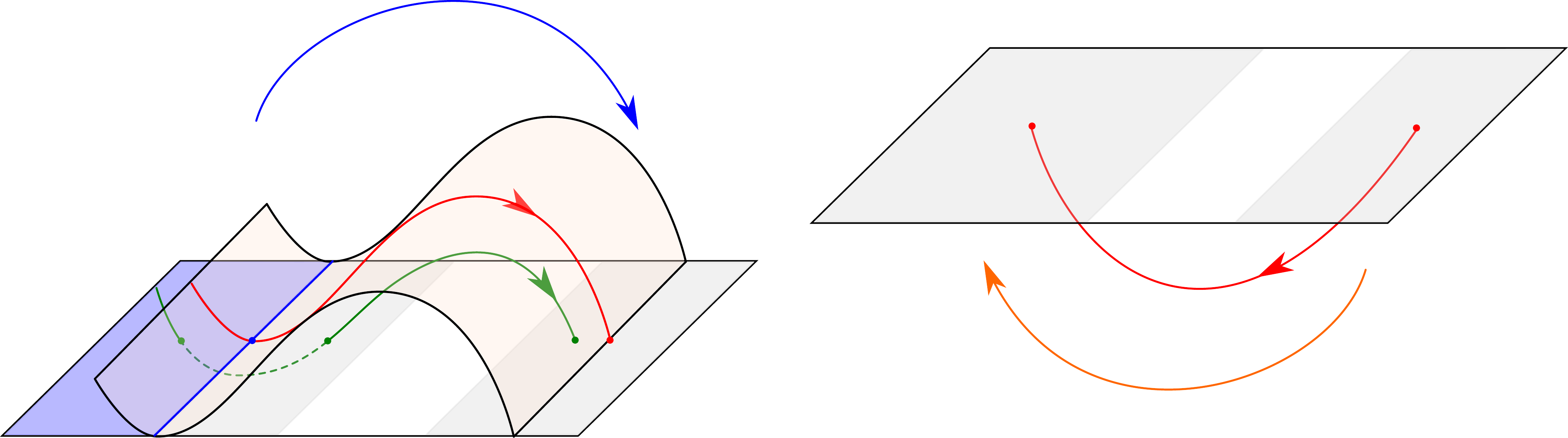}
		%								\begin{overpic}[grid,tics=5,width=10cm]{Figures/first.pdf}			
		\put(27,23){{\footnotesize $\mathcal{P}_0^+$}}		
		\put(5,15){{\footnotesize $X_0$}}		
		\put(15,8){{\scriptsize $p_0$}}	
		\put(10,4){{\scriptsize $p_1$}}	
		\put(21,4){{\scriptsize $p_2$}}									
		\put(39,5){{\scriptsize $p_0^+$}}	
		\put(90,21){{\scriptsize $p_0^+$}}	
		\put(65,21){{\scriptsize $q_0$}}													
		\put(0,5){{\footnotesize $\s$}}	
		\put(98,20){{\footnotesize $\s$}}	
		\put(75,5){{\footnotesize $\mathcal{P}_0^-$}}	
		\put(95,15){{\footnotesize $Y_0$}}						  	  		
		%		%\put(57,10){{\footnotesize $\s$}}	
		%		\put(26,22){{\footnotesize $\Gamma_0$}}			
		%		\put(10,20){{\footnotesize $X_0$}}	
		%		\put(10,2){{\footnotesize $Y_0$}}							
		%		
	\end{overpic}
	\bigskip
	\caption{Action of the maps $\mathcal{P}_0^+$ and $\mathcal{P}_0^-$.}	\label{first_fig}
\end{figure}
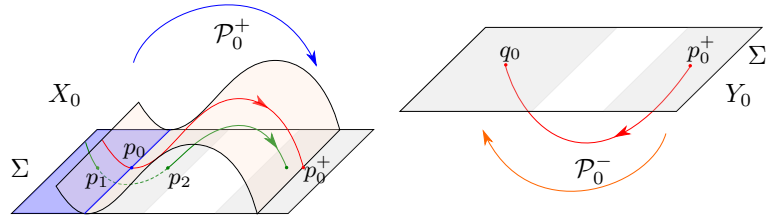 

\begin{remark}
	Notice that, in Figure \ref{first_fig}, the points $p_1\neq p_2$ have the same image through $\mathcal{P}_0^+$. We will see that $\mathcal{P}_0$ is a non-invertible $\Cr$ map and its restriction to $\overline{\s^c}$ is a $\Cr$ homeomorphism.
\end{remark}

Since $p_0$ is a visible fold-regular singularity, it follows that $Z_0$ has a (compact) $\Cr$ curve $\gamma_0\subset\s$ of visible fold-regular singularities containing $p_0$ (see \cite{V}). It follows that, $\gamma_0$ is brought to a (compact) $\Cr$ curve $\zeta_0\subset \s$ by $\mathcal{P}_0$ such that $q_0\in\zeta_0$. 

Also, still from Local Theory, one deduces directly that the sliding vector field $F_{Z_0}$ of $Z_0$ is transverse to the curve $\gamma_0$ anywhere, and there exists a neighborhood $V_0$ of $\gamma_0$ in $\s$ (with compact closure) such that:
\begin{enumerate}[i)]
	\item $Y_0$ is transverse to $\s$ at any point of $V_0$;
	\item $\gamma_0$ divides $\overline{V_0}$ into two connected components, one contained in $\s^s$ and the other one contained in $\s^c$;
	\item $F_{Z_0}$ is $\Cr$-extended onto $\overline{V_0}$ (see Lemma 24 in \cite{GT19}).%itself to $\overline{V_0}$ (see Lemma 24 in \cite{GT17}).
\end{enumerate}
See Figure \ref{v0fig}.

\begin{figure}[h!]
	\centering
	\bigskip
	\begin{overpic}[width=4cm]{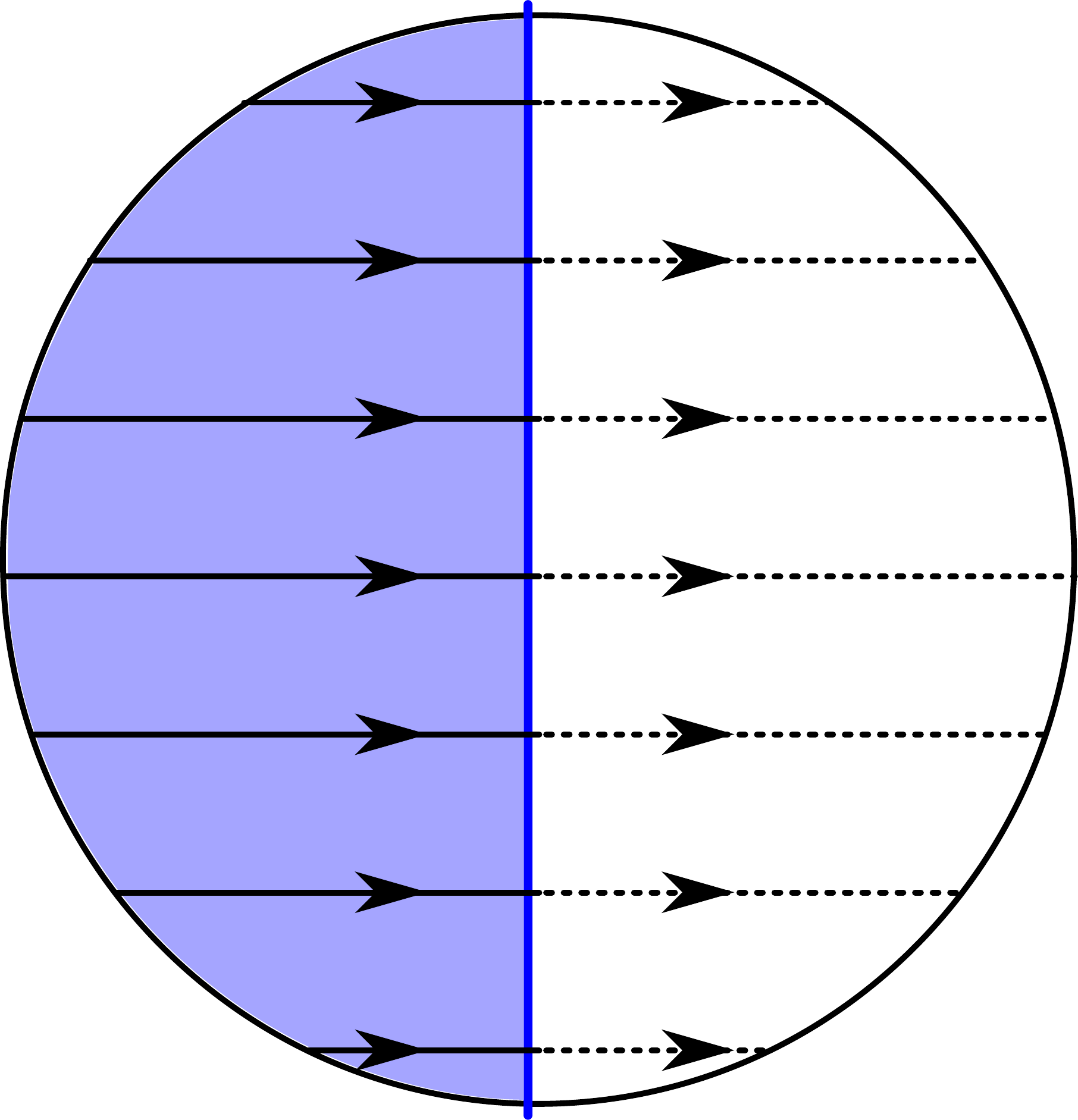}
		%								\begin{overpic}[grid,tics=5,width=13cm]{Figures/flavours2.pdf}			
		\put(50,70){{\footnotesize $\gamma_0$}}
		\put(100,45){{\footnotesize $V_0$}}								
		%		\put(75,-3){{\footnotesize $(b)$}}		
		%		\put(16.2,10.5){{\scriptsize $p_0$}}	
		%		\put(38,10.5){{\scriptsize $q_0$}}			
		%		\put(0,10){{\footnotesize $\s$}}	
		%		%\put(57,10){{\footnotesize $\s$}}	
		%		\put(26,22){{\footnotesize $\Gamma_0$}}			
		%		\put(10,20){{\footnotesize $X_0$}}	
		%		\put(10,2){{\footnotesize $Y_0$}}							
		%		
	\end{overpic}
	\bigskip
	\caption{Neighborhood $V_0\subset\s$.}	\label{v0fig}
\end{figure} 

For our purposes, we assume that $\zeta_0\subset V_0$. Accordingly, we consider $\mathcal{P}_0: V_0\rightarrow V_0$. In this case, we distinguish the following situations: (a) $\zeta_0\subset \s^s$, (b) $\zeta_0\subset \s^c$, (c) $\zeta_0$ is transverse to $\gamma_0$ at $q_0$, and (d) $\zeta_0$ is tangent to $\gamma_0$ at $q_0$ (see Figure \ref{flavours_fig}).

\begin{figure}[h!]
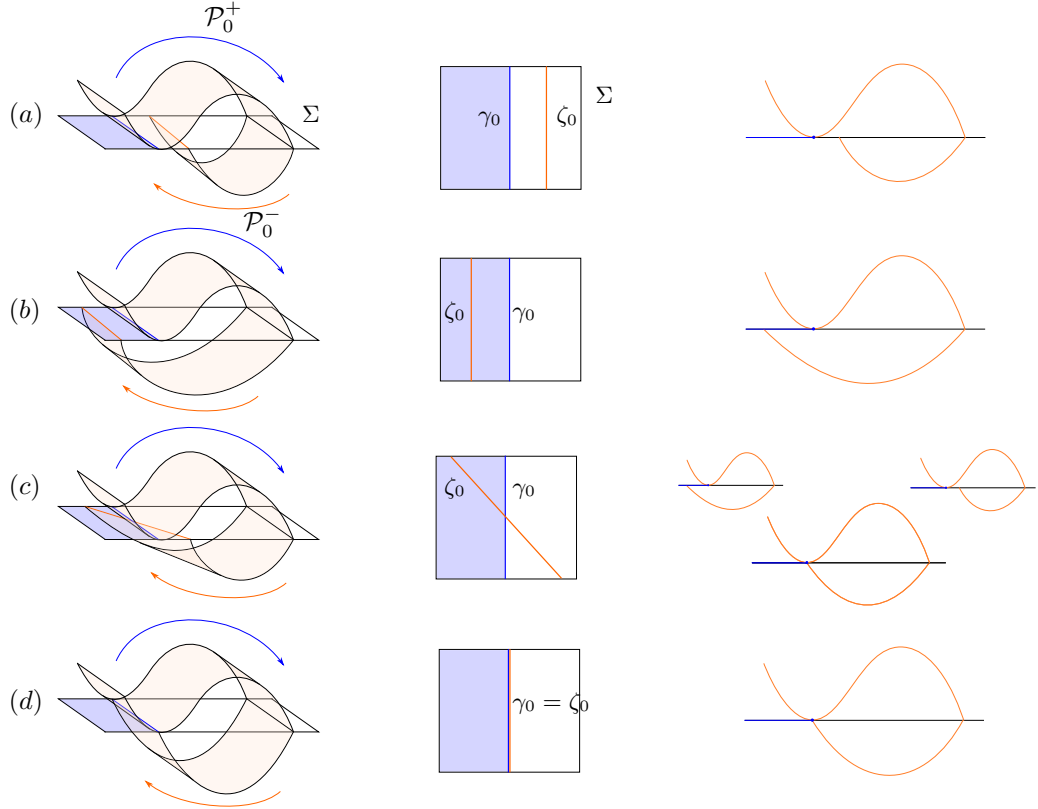

	\centering
	\bigskip
	\begin{overpic}[width=13cm]{Figures/flavours2.pdf}
		%								\begin{overpic}[grid,tics=5,width=13cm]{Figures/flavours2.pdf}			
		\put(15,80){{\footnotesize $\mathcal{P}_0^+$}}
		\put(19,59){{\footnotesize $\mathcal{P}_0^-$}}
		\put(43,70){{\scriptsize $\gamma_0$}}
		\put(51,70){{\scriptsize $\zeta_0$}}	
		\put(46.5,50){{\scriptsize $\gamma_0$}}
		\put(39.5,50){{\scriptsize $\zeta_0$}}	
		\put(46.5,32){{\scriptsize $\gamma_0$}}
		\put(25,70){{\scriptsize $\s$}}
		\put(55,72){{\scriptsize $\s$}}
		\put(39.5,32){{\scriptsize $\zeta_0$}}	
		\put(46.5,10){{\scriptsize $\gamma_0=\zeta_0$}}		
		\put(-5,10){{\footnotesize $(d)$}}	
		\put(-5,32){{\footnotesize $(c)$}}	
		\put(-5,50){{\footnotesize $(b)$}}	
		\put(-5,70){{\footnotesize $(a)$}}								
		%		\put(75,-3){{\footnotesize $(b)$}}		
		%		\put(16.2,10.5){{\scriptsize $p_0$}}	
		%		\put(38,10.5){{\scriptsize $q_0$}}			
		%		\put(0,10){{\footnotesize $\s$}}	
		%		%\put(57,10){{\footnotesize $\s$}}	
		%		\put(26,22){{\footnotesize $\Gamma_0$}}			
		%		\put(10,20){{\footnotesize $X_0$}}	
		%		\put(10,2){{\footnotesize $Y_0$}}							
		%		
	\end{overpic}
	\bigskip
	\caption{Relative position between the curve $\gamma_0$ of fold-regular singularities and its image $\zeta_0$ through the flow of $Z_0$: (a) $\zeta_0\subset \s^s$, (b) $\zeta_0\subset \s^c$, (c)$\gamma_0\pitchfork \zeta_0$ and (d) $\zeta_0=\gamma_0$.}	\label{flavours_fig}
\end{figure} 

Notice that configurations (a), (b) and (c) are robust in $\Or$. However, configuration (d) is easily broken by small perturbations. In fact, the degree of degeneracy in case (d) depends on the degree of the contact between $\gamma_0$ and $\zeta_0$ at $q_0$. The most degenerate situation occurs when $\gamma_0=\zeta_0$ (as illustrated in Figure \ref{flavours_fig}). 

Now, we discuss the possible dynamics concerning the robust situations (a), (b) and (c).

\subsection{Cases (a) $\zeta_0\subset \s^s$ and (b) $\zeta_0\subset \s^c$}\label{sliding_sec}
\

If $\zeta_0\subset \s^c$, then the dynamics of $Z_0$ is trivial around the orbit connecting $p_0$ and $q_0$. In fact, consider
\begin{enumerate}[i)]
	\item a section $\Pi^+$ at $p_0$ such that $\Pi^+$ is the restriction to $M^+$ of a local transversal section of $X_0$ at $p_0$ which intersects $\s$ at $\gamma_0$;
	\item a section $\Pi^-$ which consists on a neighborhood of $p_0$ intersected with $\overline{\s^c}$;
	\item $\Pi=\Pi^+\cup \Pi^-$.
\end{enumerate}

Thus, using the local structure of a fold-regular singularity (see \cite{V}), we obtain that all orbits of $Z_0$ in a neighborhood of $p_0$ intersect $\Pi$. Also, for a neighborhood $N_0$ of $q_0$ contained in $\s^c$, we construct a tubular flow box between $\Pi$ and $N_0$ along the orbits of $X_0$ and $Y_0$ (see Figure \ref{tubular_fig}).

\begin{figure}[h!]
	\centering
	\bigskip
	\begin{overpic}[width=9cm]{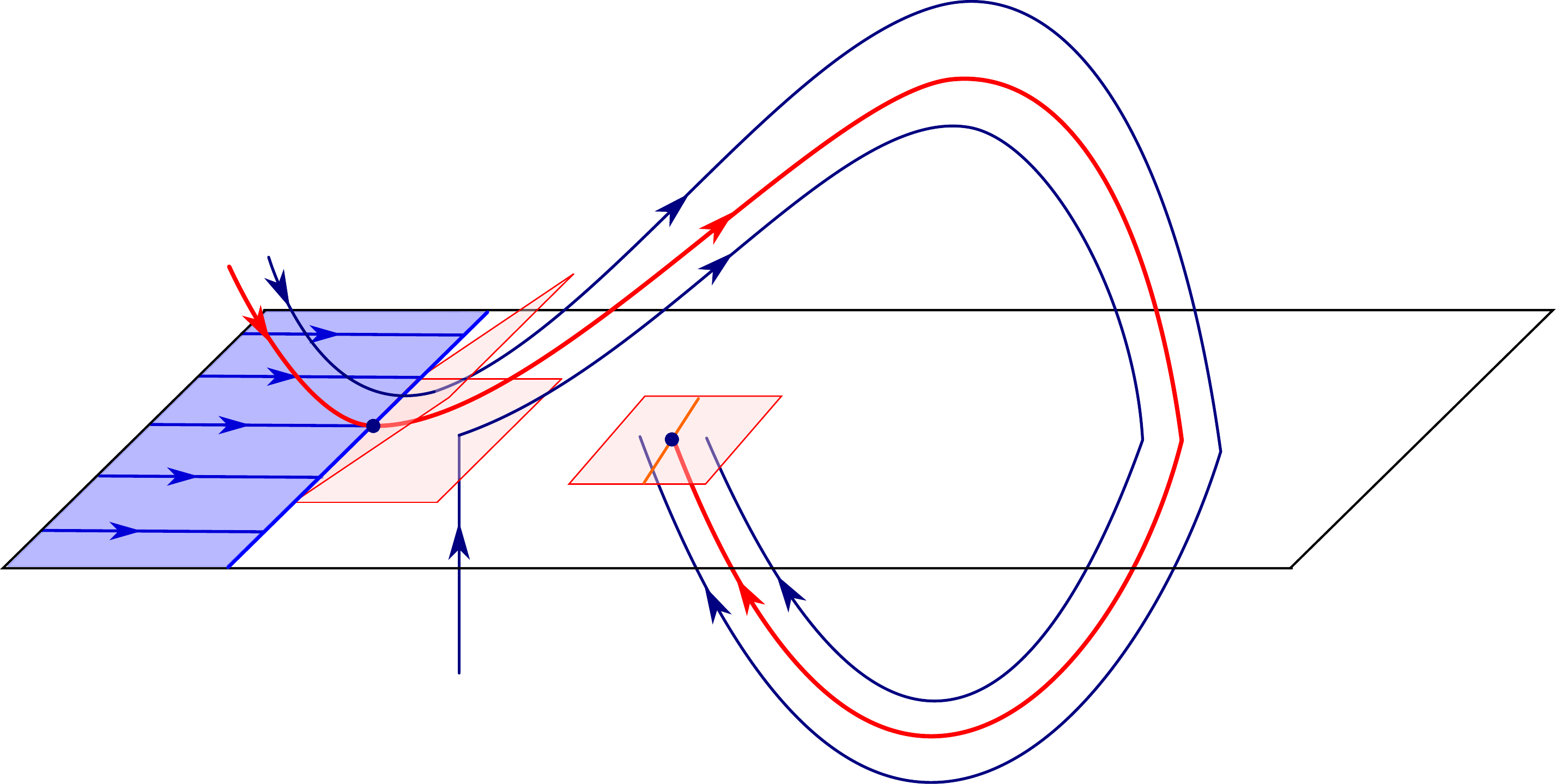}
		%								\begin{overpic}[grid,tics=5,width=9cm]{Figures/tubular.pdf}			
		\put(32,32){{\footnotesize $\Pi$}}	
		\put(50,20){{\footnotesize $N_0$}}			
		%		\put(75,-3){{\footnotesize $(b)$}}		
		%		\put(16.2,10.5){{\scriptsize $p_0$}}	
		%		\put(38,10.5){{\scriptsize $q_0$}}			
		%		\put(0,10){{\footnotesize $\s$}}	
		%		%\put(57,10){{\footnotesize $\s$}}	
		%		\put(26,22){{\footnotesize $\Gamma_0$}}			
		%		\put(10,20){{\footnotesize $X_0$}}	
		%		\put(10,2){{\footnotesize $Y_0$}}							
		%		
	\end{overpic}
	\bigskip
	\caption{Tubular flow of $Z_0$ between $\Pi$ and $N_0$.}	\label{tubular_fig}
\end{figure}

Now, consider that $\zeta_0\subset\s^s$. As we have seen, each point $p\in\gamma_0$ is brought to a point $\mathcal{P}_0(p)\in\zeta_0$ through the flow of $Z_0$ (orbits of $X_0$ and $Y_0$). Since $F_{Z_0}$ is regular in $V_0$ and transverse to $\gamma_0$, each point $q\in\zeta_0$ reaches $\gamma_0$ at a unique point $\psi_0^*(q)$ through a sliding trajectory of $Z_0$. It defines the $\Cr$ map 
\begin{equation}\label{slmap}
\psi_0: p\in\gamma_0\longmapsto \psi_0^*(\mathcal{P}_0(p))\in\gamma_0,
\end{equation}
which induces a dynamics in the fold curve $\gamma_0$. We refer to $\psi_0$ as the \textbf{fold line map} associated to $Z_0$.

In this case, the orbits of $X_0$, $Y_0$ and $F_{Z_0}$ connect $\gamma_0$ to itself and they give rise to a $Z_0$-invariant manifold $\mathcal{M}$ which is a piecewise-smooth $2D$-cylinder or a piecewise-smooth Möbius strip, depending on the identification provided by $\psi_0$. Also, the dynamics of $Z_0$ in $\mathcal{M}$ is completely characterized by the dynamics of $\psi_0$ in $\gamma_0$. Thus, 

\begin{enumerate}[i)]
	\item[($Sl_{R}$)] if $p_0\in\gamma_0$ is a regular point of $\psi_0$, then the dynamics of $Z_0$ in $\mathcal{M}$ is trivial. It means that there are no minimal sets contained in $\mathcal{M}$;
	\item[($Sl_{S}$)] if $p_0\in\gamma_0$ is a fixed point of $\psi_0$, then $Z_0$ has a sliding connection $\Gamma_0$ through $p_0$ contained in $\mathcal{M}$ (see Figure \ref{slidingcycle_fig}).
\end{enumerate}

\begin{figure}[h!]
	\centering
	\bigskip
	\begin{overpic}[width=13cm]{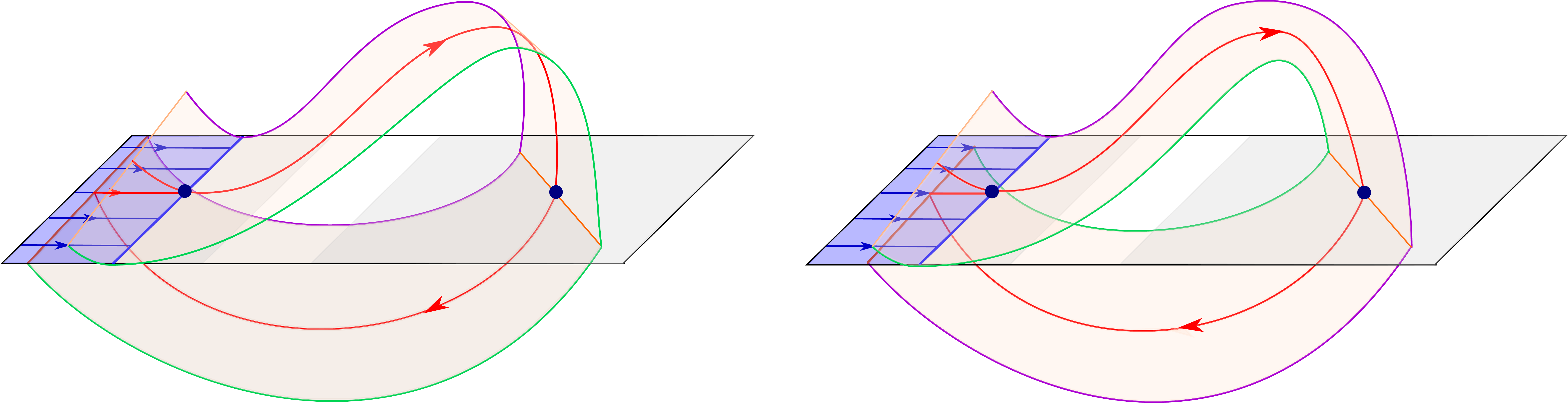}
		%							\begin{overpic}[grid,tics=5,width=13cm]{Figures/cycles.pdf}	
		\put(6,6){{\footnotesize $A$}}		
		\put(15,18){{\footnotesize $B$}}
		\put(1,6){{\footnotesize $A'$}}		
		\put(8.7,18){{\footnotesize $B'$}}			
		\put(58,6){{\footnotesize $A$}}		
		\put(67,17){{\footnotesize $B$}}
		\put(52.5,6){{\footnotesize $B'$}}		
		\put(58,16){{\footnotesize $A'$}}			
		\put(25,-3){{\footnotesize $(a)$}}		
		\put(75,-3){{\footnotesize $(b)$}}	
		\put(30,27){{\footnotesize $\mathcal{M}$}}
		\put(35,23){{\footnotesize $\Gamma_0$}}
		\put(80,27){{\footnotesize $\mathcal{M}$}}	
		\put(83,23){{\footnotesize $\Gamma_0$}}					
		%		\put(16.2,10.5){{\scriptsize $p_0$}}	
		%		\put(38,10.5){{\scriptsize $q_0$}}			
		%		\put(0,10){{\footnotesize $\s$}}	
		%		%\put(57,10){{\footnotesize $\s$}}	
		%		\put(26,22){{\footnotesize $\Gamma_0$}}			
		%		\put(10,20){{\footnotesize $X_0$}}	
		%		\put(10,2){{\footnotesize $Y_0$}}							
		%		
	\end{overpic}
	\bigskip
	\caption{A sliding connection $\Gamma_0$ of $Z_0$ in $\mathcal{M}$ where $\mathcal{M}$ is a piecewise-smooth cylinder $(a)$ or a piecewise-smooth Möbius strip $(b)$. }	\label{slidingcycle_fig}
\end{figure}
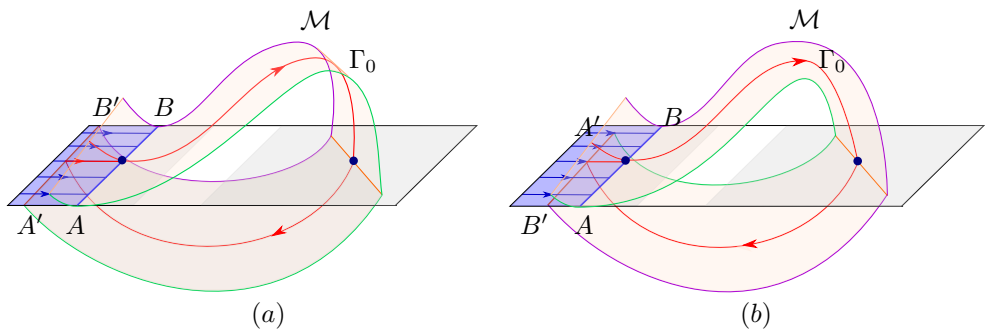 

If ($Sl_{S}$) is satisfied, then the sliding connection $\Gamma_0$ of $Z_0$ can be persistent, depending on the properties of $\psi_0$ at $p_0$. In fact, we mention the following cases. 

\begin{enumerate}[i)]
	\item If $p_0$ is a hyperbolic fixed point of $\psi_0$, then each $Z\in\Or$ nearby $Z_0$ presents a sliding connection $\Gamma$ near $\Gamma_0$, in the Hausdorff distance, with the same stability of $\Gamma_0$.
	\item If $p_0$ is a fixed point of $\psi_0$ of saddle-node type, i.e. $|\psi_0'(p_0)|=1$ and $|\psi_0''(p_0)|\neq1$, then $Z_0$ belongs to a codimension one submanifold of $\Or$. A versal unfolding of $Z_0$ in $\Or$ around $\Gamma_0$ is illustrated in Figure \ref{unfoldingsliding_fig}.
\end{enumerate}

\begin{figure}[H]
	\centering
	\bigskip
	
	\begin{overpic}[width=13cm]{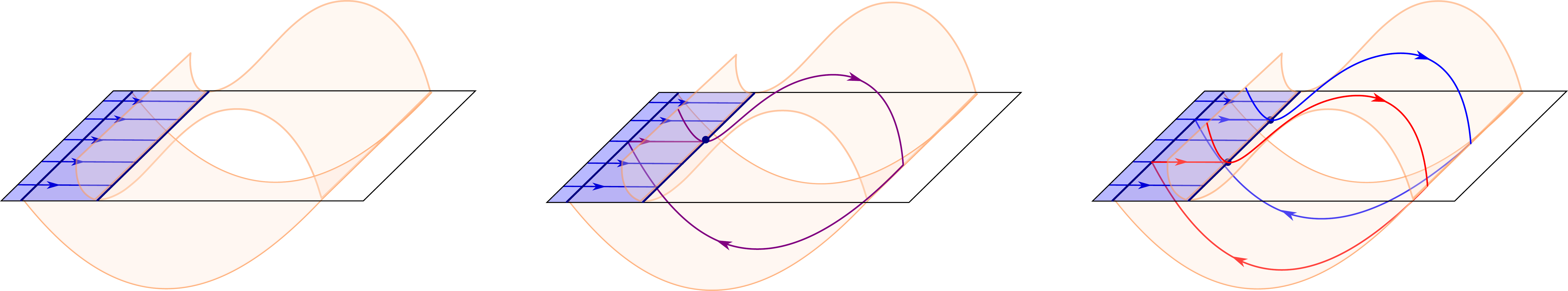}
		%							\begin{overpic}[grid,tics=5,width=13cm]{Figures/cycles.pdf}			
		\put(10,-2){{\footnotesize $\ag<0$}}		
		\put(45,-2){{\footnotesize $\ag=0$}}
		\put(57,12){{\footnotesize $\Gamma_0$}}		
		\put(80,-2){{\footnotesize $\ag>0$}}			
		%		\put(16.2,10.5){{\scriptsize $p_0$}}	
		%		\put(38,10.5){{\scriptsize $q_0$}}			
		%		\put(0,10){{\footnotesize $\s$}}	
		%		%\put(57,10){{\footnotesize $\s$}}	
		%		\put(26,22){{\footnotesize $\Gamma_0$}}			
		%		\put(10,20){{\footnotesize $X_0$}}	
		%		\put(10,2){{\footnotesize $Y_0$}}							
		%		
	\end{overpic}

	\bigskip
	
	\caption{ A versal unfolding $Z_{\ag}$ of $Z_0$ around $\Gamma_0$ when conditions $(Sl_S)$ and $(ii)$ are satisfied (saddle-node bifurcation). }	\label{unfoldingsliding_fig}
\end{figure}
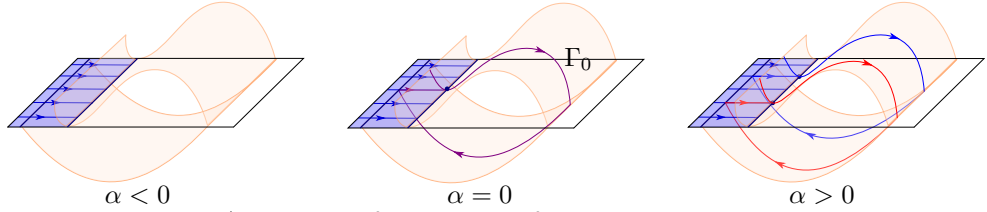 

\begin{remark} Observe that, if $p_0$ is a fixed point of $\psi_0$ having a higher degree of degeneracy, then $Z\in\Or$ nearby $Z_0$ presents complicated sliding features contained in $\widetilde{\mathcal{M}}$ (which is a $Z$-invariant manifold nearby $\mathcal{M}$ having the same topological type of $\mathcal{M}$) bifurcating from $\Gamma_0$.  % the codimension of $Z_0$ will be increased.
\end{remark}

\subsection{Case (c): $\zeta_0$ and $\gamma_0$ are transverse at $q_0$}
\

Assume that $Z_0=(X_0,Y_0)\in\Or$ satisfies \textbf{(G)} and the following assumption
\begin{center}
	\textbf{(T)} $\zeta_0\cap \gamma_0=\{q_0\}$ and $\zeta_0\pitchfork \gamma_0$ at $q_0$.
\end{center}

If $q_0\neq p_0$, then for $Z\approx Z_0$ ($\approx$ stands for nearby), hypothesis (T) implies that there exist curves in $\s$, $\gamma$ and $\zeta$, analogous to $\gamma_0$ and $\zeta_0$, satisfying $\zeta\cap\gamma=\{q\}$ for some $q\approx q_0$ and $\zeta\pitchfork \gamma$ at $q$. Also, there exists $p\approx p_0$ which is mapped to $q$ through the flow of $Z$, and $q\neq p$. It follows that the connection between $p_0$ and $q_0$ of $Z_0$ is persistent for $Z\in\Or$ nearby $Z_0$ (see Figure \ref{persist_fig}).

\begin{figure}[h!]
	\centering
	\bigskip
	\begin{overpic}[width=7cm]{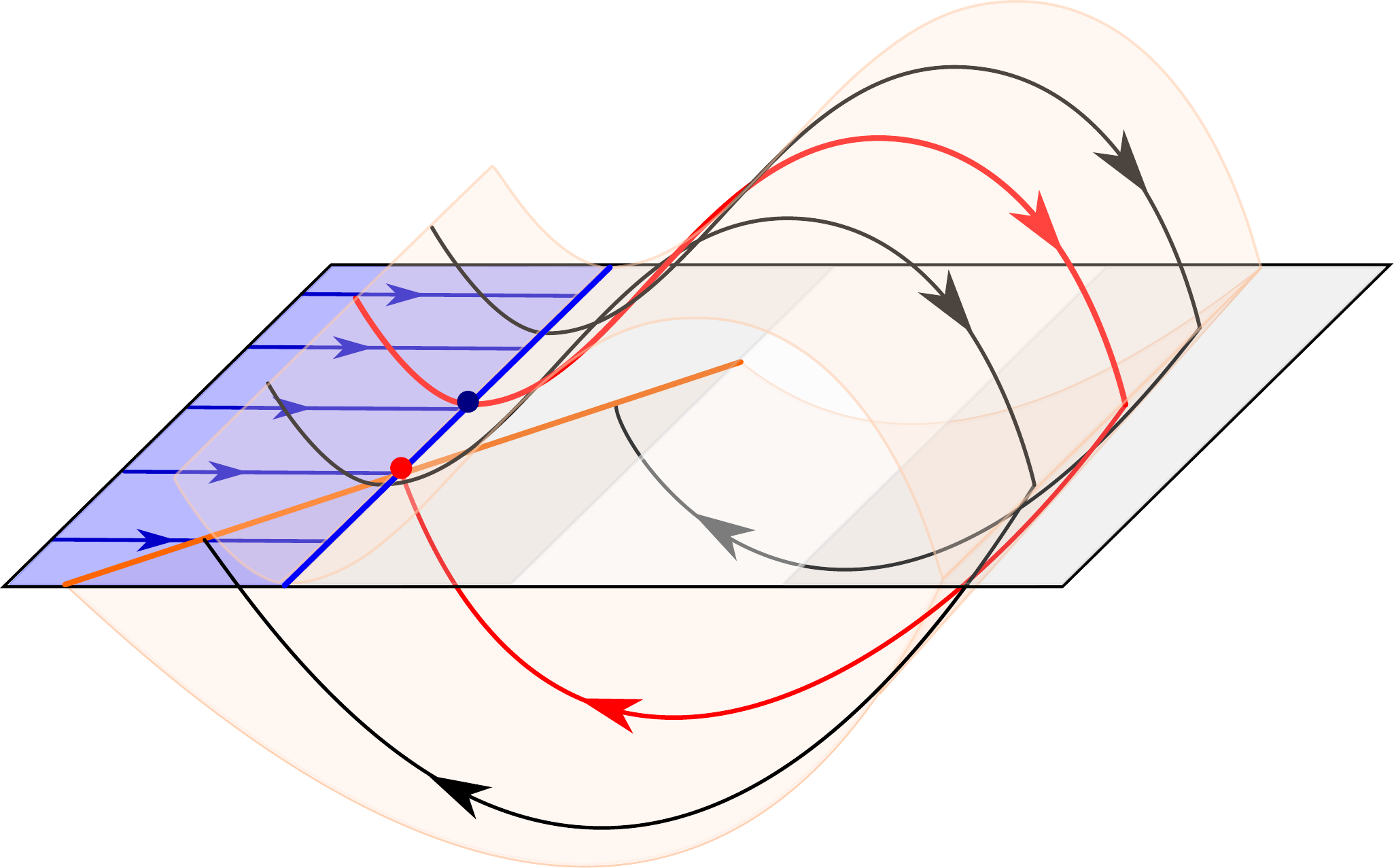}
		%									\begin{overpic}[grid,tics=5,width=7cm]{Figures/generic_con.pdf}			
		\put(70,52){{\footnotesize $\Gamma_0$}}		
		\put(30,35){{\scriptsize $p_0$}}		
		\put(25,30){{\scriptsize $q_0$}}	
		%		\put(38,10.5){{\scriptsize $q_0$}}			
		%		\put(0,10){{\footnotesize $\s$}}	
		%		%\put(57,10){{\footnotesize $\s$}}	
		%		\put(26,22){{\footnotesize $\Gamma_0$}}			
		%		\put(10,20){{\footnotesize $X_0$}}	
		%		\put(10,2){{\footnotesize $Y_0$}}							
		
	\end{overpic}
	\bigskip
	\caption{A robust connection $\Gamma_0$ of $Z_0$ between $p_0$ and $q_0$.}	\label{persist_fig}
\end{figure} 

Now, if hypothesis
\begin{center}
	\textbf{(H)} $p_0=q_0$,
\end{center}
is also satisfied, then $Z_0$ has a homoclinic-like loop $\Gamma_0$ at $p_0$ (see Figure \ref{generic_cycle_fig}). In contrast to the previous case, this phenomenon is not persistent in $\Or$. 

%\begin{figure}[h!]
%	\centering
%	\bigskip
%	\begin{overpic}[width=7cm]{Figures/generic_cycle2.pdf}
%		%							\begin{overpic}[grid,tics=5,width=13cm]{Figures/cycles.pdf}		
%		\put(70,52){{\footnotesize $\Gamma_0$}}		
%		\put(30,35){{\scriptsize $p_0$}}		
%		%		\put(25,-3){{\footnotesize $(a)$}}		
%		%		\put(75,-3){{\footnotesize $(b)$}}		
%		%		\put(16.2,10.5){{\scriptsize $p_0$}}	
%		%		\put(38,10.5){{\scriptsize $q_0$}}			
%		%		\put(0,10){{\footnotesize $\s$}}	
%		%		%\put(57,10){{\footnotesize $\s$}}	
%		%		\put(26,22){{\footnotesize $\Gamma_0$}}			
%		%		\put(10,20){{\footnotesize $X_0$}}	
%		%		\put(10,2){{\footnotesize $Y_0$}}							
%		
%	\end{overpic}
%	\bigskip
%	\caption{A homoclinic-like loop $\Gamma_0$ of $Z_0$ at a fold-regular singularity $p_0$.}	\label{generic_cycle_fig}
%\end{figure} 

\section{Quasi-generic loops and Main results}\label{main_sec}

Our aim is to describe the bifurcation diagram of a vector field $Z_0\in\Or$ satisfying hypotheses \textbf{(G)}, \textbf{(T)}, and \textbf{(H)} around its homoclinic-like loop $\Gamma_0$ at $p_0$ (see Figure \ref{generic_cycle_fig}), and characterize the dynamical features arising from such a connection. 

Generally speaking, we prove that, under some constraints, this loop is generic in one-parameter families in $\Or$. %Furthermore, we obtain a modulus of stability of $Z_0$ at $\Gamma_0$.
In what follows, we consider some classes of vector fields in $\Or$ and we state our main results concerning this topic.

Consider $Z_0\in \Or$ satisfying \textbf{(G)}, \textbf{(T)}, and \textbf{(H)}, and recall that the sliding vector field $F_{Z_0}$ is defined on the entire neighborhood $V_0$ (via extension) and it foliates $V_0$ by curves transverse to $\gamma_0$. In light of this, the fold line map $\psi_0:\gamma_0\rightarrow \gamma_0$ given in \eqref{slmap} is still defined herein in the same way. Nevertheless, in this case, $\psi_0$ is defined through orbits of $X_0$, $Y_0$ and virtual sliding orbits of $Z_0$ for some points of $\gamma_0$. 

In fact, remark that $p_0$ splits the curve $\gamma_0$ into two connected components named $C_\gamma^1$ and $C_\gamma^2$. Analogously, $p_0$ splits $\zeta_0$ into $C_\zeta^1$ and $C_\zeta^2$. Without loss of generality, assume that $C_\gamma^1$ and $C_\gamma^2$ are mapped onto $C_\zeta^1$ and $C_\zeta^2$ through the orbits of $X_0$ and $Y_0$, respectively. Now, one of the components of $\zeta_0$, say it $C_\zeta^1$, is contained in $\s^s$ and the other one is contained in $\s^c$. 

Thus, the points $p\in\gamma_0$ and $\psi_0(p)\in\gamma_0$ are just connected by an orbit of $Z_0$ if, and only if  $p\in \overline{C_\gamma^1}$ (which is mapped onto $\overline{ C_\zeta^1} \subset \overline{\s^s}$ by orbits of $Z_0$). It follows that only the restriction of $\psi_0$ to $\overline{C_\gamma^1}$ describes the dynamics of $Z_0$. 

%\textcolor{red}{revisar abaixo e pensar na notacao para os conjuntos. é a melhor opção? Diminuir número de condições?}

\begin{definition}\label{lambda1}
	Define $\Lambda_1$ as the set of vector fields $Z_0\in\Or$ such that
	\begin{enumerate}[i)]
		\item $Z_0$ satisfies hypotheses \textbf{(G)}, \textbf{(T)} and \textbf{(H)};
		\item $F_{Z_0}$ is transverse to $\zeta_0$ at $p_0$;
		\item The fold line map $\psi_{0}:\gamma_0\rightarrow\gamma_0$ induced by $Z_0$ has a hyperbolic fixed point at $p_0$.
	\end{enumerate}
	If $Z_0\in\Lambda_1$, then we say that $Z_0$ has a \textbf{quasi-generic loop $\Gamma_0$ at the fold-regular singularity $p_0$}.
\end{definition}
\begin{remark} Throughout the text, we also refer to a quasi-generic loop at a fold-regular singularity simply by a quasi-generic loop.
\end{remark}	

\begin{example}\label{exa_prop}
Given $b\neq 0$ and $\ag\in\R$ such that $|\ag|\neq 1$ and $-\ag/(1-\ag)\notin[0,1]$, consider the Filippov system $Z_0=(X_0,Y_0)$ with switching manifold $\s=\{(x,y,z)\in\rn{3};\ z=0 \}$, where $X_0$ is given by
	\begin{equation}\label{X_0}
	\begin{array}{lcl}
	X_0(x,y,z)&=&\left(\begin{array}{c}
	0\\
	1\\
	y(2-3y)
	\end{array}\right),
	\end{array}
	\end{equation}
	and $Y_0$ is given by
	\begin{equation}\label{Y_0}
	\begin{array}{lcl}
	Y_0(x,y,z)&=&\left(\begin{array}{c}
	\vspace{0.2cm}-(1-\ag)x+\dfrac{\bg (1-\ag)x^2}{\ag+(1-\bg)y}\\
	\vspace{0.2cm}-1+\dfrac{\bg x}{\ag+(1-\ag)y}\\
	1-2\left(\dfrac{y(\ag+(1-\ag)y)-\bg x}{\ag+(1-\ag)y-\bg x}\right)
	\end{array}\right).
	\end{array}
	\end{equation}
	Therefore, $Z_0\in\Lambda_1$, and the following statements hold.
	\begin{enumerate}[i)]
		\item $Z_0$ has a quasi-generic loop $\Gamma_0$ at the fold-regular singularity $(0,0,0)$, which is contained in the plane $x=0$;
		\item $X_0,Y_0$ are vector fields of class $\mathcal{C}^{\infty}$ around $\Gamma_0$;
		\item In the plane $x=0$, $Z_0$ coincides with the Filippov system $Z_0^*=(X_0,Y_0^*)$, where $Y_0^*(x,y,z)=(0,-1,1-2y)$;
		\item The fold line map of $Z_0$ is given by $\psi_{Z_0}(x)=\ag x+\er(x^2)$. 
	\end{enumerate}
\end{example}

In the result below, we show the robustness of quasi-generic loops in one-parameter families of Filippov systems.

\begin{mtheorem}\label{cod1thm}
	Given $Z_0\in\Lambda_1$. There exist a solid torus $\mathcal{A}_0$ around $\Gamma_0$, a neighborhood $\V_0$ of $Z_0$ in $\Or$ and a $\Cr$ function $\zeta:\V_0\rightarrow \R$,  such that $\zeta(Z_0)=0$, and $\zeta(Z)=0$ if, and only if, $Z$ has a unique quasi-generic loop $\Gamma$ at a fold-regular singularity $p$ contained in $\mathcal{A}_0$. Furthermore, $0$ is a regular value of $\zeta$, and thus $\Lambda_1$ is a codimension one $\Cr$-submanifold of $\Or$.
\end{mtheorem}

Now, we distinguish the following situations.
\begin{enumerate}
	\item [\textbf{(Cyl)}] The mapping $\psi_0$ preserves the connected components $C_\gamma^1$ and $C_\gamma^2$ of $\gamma_0$;
	\item [\textbf{(Mob)}] The mapping $\psi_0$ exchanges the connected components $C_\gamma^1$ and $C_\gamma^2$ of $\gamma_0$.
\end{enumerate}

Define $\Lambda_1^{C}$ and $\Lambda_1^{M}$ as the subsets of $\Lambda_1$ containing the Filippov systems $Z_0$ satisfying \textbf{(Cyl)} and \textbf{(Mob)}, respectively, and consider the following cases.
\begin{enumerate}
	\item [\textbf{(N)}] The hyperbolic fixed point $p_0$ of $\psi_0$ is attractive;
	\item [\textbf{(S)}] The hyperbolic fixed point $p_0$ of $\psi_0$ is repulsive.
\end{enumerate}

Notice that, if $Z_0\in\Lambda_1^C$, then $\gamma_0$ self-connects through orbits of $X_0$, $Y_0$, $F_{Z_0}$ and virtual orbits of $F_{Z_0}$ as a topological cylinder. Nevertheless, if $Z_0\in\Lambda_1^M$, then $\gamma_0$ self-connects as a topological Möbius strip. See Figure \ref{classes_fig}.

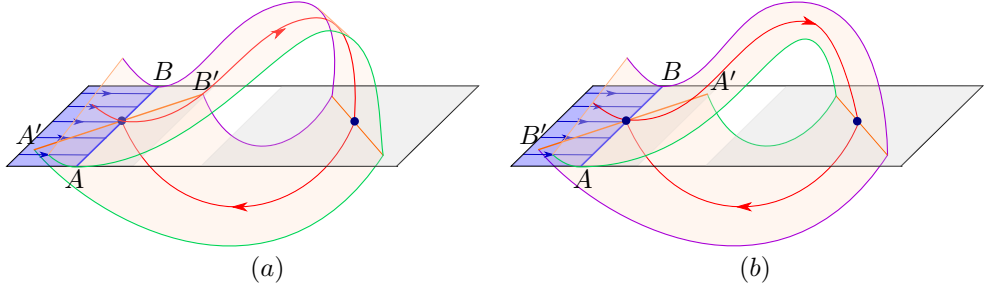
\begin{figure}[h!]
	\centering
	\bigskip
	\begin{overpic}[width=13cm]{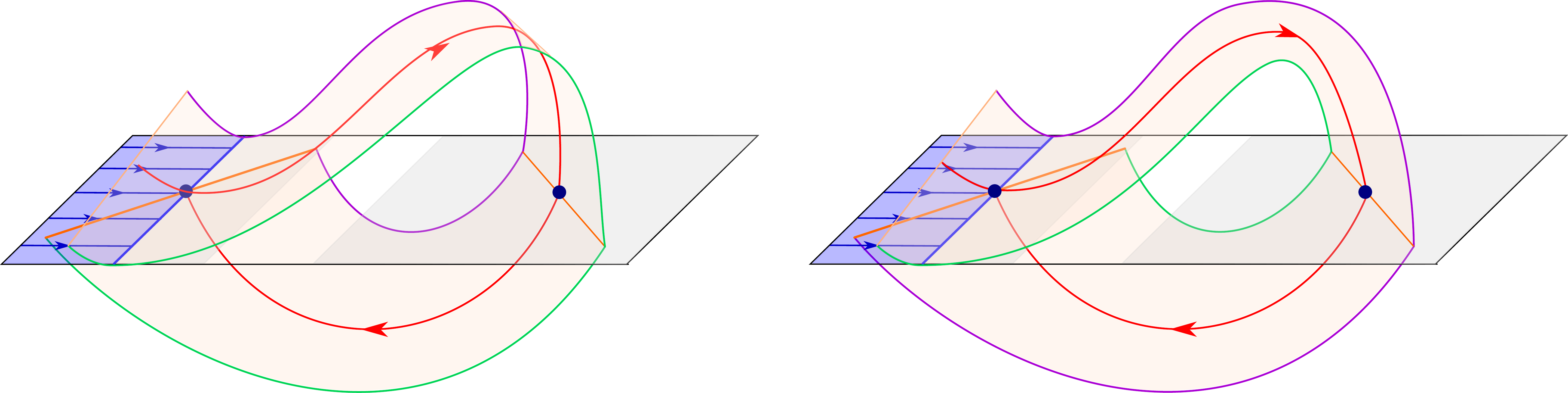}
		%									\begin{overpic}[grid,tics=5,width=13cm]{Figures/classes.pdf}			
		\put(6,6){{\footnotesize $A$}}		
		\put(15,17){{\footnotesize $B$}}
		\put(1,10.5){{\footnotesize $A'$}}		
		\put(19,16){{\footnotesize $B'$}}			
		\put(58,6){{\footnotesize $A$}}		
		\put(67,17){{\footnotesize $B$}}
		\put(52.5,10.5){{\footnotesize $B'$}}		
		\put(72,16){{\footnotesize $A'$}}	
		\put(25,-3){{\footnotesize $(a)$}}		
		\put(75,-3){{\footnotesize $(b)$}}											
		%		\put(0,10){{\footnotesize $\s$}}	
		%		%\put(57,10){{\footnotesize $\s$}}	
		%		\put(26,22){{\footnotesize $\Gamma_0$}}			
		%		\put(10,20){{\footnotesize $X_0$}}	
		%		\put(10,2){{\footnotesize $Y_0$}}							
		
	\end{overpic}
	\bigskip
	\caption{Quasi-generic loop $\Gamma_0$ of $(a)$ $Z_0\in\Lambda_1^C$ and $(b)$ $Z_0\in\Lambda_1^M$.}	\label{classes_fig}
\end{figure}

It is worth mentioning that, if $Z_0\in\Lambda_1^C$, all the iterations of the fold line map $\psi_0$ (defined in the fold line $\gamma_0$) captures the dynamics of $Z_0$, since $\psi_{0}(\overline{C^1_{\gamma}})\subset \overline{C^1_{\gamma}}$ and thus $\psi_0\big|_{\overline{C^1_{\gamma}}}$ defines a dynamical system in $\overline{C^1_{\gamma}}$. Although, it does not hold when $Z_0\in\Lambda_1^M$, since $\psi_{0}(\overline{C^1_{\gamma}})\subset \overline{C^2_{\gamma}}$, which means that $\psi_{0}\big|_{\overline{C^1_{\gamma}}}$ can not be iterated. In Section \ref{propertiesfr} below, we discuss how to adapt the fold line map $\psi_0$ to correctly describe the dynamics of $Z_0\in\Lambda_1^M$ in $\gamma_0$.

%	No primeiro caso, o campo eh transiente e vemos que as duas curvas se colam como um cilindro atraves do deslize. Nesse caso, $\psi_0$ descreve exatamente o comportamento das orbitas e pode ser iterado em qualquer ordem.
%	
%	No segundo caso, eh bem mais complicado, pq o mapa $\psi_0$ nao pode ser iterado novamente e teriamos que conhecer o que acontece com $\pi(Z_0)^n(S^0_{Z_0})$. Aparentemente, nesse caso as duas curvas tem que se colar atraves de uma fita de Mobius. De fato, um ponto $(0,v)$, com $v<0$ negativo vai atraves de $\pi(Z_0)$ para o ponto $(v,0)$, e este vai ser levado para o ponto $(0,\tilde{v})$, com $\tilde{v}>0$ atraves do deslize. Logo, a menos de difeomorfismo, o fluxo liga o segmento $(0,v)$ com  o segmento $(-v,0)$, $v>0$.provando a correlacao com a faixa de mobius. 

In the remaining results of this section, we consider only vector fields $Z_0\in\Lambda_1^C$, in order to provide an amenable analysis, nevertheless we believe that the same conclusions hold for vector fields in $\Lambda_1^M$ through slight modifications. 

The next result is devoted to identify minimal sets  bifurcating from a quasi-generic loop of a Filippov system $Z_0\in\Lambda_1^C$. 

\begin{mtheorem}\label{bifthm}
	Let $Z_0\in\Lambda_1^{C}$ having a quasi-generic loop $\Gamma_0$ at a fold-regular singularity $p_0\in\s$ and consider the torus $\mathcal{A}_0$ given by Theorem \ref{cod1thm}. If $\mathcal{Z}:(-\e,\e)\rightarrow \Or$ is a one-parameter $\mathcal{C}^1$ family such that $\mathcal{Z}(0)=Z_0$ and $\mathcal{Z}$ is transverse to $\Lambda_1$, then the following statements hold.
	\begin{enumerate}
		\item If $Z_0$ satisfies condition $(N)$, then $\mathcal{Z}(\cg)$ has a unique closed connection $\Gamma_{\cg}$ in $\mathcal{A}_0$ which is a sliding cycle when $\cg<0$ and an attractive hyperbolic crossing limit cycle when $\cg>0$, or vice-versa (see Figure \ref{unf_fig}).
		\item If $Z_0$ satisfies condition $(S)$, then $\mathcal{Z}(\cg)$ has a unique hyperbolic crossing limit cycle for either $\cg<0$ or $\cg>0$, and it has at most a unique sliding cycle in $\mathcal{A}_0$.
	\end{enumerate}
\end{mtheorem}

\begin{example}\label{examplebif}
Consider the one-parameter family of Filippov systems $Z_{\cg}=(X_0,Y_{\cg})$, where $X_0$ is given by \eqref{X_0}, $Y_{\cg}=dM\circ Y_\cg^*\circ M^{-1}$, $Y_\cg^*$ is the vector field given by
\begin{equation*}
\begin{array}{lcl}
Y_\cg^*(x,y,z) &=&\left(\begin{array}{c}
\cg\\
-1\\
1 -2y
\end{array}\right),
\end{array}
\end{equation*}
and $M:\rn{3}\rightarrow \rn{3}$ is the map given by
\begin{equation*}
\begin{array}{lcl}
M(x,y,z)&=&\left(\begin{array}{c}
\vspace{0.2cm}x\left(2-y+x(y-1)\right)\\
\vspace{0.2cm}y-x(y-1)\\
z
\end{array}\right).
\end{array}
\end{equation*}
Then $Z_{\cg}$ is an unfolding of the Filippov system $Z_0$ given by Example \ref{exa_prop} (with $\ag=2$ and $b=1$) at $\cg=0$, and there exists a solid torus $\mathcal{A}_0$ around the quasi-generic loop $\Gamma_0$ at the fold-regular singularity $(0,0,0)$ of $Z_0$ such that the following statements hold.
\begin{enumerate}[i)]
	\item If $\cg>0$, then $Z_{\cg}$ has a unique sliding cycle $\Gamma_\cg$ in $\mathcal{A}_0$, which is of repelling type;
	\item If $\cg=0$, then $Z_{\cg}$ has a unique quasi-generic loop $\Gamma_0$ passing through a fold-regular singularity in $\mathcal{A}_0$;
	\item If $\cg<0$, then $Z_{\cg}$ has a unique crossing limit cycle $\Gamma_\cg$ in $\mathcal{A}_0$, which is hyperbolic and of saddle type. 
\end{enumerate}

It is worth mentioning that, concerning the family $Z_{\cg}$ above, the first part of Theorem \ref{bifthm} also holds for the case when $Z_0$ satisfies $(S)$, despite of the stability.
\end{example}

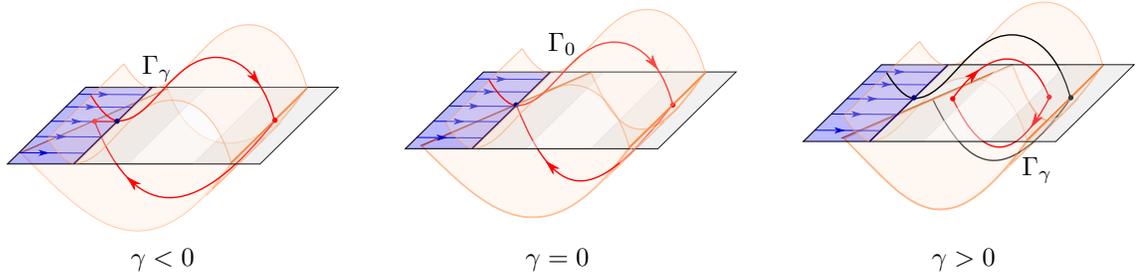
\begin{figure}[h!]
	\centering
	\bigskip
	\begin{overpic}[width=15cm]{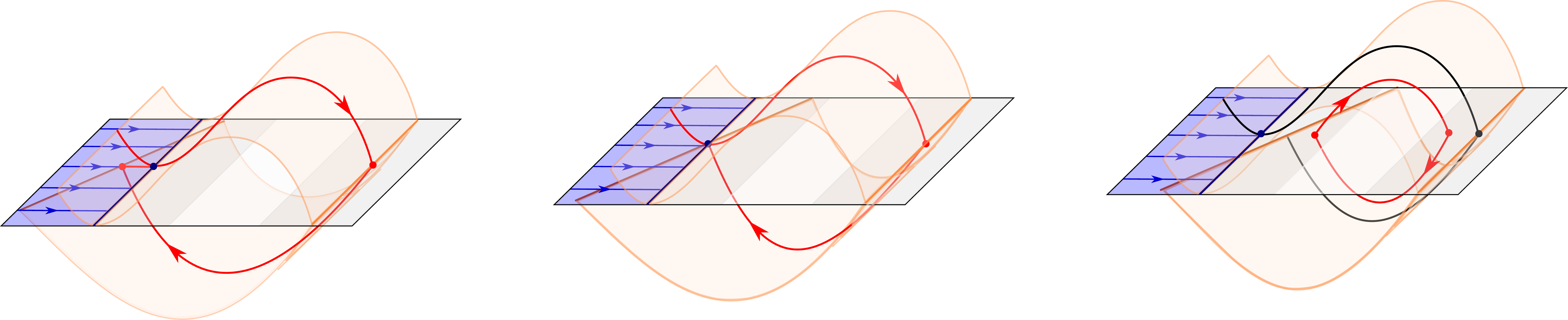}
		%									\begin{overpic}[grid,tics=5,width=8cm]{Figures/cyl_mob.pdf}								
		\put(46,-3){{\footnotesize $\cg=0$}}
		\put(48,16){{\footnotesize $\Gamma_0$}}																
		\put(11,-3){{\footnotesize $\cg<0$}}
		\put(12,14){{\footnotesize $\Gamma_\cg$}}			
		\put(82,-3){{\footnotesize $\cg>0$}}	
		\put(90,5){{\footnotesize $\Gamma_\cg$}}			
	\end{overpic}
	\bigskip
	\caption{A versal unfolding of $Z_0\in\Lambda_1^C$ satisfying (N) in $\Or$.}	\label{unf_fig}
\end{figure} 	

Now, we combine the informations encoded by the first return map and the sliding dynamics to analyze the stability of a quasi-generic loop.

\begin{mtheorem}\label{stabthm}
	Let $Z_0\in\Lambda_1^{C}$ having a quasi-generic loop $\Gamma_0$ at a fold-regular singularity $p_0\in\s$ and consider the torus $\mathcal{A}_0$ given by Theorem \ref{cod1thm}. The following statements hold.
	\begin{enumerate}
		\item If $Z_0$ satisfies condition $(N)$, then $\Gamma_0$ is an asymptotically stable minimal set;
		\item If $Z_0$ satisfies condition $(S)$, then there exists a piecewise-smooth curve $\bg$ passing through $p_0$ such that the basin of attraction of $\Gamma_0$ is given by 
		$$\mathcal{B}=\left\{p\in\mathcal{A}_0;\textrm{ there exist a time }t(p) \textrm{ such that }\p_{Z_0}(t(p);p)\in\bg\right\}.$$
		Furthermore, one of the two connected components of $\bg\setminus\{p_0\}$ is contained in $\s^s$ and the other one contained in $\s^c$.   
	\end{enumerate}
	
\end{mtheorem}	

We introduce a notion of equivalence in $\Lambda_1$ which allows us to obtain a modulus of stability for $Z_0$.

\begin{definition}\label{weakequiv}
	Let $Z,Z_0\in\Lambda_1$ having quasi-generic loops $\Gamma$ and $\Gamma_0$ at fold-regular singularities $p\in\s$ and $p_0\in\s$, respectively. We say that $Z$ and $Z_0$ are \textbf{ weakly topologically equivalent} at $(\Gamma,\Gamma_0)$ if there exist sufficiently small solid tori $\mathcal{A}_0$ and $\mathcal{A}$ containing $\Gamma_0$ and $\Gamma$, respectively, and an order-preserving homeomorphism $h:\mathcal{A}\rightarrow \mathcal{A}_0$ such that
	\begin{enumerate}[i)]
		\item $\mathcal{A}\cap \s$ and $\mathcal{A}_0\cap \s$ have connected curves $S_{Z}$ and $S_{Z_0}$ of fold-regular singularities of $Z$ and $Z_0$ intersecting $\partial \mathcal{A}\cap \s$ and $\partial \mathcal{A}_0\cap \s$ transversally, and there are no more $\s$-singularities of $Z$ and $Z_0$ contained in $\mathcal{A}\cap \s$ and $\mathcal{A}_0\cap \s$, respectively;
		\item $h:S_Z\rightarrow S_{Z_0}$ is a diffeomorphism such that $h(p)=p_0$;
		\item $h(\Gamma)=\Gamma_0$ and $h(\mathcal{A}\cap\s)=\mathcal{A}_0\cap\s$;
		\item $h$ carries orbits of $Z$ onto orbits of $Z_0$. 
	\end{enumerate}
	
\end{definition}

\begin{remark}
	Notice that, it follows from Section \ref{globalconsec} that, given $Z\in\Lambda_1$, we find a sufficiently small torus $\mathcal{A}$, such that item  $(i)$ of Definition \ref{weakequiv} is satisfied.
\end{remark}

Finally, given $Z_0\in\Lambda_1$, we define the \textbf{modulus of weak-stability} of $Z_0$ as 
\begin{equation*}
\mathcal{W}(Z_0)= \psi_0'(p_0),
\end{equation*}
where $\psi_0$ is the fold-line map of $Z_0$.
%\begin{thmD}
%	Given $Z_0,\widetilde{Z_0}\in\Lambda_1^{C}$ having quasi-generic loops $\Gamma_0$ and $\widetilde{\Gamma_0}$ at fold-regular singularities $p_0,\widetilde{p_0}\in\s$ of type $(S)$, respectively, then $Z_0$ and $\widetilde{Z}_0$ are locally topologically $\s$-equivalent at $(\Gamma_0,\widetilde{\Gamma_0})$ if and only if
%	$$\mathcal{P}(Z_0)=\mathcal{P}(\widetilde{Z_0}).$$
%\end{thmD}	
\begin{mtheorem}\label{modthm}
	Let $Z_0,\widetilde{Z_0}\in\Lambda_1^{C}$ have quasi-generic loops $\Gamma_0$ and $\widetilde{\Gamma_0}$ at fold-regular singularities $p_0,\widetilde{p_0}\in\s$ of type $(S)$, respectively.
	If $Z_0$ and $\widetilde{Z}_0$ are weakly topologically equivalent at $(\Gamma_0,\widetilde{\Gamma_0})$, then
	$$\mathcal{W}(Z_0)=\mathcal{W}(\widetilde{Z_0}).$$
\end{mtheorem}	

A direct consequence of Theorem D is given in the next corollary.
\begin{corollary}
	If $Z_0\in\Lambda_1^C$ satisfies $(S)$, then $Z_0$ has $\infty$-moduli of weak-stability in $\Lambda_1^C$. It means that there are infinitely many Filippov systems $Z_n\in\Lambda_1^C$, $n\in\N$, such that $Z_{n_1}$ and $Z_{n_2}$ are not weakly topologically equivalent, for every $n_1,n_2\geq 0$ and $n_1\neq n_2$.
\end{corollary}	

%	\begin{corollary}
%		Assume that $Z_0=(X_0,Y_0)\in\Or$ has a quasi-generic loop $\Gamma_0$ at a fold-regular singularity $p_0$ and let $\Sigma_1$ be the codimension one submanifold of $\Or$ given by Theorem A. Then $Z_0$ has $\infty$-moduli of weak-stability in $\Sigma_1$.
%	\end{corollary}	

%\begin{thmE}
%	Given $Z_0,\widetilde{Z_0}$ having quasi-generic loops $\Gamma_0$ and $\widetilde{\Gamma_0}$ at fold-regular singularities $p_0,\widetilde{p_0}\in\s$ of type $(N)$, then $Z_0$ and $\widetilde{Z_0}$ are locally topologically equivalent at $(\Gamma_0,\widetilde{\Gamma_0})$.
%\end{thmE}	
%
%\begin{corollary}
%	Assume that $Z_0=(X_0,Y_0)\in\Or$ has a quasi-generic loop $\Gamma_0$ at a fold-regular singularity $p_0$ and let $\Sigma_1$ be the codimension one submanifold of $\Or$ given in Theorem A. Then $Z_0$ has $0$-moduli of stability in $\Sigma_1$.
%\end{corollary}	
%}

\section{Structure of a homoclinic-like loop}\label{struct_sec}

In this section, we characterize the first return map $\mathcal{P}_0$ and the fold line map $\psi_0$ associated to a homoclinic-like loop $\Gamma_0$ of a system $Z_0\in\Or$. Furthermore, given a small solid torus $\mathcal{A}_0$ around $\Gamma_0$ and a vector field $Z$ sufficiently near to $Z_0$, we associate a first return map $\mathcal{P}_Z$ and a fold line map $\psi_Z$ which describe the dynamics of $Z$ inside $\mathcal{A}_0$.

Let $Z_0=(X_0,Y_0)\in\Or$ satisfying \textbf{(G)}, \textbf{(T)} and \textbf{(H)}. In order to characterize the first return map $\mathcal{P}_0$ given in \eqref{first}, we shall write 
$$\mathcal{P}_0=\mathcal{D}_0\circ \mathcal{T}_0,$$
where $\mathcal{D}_0$ is a diffeomorphism and $\mathcal{T}_0$ is a $\Cr$ map describing all trajectories around a fold-regular singularity. We refer to $\mathcal{T}_0$ as the \textbf{transition map} of $Z_0$ at the fold-regular singularity $p_0$. In \cite{AGN19} one finds the definition of transition maps in the planar case.

In Section \ref{trans_sec}, we construct and characterize the transition map $\mathcal{T}_0$. In Section \ref{first_sec}, we describe the complete first return map $\mathcal{P}_0$. Finally, in Section \ref{foldline_sec}, we characterize the fold line map $\psi_0$.

\subsection{Transition Map}\label{trans_sec}

Without loss of generality, assume that $p_0$ is a fold point of $X_0$ and a regular point of $Y_0$. In this case, the transition map depends only on the smooth vector field $X_0$. 

%Let $p_0^*\in\Gamma_0\cap M^+$ be such that the arc of $\Gamma_0$ between $p_0$ and $p_0^*$ is contained in $\overline{M^+}$ and consider a local transversal section $\tau$ of $X_0$ at $p_0^*$.

Since $p_0$ is a visible fold-regular singularity of $Z_0$, it follows from Proposition 33 of \cite{GT19} that there exist  $a_0<0<b_0$, and neighborhoods $\V_0$ of $Z_0$ in $\Or$ and $V_0$ of $p_0$ in $\s$ such that:
\begin{enumerate}[i)]
	\item $V_0$ is compact;
	\item each $Z\in\V_0$ has a curve $\gamma_{Z}:[a_0,b_0]\rightarrow V_0 $, composed just by visible fold-regular singularities of $Z$;
	\item $V_0\setminus \textrm{Im}(\gamma_{Z})$ has only regular-regular points of $Z$;
	\item $\gamma_Z$ intersects $V_0$ transversally at $\gamma_Z(a_0)$ and $\gamma_Z(b_0)$;
	\item $\gamma_Z(t)\in \textrm{int}(V_0)$ for each $t\in (a_0,b_0)$;
	\item $\gamma_{Z_0}(0)=p_0$.
\end{enumerate}

From Proposition \ref{Vishik}, there exist neighborhoods $U_0\subset \rn{3}$ of $p_0$ and $W_0\subset \rn{3}$ of the origin such that $V_0\subset U_0$, and a local coordinate system $(x,y,z): (U_0,p_0)\rightarrow (W_0,0)$ such that $f(x,y,z)=z$ and  $X_0$ is given by $$X_0(x,y,z)=(0,1,y).$$

We denote the set $V_0$ in the coordinates $(x,y,z)$ by $\widetilde{V_0}$. Notice that $\textrm{Im}(\gamma_{Z_0})$ coincides with a segment of the $x$-axis in the plane $z=0$ containing the origin, and the flow of $X_0$ is given by $$\p_{X_0}(t;x,y,z)=\left(x,y+t,z+\dfrac{(y+t)^2}{2}-\dfrac{y^2}{2}\right).$$

Given $\e>0$ sufficiently small, let $\tau$ be a local transversal section of $X_0$ at $p_\e^*=(0,\sqrt{2\e},\e)$ contained in the plane $z=\e$ and notice that the origin is connected to $p_\e^*$ through an orbit of $X_0$. From the Implicit Function Theorem, $\V_0$ can be considered such that, for each $Z=(X,Y)\in\V_0$, a point $(x,y,0)\in\widetilde{V_0}$ reaches $\tau$ through the flow of $X$ for a positive time $t(X;x,y)$.
%
%
%
%
%
%
%$\tau=\{(x,y,\e); (x,y,\e)\in W_0\}$, for some $\e>0$ sufficiently small, be a local 
%
%
%
%
%
%%Also, consider that $V_0$ is contained in the domain $U_0\subset \rn{3}$ of the coordinate system $(x,y,z): (U_0,p_0)\rightarrow (W_0,0) $ at $p_0$ given by Vishik's Theorem (see \cite{V}). In this coordinates, $X_0$ is written as $$X_0(x,y,z)=(0,1,y)$$ with $h(x,y,z)=z$. Thus, $\textrm{Im}(\gamma_{Z_0})$ coincides with a segment of the $x$-axis containing the origin in the plane $z=0$.
%
%Let $\tau^*=\{(x,y,\e); (x,y,\e)\in W_0\}$, for some $\e>0$ sufficiently small, and let $\tau_\e\subset U_0$ be the preimage of $\tau^*$ by the coordinates $(x,y,z)$. 
%
%Since $\tau$ is transverse to $\Gamma_0$, then the orbit of $X_0$ through $p_0$ reaches $\tau$ (transversally) at a point $p_0^*\in\tau$. Thus, we can use the Implicit Function Theorem to reduce $\V_0$ and $V_0$ such that, for each $Z=(X,Y)\in\V_0$, a point $p$ of $V_0$ reaches $\tau$ through the flow of $X$ for a positive time $t(X,p)$.

Therefore, given $Z=(X,Y)\in\V_0$, we define the \textbf{full transition map} $\mathcal{T}_Z: \widetilde{V_0}\rightarrow \tau$ of $Z$ by
\begin{equation}\label{transmap}
\mathcal{T}_Z(x,y,0)=\p_X(t(X;x,y);x,y,0),
\end{equation}
and notice that the dependence of $\mathcal{T}_Z$ on $Z$ is of class $\Cr$. See Figure \ref{transmap_fig}.

\begin{figure}[h!]
	\centering
	\bigskip
	\begin{overpic}[width=6cm]{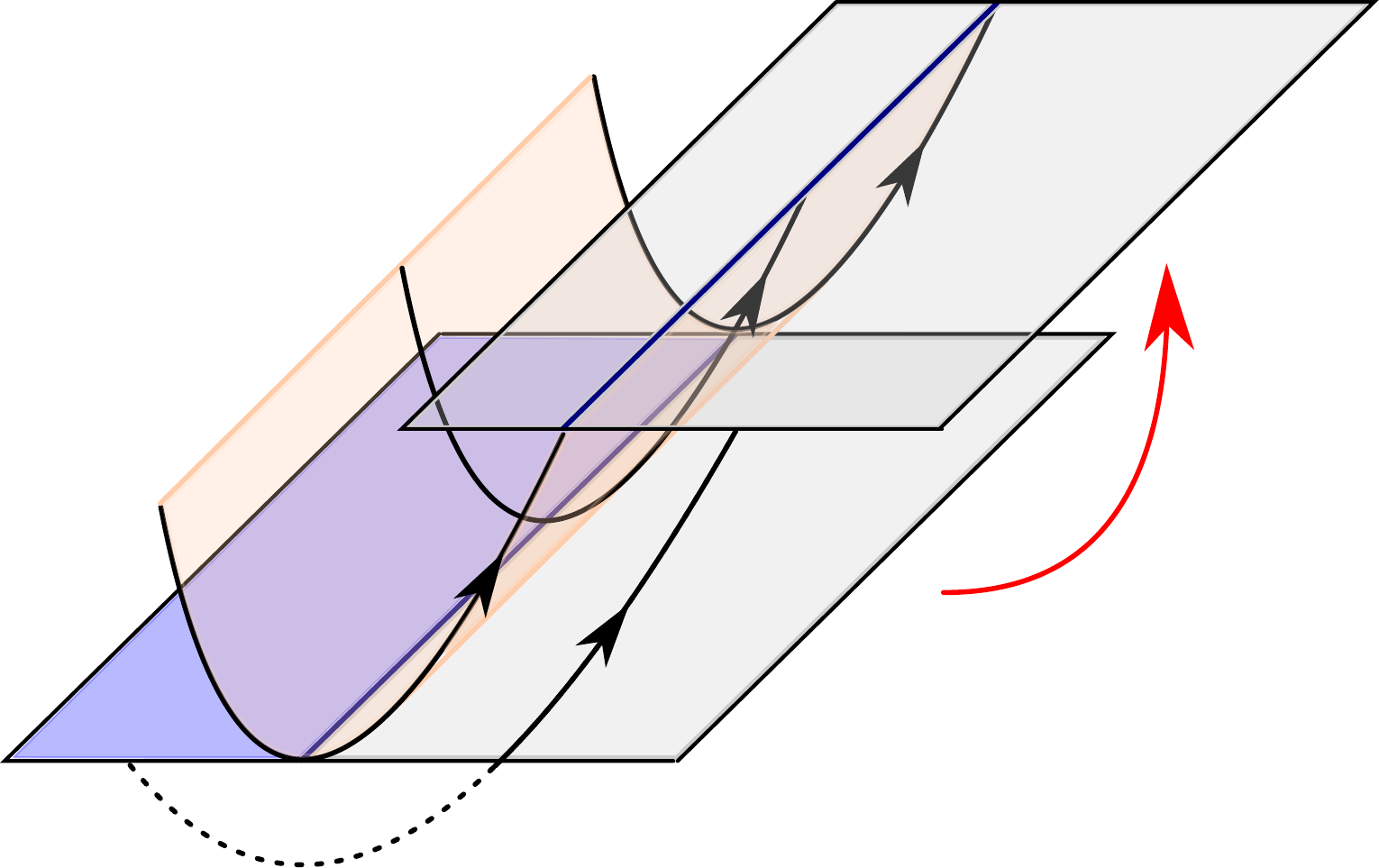}
		%								\begin{overpic}[grid,tics=5,width=6cm]{Figures/transmap.pdf}			
		\put(96,55){{\footnotesize $\tau$}}		
		\put(60,13){{\footnotesize $\s$}}		
		\put(85,25){{\footnotesize $\mathcal{T}_Z$}}	
		\put(0,20){{\footnotesize $M^+$}}	
		\put(3,10){{\footnotesize $\s^{ss}$}}
		\put(0,0){{\footnotesize $M^-$}}									
		\put(21,4){{\footnotesize $\cg_Z$}}		
	\end{overpic}
	\bigskip
	\caption{Full transition map $\mathcal{T}_Z$ for a vector field $Z$ near $Z_0$.}	\label{transmap_fig}
\end{figure}
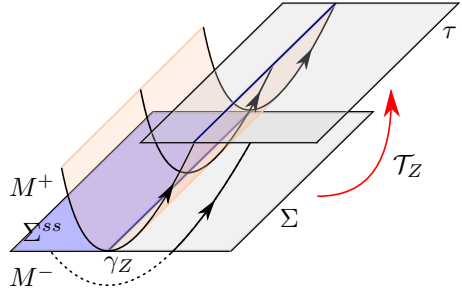

Using the expression of the flow of $X_0$, an easy computation allows us to check that $\mathcal{T}_0:=\mathcal{T}_{Z_0}$ is given by
\begin{equation*}
\mathcal{T}_0(x,y,0)= (x,\sqrt{y^2+2\e},\e).
\end{equation*}

%Now, since $V_0\subset U_0$, an easy computation allows us to see that the expression of $\pi_1(Z_0)$ in Vishik's coordinate system is given by
%\begin{equation}
%\pi_1(Z_0)(x,y,0)= (x,\sqrt{y^2+2\e},\e).
%\end{equation}	

Finally, for each $Z=(X,Y)\in\V_0$, we use the compactness of $\gamma_Z$ to construct a finite cover of $\textrm{Im}(\gamma_Z)$ by domains of Vishik's coordinate system (see Proposition \ref{Vishik}). Thus, we see that the orbit of $X$ connecting $p\in \widetilde{V_0}$ and a point of $\tau$ is contained in $M^+$ if, and only if, $p\in \widetilde{V_0}\cap\overline{\s^c}$. Therefore, $\mathcal{T}_Z$ describes the real behavior of the trajectories of $Z$ between $\s$ and $\tau$ only in the domain 
$$\sigma_Z=\widetilde{V_0}\cap \overline{\s^c}.$$ Accordingly, we define the \textbf{transition map} of $Z$ as $T_Z=\mathcal{T}_Z\big|_{\sigma_Z}$.

Notice that $T_Z$ is a homeomorphism onto its image and $\mathcal{T}_Z$ is a natural $\Cr$ extension of $T_Z$ to $V_0$ induced by the setting of the problem. Nevertheless, $\mathcal{T}_Z$ is a non-invertible map.

\subsection{First Return Map}\label{first_sec}

Consider the coordinate system and the local transversal section $\tau$ introduced in Section \ref{trans_sec}, and recall that $Y_0$ is transverse to $\s$ at each point of $\widetilde{V_0}$ and $X_0$ is transverse to $\tau$ anywhere. From conditions \textbf{(G)}, \textbf{(T)}, and \textbf{(H)}, it follows from the Implicit Function Theorem that, for each $Z\in\V_0$ (reducing $\V_0$ if necessary), there exists a diffeomorphism $\mathcal{D}_Z: \tau\rightarrow \s$ onto its image induced by regular orbits of $Z$. In particular, denoting $\mathcal{D}_0=\mathcal{D}_{Z_0}$, we obtain
$$\mathcal{D}_0(0,\sqrt{\e},\e)=(0,0,0).$$

We define the \textbf{full first return map} $\mathcal{P}_{Z}: \widetilde{V_0}\rightarrow \s$ of $Z\in\V_0$ as
\begin{equation}
\label{firstreturnmap}
\mathcal{P}_{Z}(x,y,0)=\mathcal{D}_Z\circ\mathcal{T}_Z(x,y,0),
\end{equation}
where $\mathcal{T}_Z$ is the full transition map of $Z$ given in \eqref{transmap}. Accordingly, the \textbf{first return map} of $Z$ is defined by $P_Z=\mathcal{D}_Z\circ T_Z$, where $T_Z$ is the transition map of $Z$. 

If $p\in\sigma_Z$, then $p$ and $\mathcal{P}_Z(p)$ are connected by a trajectory of $Z$, nevertheless, if $p\in \widetilde{V_0}\setminus \sigma_Z$, then $p$ and $\mathcal{P}_Z(p)$ are related by a virtual trajectory of $Z$. It follows that $P_Z$ is a $\Cr$ homeomorphism (onto its image) which completely describes the crossing dynamics of $Z$ inside the torus $\mathcal{A}_0$.% generated by $\Gamma_0$ and $ \widetilde{V_0}$. 

Notice that both $P_Z$ and $\mathcal{P}_Z$ have a $\Cr$ dependence on $Z$. Also, $\mathcal{P}_Z$ is a non-invertible map which is a $\Cr$ extension of $P_Z$ to $\widetilde{V_0}$. In particular, the origin is a fixed point of $\mathcal{P}_0=\mathcal{P}_{Z_0}$, corresponding to the homoclinic-like loop $\Gamma_0$ of $Z_0$. 

\subsection{Fold Line Map}\label{foldline_sec}

Finally, we characterize the fold line map $\psi_0$ of $Z_0$ induced by the sliding dynamics. In addition, this map can be constructed for every $Z\in\Or$ sufficiently near $Z_0$. Consider the same notation used in Section \ref{struct_sec}.

Denote the fold line $\gamma_Z$ of $Z$ by $S_{Z}$. Since $S_{Z_0}\cap \widetilde{V_{0}}$ is composed by fold-regular singularities of $Z_0$, it follows from Lemma 24 in \cite{GT19} that, reducing $\widetilde{V_0}$ if necessary, the sliding vector field $F_{Z_0}$ is extended onto $\widetilde{V_0}$, and it is transverse to $S_{Z_0}$ at $p_0$. Define the $\Cr$ map $\mathcal{G}: \V_0\times \widetilde{V_0}\times\rn{}\rightarrow\s$ given by
$$\mathcal{G}(Z,(x,y,0),s)=Xf(\p_{F_Z}(s;x,y,0)),$$
where $Z=(X,Y)$. Since $S_{Z_0}=X_0f^{-1}(0)$, it follows that
$\mathcal{G}(Z_0,(0,0,0), 0)=X_0f(0,0,0)=0$  and $\partial_s\mathcal{G}(Z_0,(0,0,0), 0)=dX_0f(0,0,0)\cdot F_{Z_0}(0,0,0)\neq0.$

From the Implicit Function Theorem, reducing $\widetilde{V_0}$ and $\V_0$ if necessary, there exists a unique $\Cr$ function $s^*:\V_0\times\widetilde{V_0}\rightarrow\rn{}$ such that $\mathcal{G}(Z,(x,y,0),s^*(Z,(x,y,0)))=0$.

Consider the full first return map $\mathcal{P}_Z:\widetilde{V_0}\rightarrow\s$ given by \eqref{firstreturnmap}. Now, for a sufficiently small neighborhood $\widetilde{V_1}$ of $(0,0,0)$ contained in $\widetilde{V_0}$ and reducing $\V_0$ if necessary, we define the \textbf{full fold line map} $\Psi_Z: S_{Z}\cap \widetilde{V_1}\rightarrow S_{Z}\cap \widetilde{V_0}$ by
\begin{equation}\label{foldlinemap}\Psi_Z(p)=\p_{F_Z}(s^*(Z,\mathcal{P}_Z(p));\mathcal{P}_Z(p)),\end{equation}
for each $Z\in\V_0$.

In order to analyze the dynamics encoded by the full fold line map, it is convenient to restrict it to the following domain
\begin{equation}\label{folddom}
\sigma_Z^{FL}= \mathcal{P}_Z^{-1}(\mathcal{P}_Z(S_Z\cap\widetilde{V_1})\cap \s^{s}).
\end{equation}
Accordingly, we define the \textbf{ fold line map} as $\psi_Z=\Psi_Z\big|_{\sigma_Z^{FL}}$. Notice that, $(0,0,0)$ is a fixed point of $\psi_0=\psi_{Z_0}$, and $\Psi_Z$ is a $\Cr$ extension of $\psi_Z$.

\begin{remark}
	Consider a map $\mathcal{H}:\V_0\times [a_0,b_0]\rightarrow \widetilde{V_0}$ such that, for each $Z\in\V_0$, $\mathcal{H}_Z:=\mathcal{H}(Z,\cdot):[a_0,b_0]\rightarrow S_Z\cap \widetilde{V_0}$ is a diffeormorphism onto its image in such a way that,  for some $a_0<a_1<0<b_1<b_0$, $\mathcal{H}_Z\big|_{[a_1,b_1]}$ parameterizes $S_Z\cap \widetilde{V_1}$. Therefore, $\mathcal{H}_Z^{-1}\circ\Psi_{Z}\circ\mathcal{H}_Z:[a_1,b_1]\rightarrow [a_0,b_0]$
	is a family of real diffeomorphisms (onto their image) which is of class $\Cr$ on $Z$. Therefore, if $p_0$ is a hyperbolic fixed point of $\psi_0$, we can use such parameterizations to see that, reducing $\V_0$ if necessary, the full fold line map $\Psi_Z$ has a unique hyperbolic fixed point (with the same type) in $S_Z\cap \widetilde{V_0}$, for each $Z\in\V_0$.
\end{remark}

\subsection{Properties}\label{propertiesfr}

In what follows, we use the full transition map $\mathcal{P}_0$ and the full fold line map $\Psi_0$ to characterize $P_0=P_{Z_0}$ and $\psi_0=\psi_{Z_0}$. We consider the coordinate system $(x,y,z)$ at $p_0$ as in Section \ref{trans_sec}, and from now on, we identify the points $(x,y,0)\in\s$ and $(x,y,\e)\in\tau$ with $(x,y)$. Also, consider the neighborhoods $\widetilde{V_1}$ and $\V_0$ of $(0,0,0)$ and $Z_0$ given in Section \ref{foldline_sec}, respectively.

\begin{lemma}\label{fullfrm}
	Given $Z_0\in\Or$ satisfying conditions \textbf{(G)}, \textbf{(H)}, and  \textbf{(T)}, there exist real constants $\ag_{i,j},\bg_{i,j}\in \rn{}$, $i=0,1,2$ and $j=0,1$ such that the Taylor expansion of the full first return map $\mathcal{P}_0$ of $Z_0$ at the origin is given by
	\begin{equation}\label{firstreturn}
	\mathcal{P}_0(x,y)= \left(\begin{array}{c}
	\ag_{1,0} x+  \ag_{0,1}y^2+ \ag_{2,0}x^2+\ag_{1,1}xy^2+\er(x^3, x^2y^2,y^4)\vspace{0.2cm}\\
	\bg_{1,0} x+  \bg_{0,1}y^2+ \bg_{2,0}x^2+\bg_{1,1}xy^2+\er(x^3, x^2y^2,y^4)
	\end{array}\right).
	\end{equation}
	Furthermore, the following statements hold.
	\begin{enumerate}[i)]
		\item $d=\ag_{1,0}\bg_{0,1}-\ag_{0,1}\bg_{1,0}\neq 0$;
		\item $\sgn(d)=\sgn( J\mathcal{D}_0(0,\sqrt{2\e}))$, where $\mathcal{D}_0:\tau\rightarrow \s$ is the diffeomorphism induced by the flow of $Z_0$ and $J\mathcal{D}_0$ denotes the Jacobian of $\mathcal{D}_0$;
		\item If $F_{Z_0}$ is transverse to $\mathcal{P}_0(S_{Z_0}\cap\widetilde{V_1})$ at the origin, then $\ag_{1,0}\neq 0$.
	\end{enumerate}
	% $d=\ag_{1,0}\bg_{0,1}-\ag_{0,1}\bg_{1,0}\neq 0$ and $\sgn(d)=\sgn( J\mathcal{D}_0(0,\sqrt{2\e}))$, where $\mathcal{D}_0:\tau\rightarrow \s$ is the diffeomorphism induced by the flow of $Z_0$ and $J\mathcal{D}_0$ denotes the Jacobian of $\mathcal{D}_0$. 
\end{lemma}
\begin{proof}
	Since $\mathcal{D}_0$ is a diffeomorphism such that $\mathcal{D}_0(\sqrt{2\e},0)=(0,0)$, it follows that,
	$$
	\mathcal{D}_0(x,y)= \left(\begin{array}{c}
	a_{1,0} x+  a_{0,1}(y-\sqrt{2\e})+ a_{2,0}x^2+a_{1,1}x(y-\sqrt{2\e})+a_{0,2} (y-\sqrt{2\e})^2+\er_3(x,y-\sqrt{2\e})\vspace{0.2cm}\\
	b_{1,0} x+  b_{0,1}(y-\sqrt{2\e})+ b_{2,0}x^2+b_{1,1}x(y-\sqrt{2\e})+b_{0,2} (y-\sqrt{2\e})^2+\er_3(x,y-\sqrt{2\e})
	\end{array}\right),
	$$
	where $a_{i,j},b_{i,k}\in \rn{}$  are constants satisfying $a_{1,0}b_{0,1}-a_{0,1}b_{1,0}\neq 0$.	Also, using the expression of $\mathcal{T}_0$ given in \eqref{transmap}, it follows that
	$$
	\mathcal{T}_0(x,y)= \left(\begin{array}{c}
	x\vspace{0.2cm}\\
	\sqrt{2\e}+ Ky^2+\er_4(y)
	\end{array}\right),
	$$	
	where $K>0$. Straightforwardly, we obtain \eqref{firstreturn} and prove items $(i)$ and $(ii)$.
	
	Finally, assume that $F_{Z_0}$ is transverse to $S_{Z_0}$ at the origin. Denoting $Y_0(x,y,z)=(f_1,f_2,f_3)$ in this coordinate system, where $f_i=f_i(x,y,z)$, $i=1,2,3$, we obtain
	$$Y_0f(x,y,z)=f_3(x,y,z).$$
	
	Recalling that $f(x,y,z)=z$ and $X_0(x,y,z)=(0,1,y)$, we have that the correspondent sliding vector field is expressed as
	$$
	\begin{array}{lcl}
	\vspace{0.2cm}F_{Z_0}(x,y)&=&\dfrac{Y_0f(x,y,0) X_0(x,y,0) -X_0f(x,y,0) Y_0(x,y,0)}{Y_0f(x,y,0)- X_0f(x,y,0)}\\
	\vspace{0.2cm} &=&\dfrac{f_3(x,y,0) (0,1,y) -y(f_1(x,y,0),f_2(x,y,0),f_3(x,y,0))}{f_3(x,y,0)- y}\\
	\vspace{0.2cm} &=&\left(\dfrac{-yf_1(x,y,0)}{f_3(x,y,0)-y}, \dfrac{f_3(x,y,0)-yf_2(x,y,0)}{f_3(x,y,0)-y}\right).
	\end{array}
	$$
	
	Since $Y_0$ is transverse to $\s$ at $p_0$, it follows that $a_0=f_3(0,0,0)\neq 0$, and consequently,
	$F_{Z_0}(0,0)=(0,1).$
	
	Now, notice that $S_{Z_0}=\left\{ (x,0);\ x\in (-\e,\e)\right\}$ and therefore \begin{equation}\label{firstiteration}\zeta_0=\mathcal{P}_0(S_{Z_0}\cap \widetilde{V_1})=\left\{ (\ag_{1,0} x+\er_2(x),\bg_{1,0} x+ \er_2(x));\ x\in (-\e,\e)\right\},\end{equation} for some $\e>0$. It follows that $T_{\textbf{0}}\zeta_0=\textrm{span}\{(\ag_{1,0},\bg_{1,0})\}$, and since $F_{Z_0}$ is transverse to $\zeta_0$ at the origin, we obtain that $\ag_{1,0}\neq 0$.
\end{proof}

\begin{remark}
	Notice that $\mathcal{P}_0(S_{Z_0}\cap\widetilde{V_1})$ coincides with the curve $\zeta_0$ given in Section \ref{main_sec}.
\end{remark}

The proof of the following lemma is straightforward and will be omitted.

\begin{lemma}\label{normalform}
	Consider the same hypotheses of Lemma \ref{fullfrm} and assume that $\ag_{1,0}\neq 0$. Then, the local change of coordinates at the origin of the plane $\s$ given by 
	$$\left\{ \begin{array}{l}
	u=x- \dfrac{\ag_{0,1}}{\ag_{1,0}}y^2,\vspace{0.1cm}\\
	v=y,
	\end{array}\right.
	$$
	brings the full first return map $\mathcal{P}_0$ into
	\begin{equation*}\label{firstreturn2}
	\overline{\mathcal{P}_0}(u,v)= \left(\begin{array}{c}
	\ag_{1,0} u \vspace{0.2cm}\\
	\bg_{1,0} u+ \dfrac{d}{\ag_{1,0}}v^2
	\end{array}\right)+u^2 A_1(u,v)+uv^2 A_2(u,v)+v^4 A_3(u,v),
	\end{equation*}	
	where $A_i(u,v)$ are bounded vector-valued functions.

	%	Given $Z_0\in\Or$ satisfying conditions \textbf{(G)}, \textbf{(H)}, and  \textbf{(T)}, there exist real constants $\ag_{i,j},\bg_{i,j}\in \rn{}$, $i=0,1,2$ and $j=0,1$ such that the Taylor expansion of the first return map $\mathcal{P}_0$ of $Z_0$ at the origin is given by
	%	\begin{equation}\label{firstreturnexp}
	%	\mathcal{P}_0(x,y)= \left(\begin{array}{c}
	%	\ag_{1,0} x+  \ag_{0,1}y^2+ \ag_{2,0}x^2+\ag_{1,1}xy^2+\er(x^3, x^2y^2,y^4)\vspace{0.2cm}\\
	%	\bg_{1,0} x+  \bg_{0,1}y^2+ \bg_{2,0}x^2+\bg_{1,1}xy^2+\er(x^3, x^2y^2,y^4)
	%	\end{array}\right).
	%	\end{equation}
	%	Moreover $\dg=\ag_{1,0}\bg_{0,1}-\ag_{0,1}\bg_{1,0}\neq 0$ and $\sgn(\dg)=\sgn( J\mathcal{D}_0(0,\sqrt{2\e}))$, where $\mathcal{D}_0:\tau\rightarrow \s$ is the diffeomorphism induced by the flow of $Z_0$ and $J\mathcal{D}_0$ denotes the Jacobian of $\mathcal{D}_0$. 
\end{lemma}

Notice that, the change of coordinates exhibited in Lemma \ref{normalform} does not modify the structure of the  problem in the coordinate system $(x,y)$. In fact, the tangency set of $Z_0$ remains fixed through this change of coordinates and it is expressed as $S_{Z_0}=\{(u,0);\ u\in(-\e,\e) \}$, for some $\e>0$ sufficiently small. For the sake of simplicity, we make no distinction between the coordinates $(u,v)$ and $(x,y)$ and so $\mathcal{P}_0$ writes as 
\begin{equation}\label{firstreturn3}
\mathcal{P}_0(x,y)= \left(\begin{array}{c}
\ag x \vspace{0.2cm}\\
b x+ c y^2
\end{array}\right)+x^2 A_1(x,y)+xy^2 A_2(x,y)+y^4 A_3(x,y),
\end{equation}	
where $\ag=\ag_{1,0}$, $b=\bg_{1,0}$, $c=\dfrac{d}{\ag_{1,0}}$, and $A_i$ are bounded vector-valued functions, $i=1,2,3$.
\begin{lemma}\label{genericlemma}
	Let $Z_0\in\Lambda_1$. Consider the full fold line map $\Psi_0$ and the full first return map $\mathcal{P}_0$ of $Z_0$ given by \eqref{foldlinemap} and \eqref{firstreturn3}, respectively. The following statements hold:
	\begin{enumerate}[i)]
		\item $\ag\neq 0$, $|\ag|\neq 1$, $b\neq 0$ and $c\neq 0$;
		\item $\Psi_0(x,0)=(\ag x+\er(x^2),0)$, for $x$ small;
		\item the origin is a hyperbolic fixed point of $\mathcal{P}_0$ with real eigenvalues $0$ and $\ag$;
		\item the eigenspaces of $\mathcal{P}_0$ corresponding to the eigenvalues $0$ and $\ag$ are given by $\mathcal{E}_0=\textrm{span}\{(0,1)\}$ and $\mathcal{E}_{\ag}=\textrm{span}\{(\ag,b)\}$, respectively.
	\end{enumerate} 
\end{lemma}
\begin{proof}
	First, notice that items $(iii)$ and $(iv)$ follows straightly from item $(i)$ and the expression of $\mathcal{P}_0$ given in \eqref{firstreturn3}. Now, we prove items $(i)$ and $(ii)$. Since $Z_0\in\Lambda_1$, it follows from Lemma \ref{fullfrm} that $\ag\neq 0$ and $c\neq 0$.
	
	From \eqref{firstiteration} (with $\ag_{1,0}=\ag$ and $\bg_{1,0}=b$), we deduce that $T_{\textbf{0}}\gamma_0=\textrm{span}\{(1,0)\}$ and $T_{\textbf{0}}\zeta_0=\textrm{span}\{(\ag,b)\}$, where $\gamma_0=S_{Z_0}\cap\widetilde{V_1}$ and $\zeta_0=\mathcal{P}_0(\gamma_0)$. From hypothesis \textbf{(T)}, we have that $\gamma_0\pitchfork\zeta_0$ at the origin. It implies that the vectors $(1,0)$ and $(\ag,b)$ are linearly independent. Hence, $b\neq 0$.
	
	Now, from the computations done in the proof of Lemma \ref{fullfrm}, we derive that 
	$$F_{Z_0}(x,y)= (0,1)+(F_1,F_2),$$
	where $F_1,F_2=\er_1(x,y)$. Denoting $\p_{F_{Z_0}}=(\p_1,\p_2)$, we have that:
	\begin{equation*}
	\left\{
	\begin{array}{l}
	\p_1(t;x,y)= x+ F_1(x,y)t +\er_2(t),\\
	\p_2(t;x,y)=y +(1+F2(x,y))t +\er_2(t),
	\end{array}
	\right.
	\end{equation*}
	for $t,x,y$ small enough.
	
	Now, $\p_2(0;0,0)=0$ and $\partial_t \p_2(0;0,0)=1$. Thus, we use the Implicit Function Theorem to obtain a unique $\Cr$ function $t^*(x,y)$ such that $t^*(0,0)=0$ and $\p_2(t^*(x,y);x,y)=0$, for $(x,y)$ small enough, with $t^*(0,0)=0$. Also, we have that $\partial_x t^*(0,0)=0$ and $\partial_y t^*(0,0)=-1$. Thus,
	$$t^*(x,y)=-y+\er_2(x,y).$$
	
	Notice that, $\gamma_0=\{(x,0);\ x\in(-\e,\e)\}$, for $\e>0$ sufficiently small. Therefore, the full fold line map $\Psi_0:\gamma_0\rightarrow\gamma_0$ writes as
	$$\Psi_{0}(x,0)=(\p_1(t^*(\mathcal{P}_0(x,0));\mathcal{P}_0(x,0)),0).$$
	Hence, it is straightforward to check that
	\begin{equation*}\label{expansionfoldline}
	\Psi_0(x,0)= (\ag x+\er_2(x),0).
	\end{equation*}
	
	Since $Z_0\in\Lambda_1$, we conclude that the full fold line map $\Psi_0$ of $Z_0$ has a hyperbolic fixed point at the origin.  Therefore $|\ag|\neq 1$. 
\end{proof}

%\begin{proof}
%In the coordinate system adopted, we have that $S_{Z_0}$ is a subset of the $x$-axis containing the origin. Now, if we restrict ourselves to $\s=\{z=0\}$, we have that $T_{0}S_{Z_0}=\textrm{span}\{(1,0)\}$	
%	
%$\bg\neq 0$.
%
%Notice that $T_{0}S_{Z_0}=\textrm{span}\{(1,0)\}$ and $T_{0}S^1_{Z_0}=\textrm{span}\{(\ag,\bg)\}$. Now, from condition \textbf{(GH1)}, we have that $(1,0)$ and $(\ag,\bg)$ must be linearly independents. Thus, $\bg\neq 0$.
%\end{proof}

\begin{remark}
	Notice that the curve $\zeta_0$ is tangent to the eigenspace $\mathcal{E}_{\ag}$ at the origin. So, it is an intrinsic degeneracy of this problem which can not be avoided. 
\end{remark}

Using Lemma \ref{genericlemma}, we can apply some near-identity transformations to express the map $\mathcal{P}_0$ given by \eqref{firstreturn3}
in a more accurate normal form.

\begin{prop}\label{firstprop}
	There exists a change of coordinates $\eta: (\rn{2},0)\rightarrow (\rn{2},0)$ such that 
	\begin{equation}\label{firstreturntermos}
	\widetilde{\mathcal{P}_0}(x,y)=\eta\circ\mathcal{P}_0\circ\eta^{-1}(x,y)= \left(\begin{array}{c}
	\ag x-c\ag y^2+c x^2+\er_3(x,y)\vspace{0.2cm}\\
	x
	\end{array}\right).
	\end{equation}	
	In addition, $\widetilde{\mathcal{P}_0}$ is symmetric with respect to an involution $\mathcal{I}$ such that
	$$\textrm{Fix}(\mathcal{I})=\left\{(x,y);\ y=\dfrac{B}{b} x^2+\er_3(x)\right\},$$
	where \begin{equation}\label{B}B=\dfrac{b\partial^2_x \pi_1\circ\mathcal{P}_0(0,0)+ \ag(\ag-1)\partial^2_x \pi_2\circ\mathcal{P}_0(0,0)}{\ag^3(\ag-1)}.\end{equation}
\end{prop}
\begin{proof}
	First, we consider the change of coordinates 
	$$\eta_1(x,y)=\left(\begin{array}{c}
	x-A x^2+\er_3(x,y)\\
	y-Bx^2+ \er_3(x,y)
	\end{array}\right),$$
	such that
	$$\eta_1^{-1}(x,y)=\left(\begin{array}{c}
	x+A x^2\\
	y+Bx^2
	\end{array}\right),$$
	with $B$ given by \eqref{B} and $A=\partial^2_x \pi_1\circ\mathcal{P}_0(0,0)(\ag(\ag-1))^{-1}$. Thus, using that $\mathcal{P}_0$ is given by \eqref{firstreturn3}, we obtain
	$$\eta_1\circ\mathcal{P}_0\circ\eta_1^{-1}(x,y)=\left(\begin{array}{c}
	\ag x +G_1(x,y)\\
	bx+c y^2+G_2(x,y)
	\end{array}
	\right),$$
	where $G_1(x,y)=\er_3(x,y)$ and  $G_2(x,y)=\er_3(x,y)$, and $\eta_1\circ\mathcal{P}_0\circ\eta_1^{-1}$ is symmetric with respect to the symmetry $\mathcal{I}_1(x,y)=(x,-y-2B x^2)$, which has the following set of fixed points
	$$\textrm{Fix}(\mathcal{I}_1)=\left\{(x,y);\ y=B x^2\right\}.$$
	
	Now, considering the change of coordinates
	$$\eta_2(x,y)=\left(\begin{array}{c}
	bx +G_2(x,y)+cy^2\\
	y
	\end{array}\right),$$
	and taking $\eta=\eta_2\circ\eta_1$, the proof follows directly.
\end{proof}	

\begin{remark}
	Notice that the change of coordinates $\eta$ provided by Proposition \ref{firstreturn} carries the fold line $S_{Z_0}$ of $Z_0$ onto the set $\textrm{Fix}(\mathcal{I})$.
\end{remark}

The next result follows straightly from Lemma \ref{genericlemma} and the Stable Manifold Theorem for $\Cr$ maps (see Theorem $10.1$ in \cite{R}).

\begin{prop} \label{variedades}
	Let $Z_0\in\Lambda_1$, and consider the non-invertible full first return map $\mathcal{P}_0$ of $Z_0$ given by \eqref{firstreturn3}. Therefore, $\mathcal{P}_0$ has a local stable invariant manifold $W^s_{0}$ at the origin tangent to $\mathcal{E}_0$ and either one of the following statements hold.
	\begin{enumerate}
		\item If $|\ag|<1$, then $\mathcal{P}_0$ has a fixed point of nodal type at the origin and it has a local stable invariant manifold $W^s_{\ag}$ at the origin tangent to $\mathcal{E}_{\ag}$ (see Figure \ref{first_return_fig} - $(ii)$ and $(iv)$). 
		\item If $|\ag|>1$, then $\mathcal{P}_0$ has a fixed point of saddle type at the origin and it has a local unstable invariant manifold $W^u_{\ag}$ at the origin tangent to $\mathcal{E}_{\ag}$ (see Figure \ref{first_return_fig} - $ (i)$ and $(iii)$). 		
	\end{enumerate} 
\end{prop}

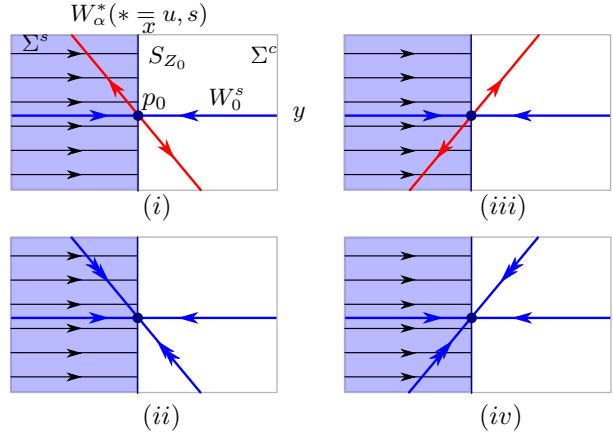
\begin{figure}[h!]
	\centering
	\bigskip
	\begin{overpic}[width=8cm]{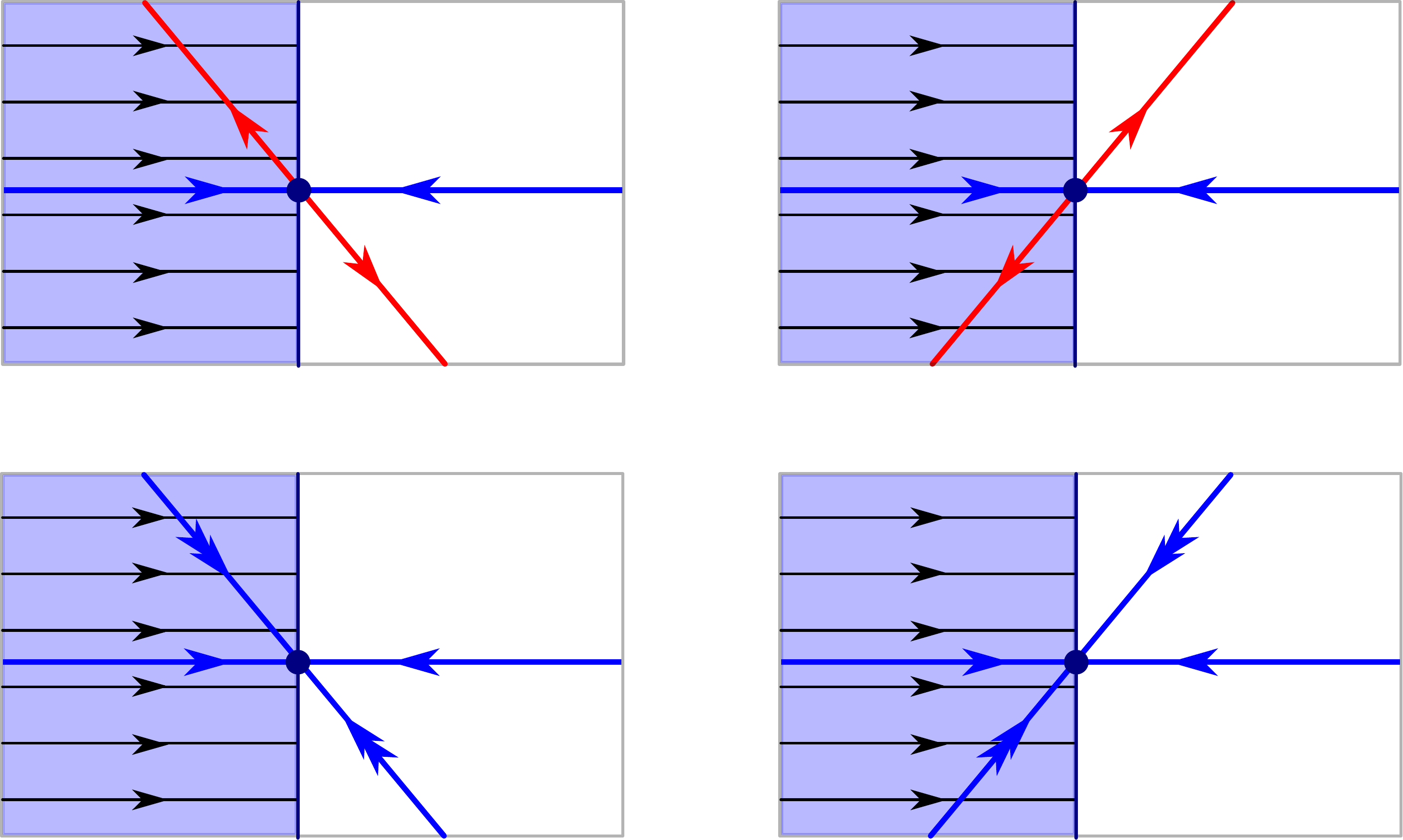}
		%									\begin{overpic}[grid,tics=5,width=8cm]{Figures/firstreturn.pdf}			
		\put(2,57){{\scriptsize $\s^s$}}		
		\put(23,55){{\scriptsize $S_{Z_0}$}}		
		\put(33,48){{\scriptsize $W_0^s$}}	
		\put(22,60){{\scriptsize $x$}}	
		\put(47,46){{\scriptsize $y$}}			
		\put(10,62){{\scriptsize $W_\ag^*(*=u,s)$}}			
		\put(22,48){{\footnotesize $p_0$}}	
		\put(40,55){{\scriptsize $\s^c$}}									
		\put(22,30){{\footnotesize $(i)$}}		
		\put(22,-5){{\footnotesize $(ii)$}}	
		\put(78,30){{\footnotesize $(iii)$}}		
		\put(78,-5){{\footnotesize $(iv)$}}					
	\end{overpic}
	\bigskip
	\caption{Configurations of the local invariant manifolds of $\mathcal{P}_0$ at $p_0$ for $b<0$, $(i)$ $\ag>1$, $(ii)$ $0<\ag<1$, $(iii)$ $\ag<-1$ and $(iv)$ $-1<\ag<0$. If $b>0$, then $(i)$ is switched by $(iii)$ such as $(ii)$ and $(iv)$.}	\label{first_return_fig}
\end{figure} 	

Finally, we characterize the classes $\Lambda_1^C$ and $\Lambda_1^M$ of $\Lambda_1$ and the hypotheses $(N)$ and $(S)$ introduced in Section \ref{main_sec}, which generate four possible types of quasi-generic loop $\Gamma_0$ passing through a fold-regular singularity $p_0$ of $Z_0\in\Lambda_1$.

\begin{prop}\label{classes}
	Let $Z_0\in\Lambda_1$, and consider the full fold line map $\Psi_0$ of $Z_0$ given by \eqref{foldlinemap}. The following statements hold:
	\begin{enumerate}[i)]
		\item $Z_0\in \Lambda_1^C$ and satisfies $(S)$ if, and only if, $\ag>1$; 
		\item $Z_0\in \Lambda_1^C$ and satisfies $(N)$ if, and only if, $0<\ag<1$;
		\item $Z_0\in \Lambda_1^M$ and satisfies $(S)$ if, and only if, $\ag<-1$;
		\item $Z_0\in \Lambda_1^M$ and satisfies $(N)$ if, and only if, $-1<\ag<0$.					
	\end{enumerate}
\end{prop}
\begin{proof}
	From Lemma \ref{genericlemma}, the full fold line map $\Psi_0$ of $Z_0$ writes as $\Psi_0(x,0)=(\ag x+\er(x^2),0)$. In this case, the map $\Psi_0$ preserves the connected components $(-\e,0)\times\{0\}$ and $(0,\e)\times\{0\}$ of $S_{Z_0}$ if, and only if, $\ag>0$. The result follows from Proposition \ref{variedades}.
\end{proof}	

\begin{remark}
	Notice that, the geometry of this problem allows us to see that the first return map $P_0$ preserves the orientation of the $y$-axis, nevertheless the orientation of the $x$-axis is reversed if $\ag<0$, and it is preserved if $\ag>0$. Therefore, $P_0$ preserves orientation if, and only if, $Z_0\in\Lambda_1^C$. 
	
	Since the transition map $T_0$ does not provide any changes in the orientation of $\overline{\s^c}$, it follows that $P_0=\mathcal{D}_0\circ T_0$ preserves orientation if, and only if, $\mathcal{D}_0$ preserves orientation. Hence, if $Z_0\in\Lambda_1^C$, it follows from \eqref{firstreturn}, \eqref{firstreturn3}, and Proposition \ref{classes} that $c>0$.
\end{remark}
%	\textcolor{red}{COLOCAR DESENHO DA ORIENTACAO}

%{\color{red}\begin{remark}
%		Notice that, the geometry of the problem allows us to see that the first return map $P_0$ preserves the orientation of the $y$-axis, nevertheless the orientation of the $x$-axis is reversed if $\ag<0$ and it is preserved if $\ag>0$. Therefore, we can see that $P_0$ preserves orientation if, and only if $Z_0\in\lambda_1^C$. It follows from the construction of $P_0$ that the constant $c$ in \eqref{firstreturn3} satisfies $c>0$.
%\end{remark}}

As mentioned in Section \ref{main_sec}, if $Z_0\in\Lambda_1^C$, then the fold line map $\psi_0$ defines a dynamics on $\sigma_{Z_0}^{FL}$ (which is an open interval of the $x$-axis) induced by the orbits of $Z_0$.

Now, let $Z_0\in\Lambda_1^M$, and without loss of generality, assume that $b<0$ in \eqref{firstreturn3}. From the proof of Lemma \ref{genericlemma}, we have that the map $\psi_0^*:V_0\rightarrow\gamma_0=S_{Z_0}\cap \widetilde{V_1}$ induced by the flow of $F_{Z_0}$ is given by $\psi_0^*(x,y)=(x+\er_2(x,y),0)$, and notice that, in this coordinate system, $\s^s=\widetilde{V_1}\cap\{y<0\}$ and $S_{Z_0}=\{y=0\}$. From Proposition \ref{classes}, we have that $\ag<0$ and thus $\mathcal{P}_0(x,0)=(\ag x,b x)+\er(x^2)\in\s^s$ if, and only if $x\geq0$.

Given $x>0$ small, we have that $\Psi_0(x,0)=\psi_0^*\circ\mathcal{P}_0(x,0)$ and $(x,0)$ are connected by orbits of $Z_0$. Now, 
$$\mathcal{P}_0(\Psi_0(x,0))=(\ag^2x, b\ag x)+\er(x^2)$$
does not belong to $\s^s$, since $b\ag x>0$, and therefore the points $\Psi_0^2(x,0)$ and $(x,0)$ are not connected by orbits of $Z_0$. That means that the iterations of $\Psi_0$ do not describe the dynamics of $Z_0$. In other words, the fold line map $\psi_0$ (which is the restriction of $\Psi_0$ to $x\geq 0$) does not induce any dynamics in the interval $\widetilde{V_1}\cap\{x\geq 0, y=0\}$. 

Nevertheless, given $x<0$, we have that
$$\mathcal{P}_0^2((x,0))=(\ag^2x, b\ag x)+\er(x^2)\in \s^s,$$
and hence $\psi_0^*\circ \mathcal{P}_0^2(x,0)$ and $(x,0)$ are connected by orbits of $Z_0$ (with a unique segment of sliding orbit). Therefore, we define the \textbf{full Möbius fold line map} $\Psi_0^{M}:\gamma_0\rightarrow\gamma_0$ of $Z_0$ as
\begin{equation*}
\Psi_0^M(x,0)=\psi_0^*\circ\mathcal{P}_0^2(x,0)=(\ag^2x+\er(x^2),0),
\end{equation*}
and the domain
\begin{equation*}
\sigma_{Z_0}^{M}= \mathcal{P}_0^{-1}(\mathcal{P}_0(S_0\cap\widetilde{V_1})\cap \s^{c})=\{x\leq 0\}\times\{0\}\cap \widetilde{V_1}.
\end{equation*} 
Accordingly as above, we define the \textbf{Möbius fold line map} of $Z_0$ as $\psi_0^{M}=\Psi_0^{M}\big|_{\sigma_{Z_0}^{M}}$. We conclude that $\psi_0^{M}(\sigma_{Z_0}^{M})\subset \sigma_{Z_0}^{M}$ and thus, this map defines a dynamical system in $\sigma_{Z_0}^{M}$ induced by the (real) orbits of $Z_0$, whether $Z_0\in\Lambda_1^M$. 

\begin{figure}[H]
	\centering
	\bigskip
	\begin{overpic}[width=7cm]{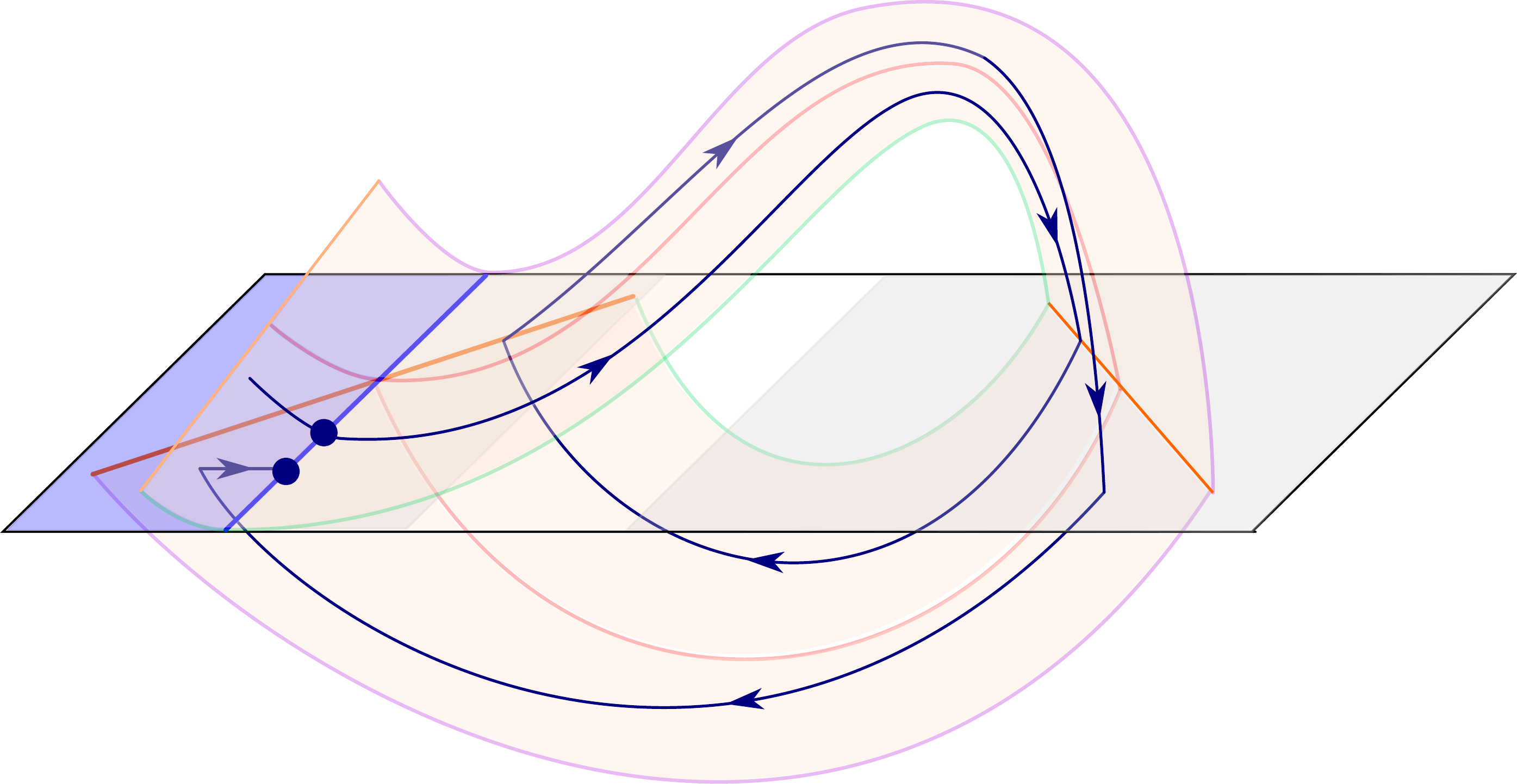}
		%							\begin{overpic}[grid,tics=5,width=13cm]{Figures/cycles.pdf}			
		%		\put(25,-3){{\footnotesize $(a)$}}		
		%		\put(75,-3){{\footnotesize $(b)$}}		
		%		\put(16.2,10.5){{\scriptsize $p_0$}}	
		%		\put(38,10.5){{\scriptsize $q_0$}}			
		%		\put(0,10){{\footnotesize $\s$}}	
		%		%\put(57,10){{\footnotesize $\s$}}	
		%		\put(26,22){{\footnotesize $\Gamma_0$}}			
		%		\put(10,20){{\footnotesize $X_0$}}	
		%		\put(10,2){{\footnotesize $Y_0$}}							
		
	\end{overpic}
	\bigskip
	\caption{Action of the Möbius fold line map $\psi_0^M$ of $Z_0\in\Lambda_1^M$.}	\label{mob_fig}
\end{figure}

\begin{remark}
	Notice that, if $Z_0\in\Lambda_1^M$, we can still define the Möbius fold line map $\psi_Z^M$ for every $Z$ sufficiently near $Z_0$, combining the ideas above with Section \ref{trans_sec}. Also, the origin is a hyperbolic fixed point of $\psi_0^M$ if, and only if, it is a hyperbolic fixed point of the fold line map $\psi_0$ of $Z_0$.
\end{remark}

\section{Proofs of Theorems \ref{cod1thm}, \ref{bifthm}, \ref{stabthm} and \ref{modthm}}\label{proofsqg}

In this section, we use the maps constructed in Section \ref{struct_sec} to prove Theorems  \ref{cod1thm}, \ref{bifthm}, \ref{stabthm} and \ref{modthm}.
%In this section, we prove Theorems A and B. More specifically, we show that quasi-generic homoclinic connections at a fold-regular singularity occur generically in one-parameter families in $\Or$, and thus it has codimension one in $\Or$.

\subsection{Proof of Theorem \ref{cod1thm}}
From Section \ref{first_sec}, there exist neighborhoods $\V_0$ of $Z_0$ in $\Or$ and $V_0$ of $p_0$ in $\s$ sufficiently small, such that, each $Z\in\V_0$ is associated to a full first return map $\mathcal{P}_Z: V_0\rightarrow \s$.%, where  $V_0$ is a neighborhood of $p_0$, such that $p_0\in C_0\subset V_Z$ and $C_0$ is a fixed compact set of $\s$.

Let $\mathcal{A}_0$ be a solid torus containing $\Gamma_0$ such that $V_0=\mathcal{C}_{p_0}(\mathcal{A}_0\cap\s)$ (connected component of $\mathcal{A}_0\cap\s$ containing $p_0$). In addition, for each $Z\in\V_0$, there exist coordinates $(x,y,z)$ (which has a $\Cr$-dependence on $Z$) defined in $V_0$, such that $\s$ is given by the plane $z=0$ and $S_Z\cap V_0$ is given by the $x$-axis.

Since $\mathcal{P}_{Z_0}$ has a unique hyperbolic fixed point $p_0$ in $V_0$, it follows from the Implicit Function Theorem that $\mathcal{P}_{Z}$ has a unique hyperbolic fixed point $p_Z$ in $V_0$, for each $Z\in\V_0$ (reduce $\V_0$ if necessary). Denoting the $y$-coordinate of $p_Z$ in the coordinate system $(x,y,z)$ by $p_Z^y$, it follows that $p_Z\in S_Z$ if, and only if $p_Z^y=0$.

Define $\zeta(Z)=p^y_Z$, for each $Z\in \V_0$. Therefore, it is straightforward to see that $\zeta(Z)=0$ if, and only if, $Z$ has a  homoclinic-like loop at $p_Z$ contained in $\mathcal{A}_0$. Also, it is not difficult to see that conditions $(G)$, $(T)$, $(ii)$ and $(iii)$ of Definition \ref{lambda1} hold for every $Z\in\V_0$, which means that $\zeta(Z)=0$ if, and only if, $Z$ has a quasi-generic loop at $p_Z$ contained in $\mathcal{A}_0$.

Now, let $Z^*=(X^*,Y^*)\in\V_0$ such that $\zeta(Z^*)=0$, and let $\mathcal{Z}_{\lambda}$ be a curve in $\Or$ such that $\mathcal{Z}_0=Z^*$, and $\mathcal{Z}_\lambda=(X^*, Y_{\lambda})$. In this case, 
$$\mathcal{P}_{\mathcal{Z}_0}(x,y)=(\ag x,b x)+\er_2(x,y),$$
for some $\ag\neq 0,\pm1$, and $b\neq 0$. Given $v\in\R$, we can take $Y_{\lambda}$ such that
$$\mathcal{P}_{\mathcal{Z}_\lambda}(x,y)=(0,-\lambda v)+(\ag x,b x)+\er_2(x,y,\lambda).$$

Again, applying the Implicit Function Theorem, we can see that $\zeta(\mathcal{Z}_{\cg})=v\cg+\er_2(\cg)$, hence  
$$\left.\dfrac{d}{d\lambda}\zeta(\mathcal{Z}_{\lambda})\right|_{\lambda=0}=v.$$

We conclude that $0$ is a regular value of $\zeta$. The result follows by noticing that $\Lambda_{1}\cap\mathcal{V}_0=\zeta^{-1}(0)$.

\subsection{Proof of Theorem \ref{bifthm}}

Let $\mathcal{Z}:(-\e,\e)\rightarrow \Or$ be a one-parameter $\mathcal{C}^1$ family such that $\mathcal{Z}(0)=Z_0$, which is transverse to $\Lambda_1$.

From Section \ref{first_sec}, there exist $\e>0$ sufficiently small and a neighborhood $V_0$ of $p_0$ in $\s$ sufficiently small, such that, each $\mathcal{Z}(\cg)$ is associated to a full first return map $\mathcal{P}_{\mathcal{Z}(\cg)}: V_0\rightarrow \s$. Let $\phi_{\cg}: V_0 \rightarrow \rn{3}$ be a change of coordinates (which has a $\Cr$ dependence on $\cg$) such that
\begin{itemize}
	\item $\s$ is brought into the plane $z=0$;
	\item The fold line $S_{\cg}$ of $\mathcal{Z}(\gamma)$ in $V_0$ is brought into the $x$-axis;
	\item If we denote $S_{\cg}^1=\mathcal{P}_{\mathcal{Z}(\cg)}(S_{\cg})$, then the point $S_{\cg}\cap S_{\cg}^1$ is carried into $(0,0,0)$. 
\end{itemize}

Consider the family $\overline{\mathcal{Z}}(\cg)= d\phi_{\cg}\circ\mathcal{Z}(\cg)\circ\phi_{\cg}^{-1}$ and notice that the families $\mathcal{Z}$ and $\overline{\mathcal{Z}}$ are equivalents. Since $\mathcal{Z}$ is transverse to $\Lambda_1$ at $0$, it follows that the same holds for $\overline{\mathcal{Z}}$.

Thus, the first return map $P_{\cg}=P_{\overline{\mathcal{Z}}(\cg)}$ (see Section \ref{first_sec}) is defined for $y\geq0$, and its extension $\mathcal{P}_{\cg}$ has a fixed point $$p_{\cg}=(\er(\cg), a\cg+\er(\cg^2), 0),$$
with $a\neq 0$. For instance, assume that $a>0$.

It means that $p_{\cg}\in\s^c$ if, and only if $\cg>0$, and thus $\overline{\mathcal{Z}}(\cg)$ has a unique hyperbolic crossing limit cycle in $\mathcal{A}_0$, if, and only if $\cg>0$.

Now, recall that the full fold line map $\psi_{\cg}=\psi_{\overline{\mathcal{Z}}(\cg)}$ defined in the fold line $y=0$ introduced in Section \ref{foldline_sec} controls the existence of sliding cycles. More specifically, it associates sliding cycles with fixed points of $\psi_{\cg}$ belonging to a certain domain $\sigma_{\cg}^{FL}=\sigma_{\overline{\mathcal{Z}}(\cg)}^{FL}$ which is given by \eqref{folddom}.

%Without loss of generality, assume that $S_{\cg}^1$ is tangent to the line $y=kx$, with $k<0$. Thus, the domain $D_{\cg}$ is given by
%$$D_{\cg}=\{(x,0,0); x>0\},\ \forall \cg\in(-\e,\e).$$

Since the origin is a hyperbolic fixed point of $\psi_{0}(x)$, it follows that $\psi_{\cg}(x)$ has a unique hyperbolic fixed point $x_{\cg}$. Hence $\overline{\mathcal{Z}}(\cg)$ has at most a unique sliding cycle in $\mathcal{A}_0$.

Now, we must check whether $x_{\cg}$ belongs to $\sigma_{\cg}^{FL}$. If $\gamma>0$, then $p_{\cg}\in\s^c$ and thus their invariant manifolds $W^1=W_0^s$ and $W^2=W_\ag^s$ (given by Theorem \ref{variedades}) intersect the $x$-axis in the points $x_1^+$ and $x_2^+$, respectively. Also, if $\gamma<0$, then $p_{\cg}\in\s^s$ and $W^1$ and $W^2$ intersect the $x$-axis in the points $x_1^-$ and $x_2^-$, respectively. Without loss of generality, assume that $S_{\cg}^1$ is tangent to the line $y=kx$ at the origin, with $k<0$. It follows that $x_1^+<x_2^+$ and $x_1^->x_{2}^-$. Now, assume that $\mathcal{Z}(0)$ satisfies $(N)$. 

In the case $\cg>0$, the point $p_{\cg}$ is in $\s^c$ and it is attractive. Using that $\mathcal{P}_{\cg}(x_{i}^{+},0,0)$ must stay in $W^i$ and goes to $p_{\cg}$, it follows that $\mathcal{P}_{\cg}(x_{i}^{+},0,0)\in\s^c$, $i=1,2$, which means that if $x\leq x_{2}^+$ then $\mathcal{P}_{\cg}(x,0,0)$ belongs to $\s^c$ and thus all these points do not belong to the domain $\sigma_{\cg}^{FL}$ (recall that $k<0$). Nevertheless, we know that the $x$-axis has a unique attractive fixed point $x_{\cg}$ of $\psi_{\cg}$ and $\pi_1\circ\mathcal{P}_{\cg}(x_2^+,0,0)<x_2^+$ leads us to $\psi_{\cg}(x_2^+)<x_2^+$, which means that $x_{\cg}<x_{2}^+$, and thus $x_{\cg}\notin \sigma_{\cg}^{FL}$. We conclude that, if $\cg>0$, then $\overline{\mathcal{Z}}(\cg)$ has no sliding cycles. 

Now, if $\cg<0$, then $p_{\cg}\in\s^s$ and through similar arguments, it follows that $\mathcal{P}_{\cg}(x_i^-,0,0)\in\s^s$, $i=1,2$, and thus, if $x\geq x_2^-$ then $\mathcal{P}_{\cg}(x,0,0)$ belongs to $\s^s$ and thus all these points belong to the domain $\sigma_{\cg}^{FL}$. In this case, $\pi_1\circ\mathcal{P}_{\cg}(x_2^-,0,0)<x_2^-$ and thus $\psi_{\cg}(x_2^-)>x_2^-$, which means that $x_{\cg}> x_{2}^-$ and hence $x_{\cg}\in \sigma_{\cg}^{FL}$. We conclude that,  if $\cg<0$, then $\overline{\mathcal{Z}}(\cg)$ has a unique sliding cycle. See Figure \ref{nodebif_fig}.

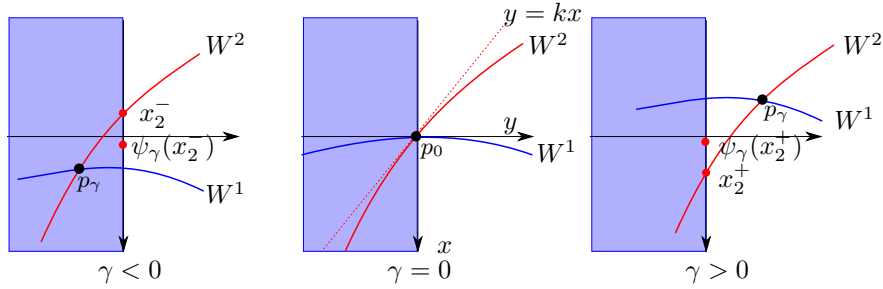
\begin{figure}[h!]
	\centering
	\begin{overpic}[width=11cm]{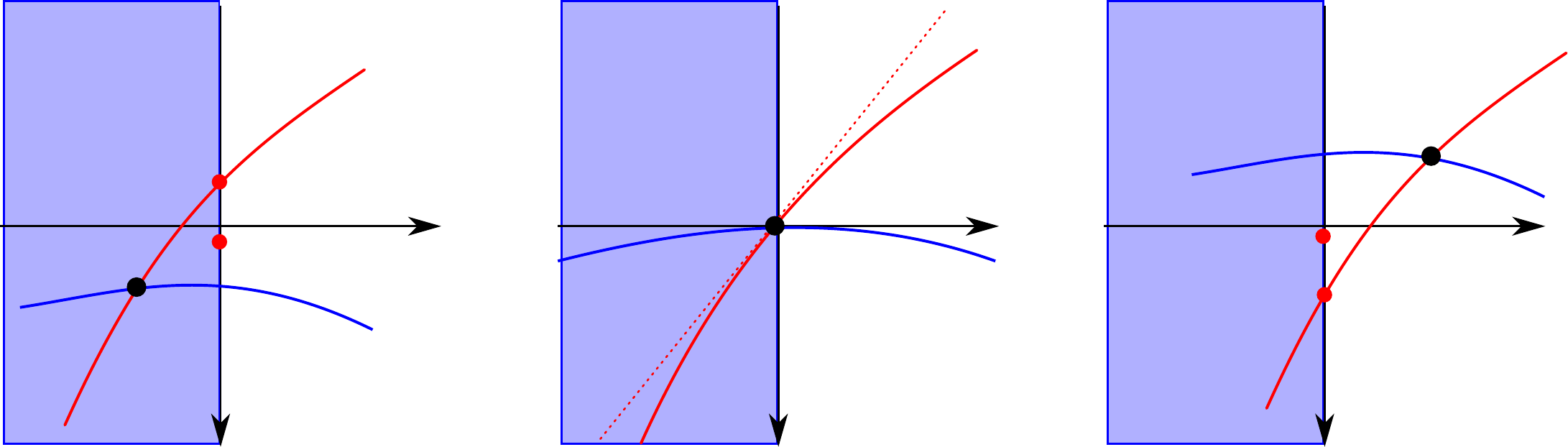}
		%	\begin{overpic}[grid,tics=5,width=11cm]{Figures/nodebif.pdf}	
		\put(8.5,8){{\scriptsize $p_{\cg}$}}		
		\put(50,12){{\scriptsize $p_{0}$}}
		\put(91.5,16.5){{\scriptsize $p_{\cg}$}}											
		\put(46,-3){{\footnotesize $\cg=0$}}																
		\put(11,-3){{\footnotesize $\cg<0$}}	
		\put(82,-3){{\footnotesize $\cg>0$}}	
		\put(24,24){{\footnotesize $W^2$}}			
		\put(63,24){{\footnotesize $W^2$}}		
		\put(101,24){{\footnotesize $W^2$}}			
		\put(24,6){{\footnotesize $W^1$}}					
		\put(64,11){{\footnotesize $W^1$}}		
		\put(100,15){{\footnotesize $W^1$}}
		\put(86,12){{\footnotesize $\psi_{\cg}(x_2^+)$}}	
		\put(86,8){{\footnotesize $x_2^+$}}	
		\put(15,12){{\footnotesize $\psi_{\cg}(x_2^-)$}}	
		\put(16,16){{\footnotesize $x_2^-$}}		
		\put(60,15){{\footnotesize $y$}}	
		\put(52,0){{\footnotesize $x$}}	
		\put(60,28){{\footnotesize $y=kx$}}																														
	\end{overpic}
	\bigskip
	\caption{Position of the invariant manifolds $W^1$ and $W^2$ of $\mathcal{P}_{\cg}$ in the case $(N)$ as $\cg$ varies.}	\label{nodebif_fig}
\end{figure}

% $\s$ into the plane $z=0$ and the fold line of $\mathcal{Z}(\gamma)$ into the $x$-axis. Thus, $\mathcal{Z}(\gamma)$ is equivalent to the family of vector fields $\overline{\mathcal{Z}}(\gamma)=\mathcal{Z}(\gamma)\circ \phi_{\cg}^{-1}$ such that in the system of coordinates $(x,y,z)$, the switching manifold $\s$ is brought to $z=0$ and the fold line of $dd$  
%\subsection{Bifurcation of Crossing and Sliding Cycles}

\subsection{Proof of Theorem \ref{stabthm}}

From the construction of the full first return map $\mathcal{P}_0$ of $Z_0$ in Section \ref{first_sec}, it follows that, to prove Theorem \ref{stabthm}, it is enough to compute the basin of attraction of the origin of the map $\mathcal{P}_0$ and to analyze the sliding dynamics of $Z_0$. 

If $Z_0$ satisfies $(N)$, then the origin is a hyperbolic fixed point of $\mathcal{P}_0$ of nodal type and thus, there exists a neighborhood of the origin which is the basin of attraction of $\mathcal{P}_0$ at $(0,0)$. Since all the sliding orbits of $Z_0$ near the origin reaches the fold line $S_{Z_0}$ of $Z_0$ and the origin is an attractive hyperbolic fixed point of the fold line map $\psi_0$, it follows that every orbit of $Z_0$ near the origin goes to the origin. Statement $(i)$ of Theorem \ref{stabthm} follows directly.

Now, if $Z_0$ satisfies $(S)$, then the origin is a hyperbolic fixed point of $\mathcal{P}_0$ of saddle type, and thus the basin of attraction of 
$\mathcal{P}_0$ at $(0,0)$ is given by the stable invariant manifold $W^u_{\ag}$. Hence all the orbits of $Z_0$ passing through $W^u_{\ag}\cap \s^c$ goes to the origin.

Also, there exists a unique sliding orbit $\cg_s$ of the sliding vector field $F_{Z_0}$ which goes to the origin. Since the origin is a repelling hyperbolic fixed point of the fold line map $\psi_0$, it follows that an orbit $\Gamma$ of $Z_0$ goes to the origin if and only $\Gamma\cap\s$ contains a point of the piecewise-smooth curve $\bg=W^u_{\ag}\cap \s^c\cup \cg_s$. The proof of Theorem \ref{stabthm} follows directly.

\subsection{Proof of Theorem \ref{modthm}}

In order to prove Theorem \ref{modthm}, we study the behavior of the iterations of the fold line $S_0$ of $Z_0\in\Lambda_1^C$ through its full first return map $\mathcal{P}_0$.

\begin{lemma}[Accumulation] \label{acc_lemma}
	Let $Z_0=(X_0,Y_0)\in\Lambda_1^C$ having a quasi-generic loop $\Gamma_0$ at $p_0$ and let $\mathcal{P}_0$ be the full first return map associated to $Z_0$ given by \eqref{firstreturn3}. If $Z_0$ satisfies $(S)$, then $S_n=\mathcal{P}_0^n(S_{X_0})$, $n\in\N$, is a sequence of smooth curves tangent to the eigenspace $\mathcal{E}_{\ag}$ given by Proposition \ref{variedades} at $p_0$, such that, for each $\e>0$ sufficiently small, there exists $N_0\in\N$ such that $S_n$ is $\e$- close to the unstable invariant manifold $W^u_\ag$ of $\mathcal{P}_0$ at $p_0$, for every $n\geq N_0$. Furthermore, for each $n\geq 2$, $S_n$ is a curve having an even contact with $S_{n-1}$ at $p_0$ and the following statements hold
	\begin{enumerate}[i)]
		\item In $\s^c$, $S_{n-1}$ and $S_n$ are given by arcs clockwise ordered, and thus $S_n\cap \s^c$ and $W^u_{\ag}$ are clockwise ordered for every $n\in\N$;
		\item In $\s^s$, $S_n$ is flipped back to the region delimited by $S_{n-1}$ and $S_{n-2}$, for every $n\geq 2$. Thus, $S_n$ alternates the side of $W^u_{\ag}$ (flip property), which means that, if $W^u_{\ag}$ and $S_{n-1}$ are counterclockwise ordered, then $S_n$ and $W^u_{\ag}$ are counterclockwise ordered, and vice-versa.
	\end{enumerate}
	
\end{lemma}
\begin{proof}
	From Proposition \ref{classes}, we have that the parameters in  \eqref{firstreturn3} satisfy $\ag>0$ and $c>0$. Recall that the coordinate system $(x,y,z)$ at $p_0$ used to express $\mathcal{P}_0$ as  \eqref{firstreturn3} satisfies the following properties.
	%
	%%\subsubsection{A Fundamental Domain for $\mathcal{P}_0$}
	%In what follows, given $Z_0\in\Lambda_1^C$, we construct a fundamental domain for the full first return map $\mathcal{P}_0$ given by \eqref{firstreturn3} (with $\ag>0,\ag\neq 1$ and $c>0$). We present a geometrical description of the domain.% and secondly we develop it through analytical arguments.
	%
	%Since $Z_0\in\Lambda_1^C$, there exist coordinates $(x,y,z)$ at $p_0$ such that
	\begin{enumerate}
		\item $\s$ is expressed by $\{z=0\}$;
		\item The fold line $S_{Z_0}=S_0$ of $Z_0$ is given by the $x$-axis;
		\item Without loss of generality, we assume that $b<0$ in \eqref{firstreturn3}, thus the curve $S_1=\mathcal{P}_0(S_0)$ of $S_0$ is tangent to the line $y=kx$ at the origin, where $k=b/\ag<0$.
	\end{enumerate}
	
Such a configuration in the switching manifold ($z=0$) is illustrated in Figure \ref{position_fig}.
	
	\begin{figure}[h!]
		\centering
		\bigskip
		\begin{overpic}[width=5cm]{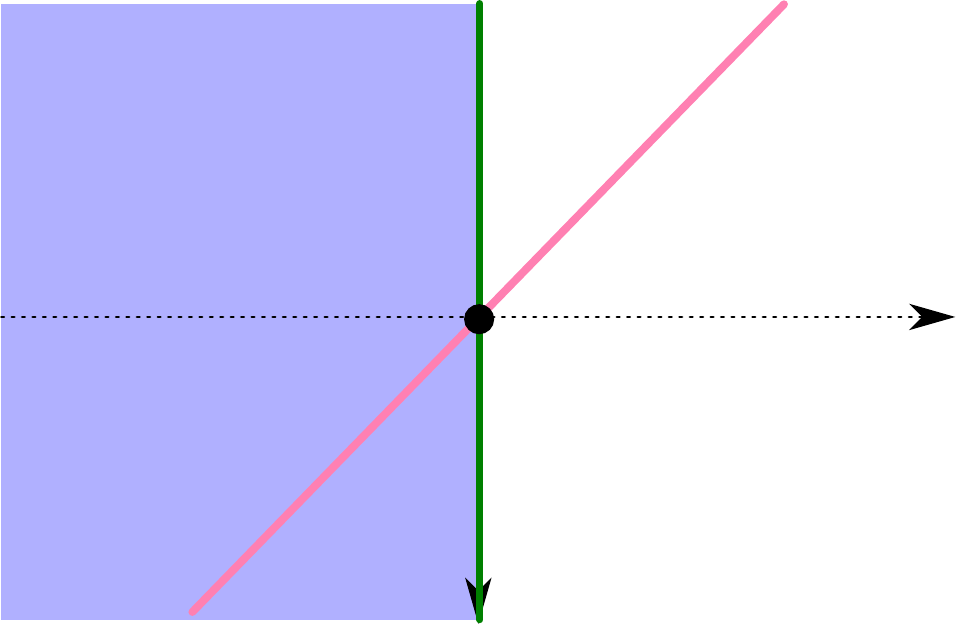}
			%										\begin{overpic}[grid,tics=5,width=6cm]{Figures/config.pdf}								
			\put(52,60){{\footnotesize $S_0$}}																
			\put(85,60){{\footnotesize $S_1$}}	
			\put(5,58){{\footnotesize $\s^s$}}	
			\put(52,2){{\footnotesize $x$}}	
			\put(95,27){{\footnotesize $y$}}						
		\end{overpic}
		\bigskip
		\caption{Configuration of $S_0$ and $S_1$ in $\s$.}	\label{position_fig}
	\end{figure}
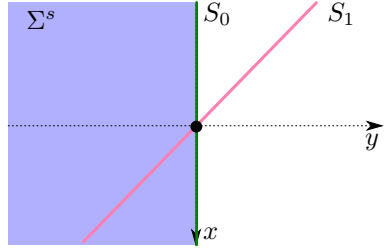

	Since $Z_0$ satisfies $(S)$, it follows from Proposition \ref{variedades} that the map $\mathcal
	P_0$ has a fixed point of saddle type at the origin which has a stable invariant manifold $W^s_{\ag}$ tangent to the $y$-axis and a unstable invariant manifold $W^u_{\ag}$ tangent to the line $y=kx$. 
	
	%In what follows, we show that the first and third quadrant delimited by $S_0$ and $S_1$ are fundamental domains of $\mathcal{P}_0$.
	In what follows, we describe how the iterations of $S_0$ through $\mathcal P_0$ behave. From the expression of $\mathcal
	P_0$ in \eqref{firstreturn3}, we have that $S_n=\mathcal P_0^n(S_0)$ is a smooth curve passing through $(0,0)$ tangent to the line $y=kx$ at $(0,0)$, for each $n\in\N$. Clearly, $W^s_{\ag}\cap \mathcal{P}_0^n(S_0)=\emptyset$, for each $n\in \N$, since $S_0$ and $S_1$ are transversal.
	
	Now, in order to obtain the positions of the curves $S_n$ in $\s$, we must recall the construction of the map $\mathcal P_0$. In Section \ref{first_sec}, $\mathcal
	P_0$ is written as the composition $\mathcal P_0=\mathcal D_0\circ\mathcal T_0$, where $\mathcal T_0$ is a transition map from $\s$ to a transversal section $\tau=\{z=\e\}$, for $\e>0$ small, and $\mathcal D_0$ is an orientation-preserving diffeomorphism from $\tau$ to $\s$. In addition, notice that
	$$\mathcal T_0(x,y)=(x,\sqrt{2\e} + Ky^2+\er(y^3)),\textrm{ for some }K>0.$$
	
	Without loss of generality, consider that $S_1$ is the line $y=kx$. Now, we describe how to obtain $S_n$, for $n\geq 2$. 
	
	We consider $n=2$, since the other cases follow completely analogous. Notice that $\mathcal T_0(x,kx)=(x,\sqrt{2\e}+K_2x^2+\er(x^3))$, where $K_2=Kk^2>0$, describes a parabola tangent to the origin contained in the semi-plane $y\geq \sqrt{2\e}$ of section $\tau$. Since the line $y=\sqrt{\e}$ is sent to the line $S_1$ in $\s$ through the diffeomorphism $\mathcal D_0$ (which preserves the orientation of the section $\tau$), it follows that $\mathcal P_0(S_1)$ is a parabola which has a quadratic contact with $S_1$ at the origin. In addition, $\mathcal P_0(S_1)\cap \s^c$ is contained in the first quadrant delimited by $S_0$ and $S_1$ and $\mathcal P_0(S_1)\cap \s^s$ is contained in the fourth quadrant generated by $S_0$ and $S_1$ (see Figure \ref{sn_fig}). 
	
	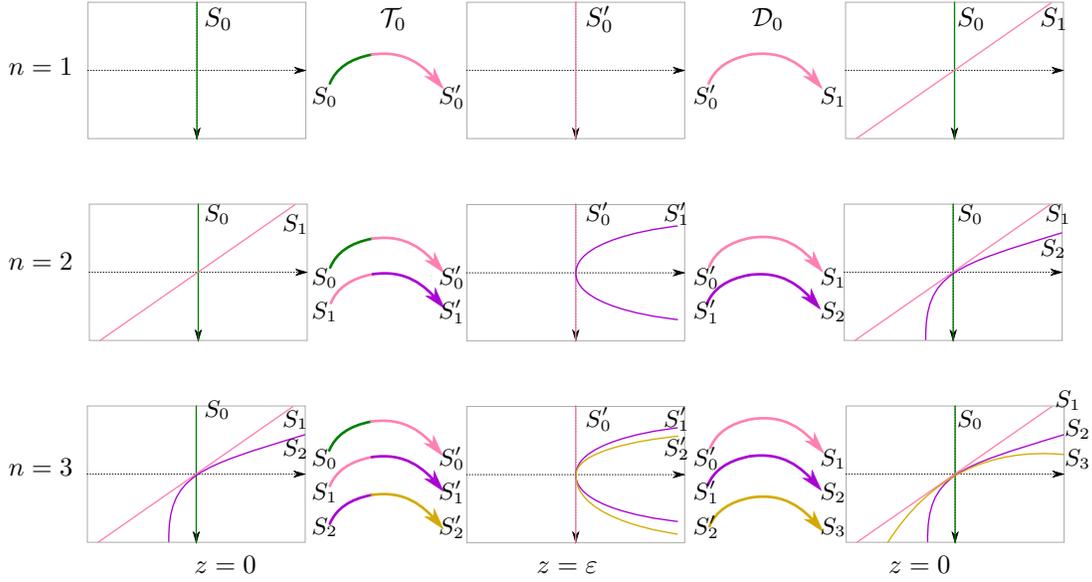
\begin{figure}[h!]
		\centering
		\bigskip
		\begin{overpic}[width=13cm]{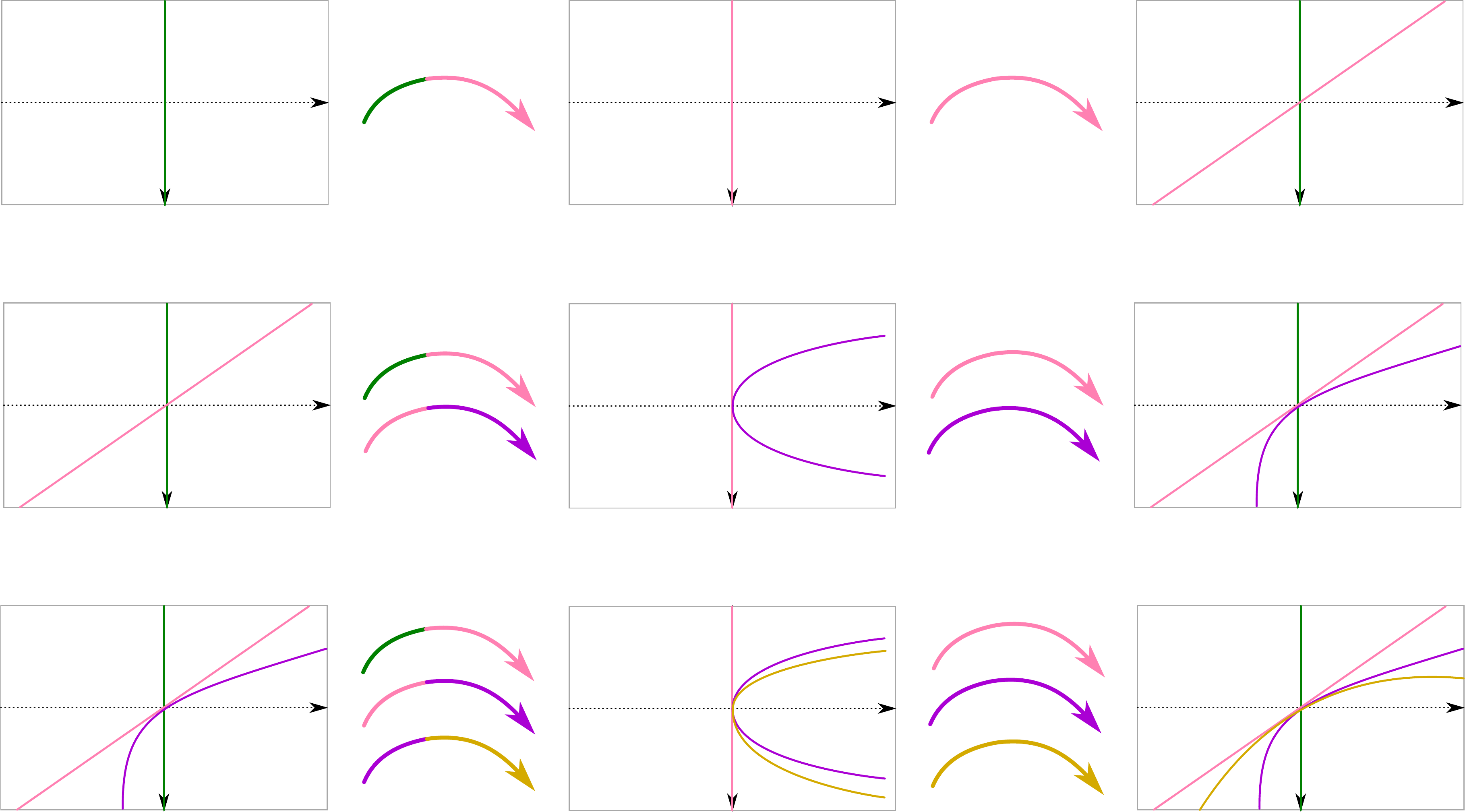}
			%					\begin{overpic}[grid,tics=5,width=13.5cm]{Figures/sn.pdf}								
			\put(46,-3){{\footnotesize $z=\e$}}																
			\put(11,-3){{\footnotesize $z=0$}}	
			\put(82,-3){{\footnotesize $z=0$}}	
			\put(30,53){{\footnotesize $\mathcal T_0$}}	
			\put(68,53){{\footnotesize $\mathcal D_0$}}	
			\put(-8,7){{\footnotesize $n=3$}}	
			\put(-8,28){{\footnotesize $n=2$}}		
			\put(-8,48){{\footnotesize $n=1$}}

			\put(23,45){{\scriptsize  $S_0$}}		
			\put(36,45){{\scriptsize  $S_0'$}}	
			
			\put(23,26.5){{\scriptsize  $S_0$}}		
			\put(36,26.5){{\scriptsize  $S_0'$}}

			\put(23,23){{\scriptsize  $S_1$}}		
			\put(36,23){{\scriptsize  $S_1'$}}	
			
			\put(23,8){{\scriptsize  $S_0$}}		
			\put(36,8){{\scriptsize  $S_0'$}}		
			
			\put(23,4.5){{\scriptsize  $S_1$}}		
			\put(36,4.5){{\scriptsize  $S_1'$}}
			
			\put(23,1){{\scriptsize  $S_2$}}		
			\put(36,1){{\scriptsize  $S_2'$}}

			\put(62,45){{\scriptsize  $S_0'$}}		
			\put(75,45){{\scriptsize  $S_1$}}	
			
			\put(62,26.5){{\scriptsize  $S_0'$}}		
			\put(75,26.5){{\scriptsize  $S_1$}}

			\put(62,23){{\scriptsize  $S_1'$}}		
			\put(75,23){{\scriptsize  $S_2$}}	
			
			\put(62,8){{\scriptsize  $S_0'$}}		
			\put(75,8){{\scriptsize  $S_1$}}		
			
			\put(62,4.5){{\scriptsize  $S_1'$}}		
			\put(75,4.5){{\scriptsize  $S_2$}}
			
			\put(62,1){{\scriptsize  $S_2'$}}		
			\put(75,1){{\scriptsize  $S_3$}}	
			
			\put(12,53){{\footnotesize $S_0$}}
			\put(51,53){{\footnotesize $S_0'$}}		
			\put(89,53){{\footnotesize $S_0$}}
			\put(97.5,53){{\footnotesize $S_1$}}		
			
			\put(12,33){{\scriptsize $S_0$}}	
			\put(20,32){{\scriptsize $S_1$}}	
			\put(89,33){{\scriptsize $S_0$}}	
			\put(98,33){{\scriptsize $S_1$}}	
			\put(97.5,29.5){{\scriptsize $S_2$}}	
			\put(51,33){{\scriptsize $S_0'$}}	
			\put(59,33){{\scriptsize $S_1'$}}
			
			\put(12,13){{\scriptsize $S_0$}}	
			\put(20,12){{\scriptsize $S_1$}}
			\put(20,9){{\scriptsize $S_2$}}		
			
			\put(51,12){{\scriptsize $S_0'$}}	
			\put(59,12){{\scriptsize $S_1'$}}	
			\put(59,9){{\scriptsize $S_2'$}}	
			
			\put(89,12){{\scriptsize $S_0$}}	
			\put(99,14){{\scriptsize $S_1$}}	
			\put(100,11){{\scriptsize $S_2$}}	
			\put(100,8){{\scriptsize $S_3$}}			
		\end{overpic}
		\bigskip
		\caption{Iteration scheme of the fold line $S_0$ through $\mathcal P_0$. Denote $\mathcal T_0(S_i)=S_i'$.}	\label{sn_fig}
	\end{figure} 
	
	Notice that, in $\s^c$, the iterations $S_0$, $S_1$ and $S_2$ are clockwise ordered, nevertheless, in $\s^s$, $S_0$, $S_2$, $S_1$ are counterclockwise ordered. It allows us to see that, in $\s^s$, the second iteration of $S_0$ have flipped back to the region between $S_0$ and $S_1$. Following the same scheme, we prove items $(i)$, and $(ii)$.
	
	Now, using Proposition \ref{firstreturntermos} and the dominant part of $\mathcal{P}_0$, it follows that $S_n$ accumulates onto $W^u_{\ag}$ in the $\mathcal{C}^0$-topology. 
\end{proof}
%Since $S_n$ and $W^u_{\ag}$ are tangent to the same line at the origin, it follows that $S_n$ must accumulate in $W^u_{\ag}$ as $n\rightarrow\infty$ in the $\mathcal{C}^0$-topology. 

Notice that Lemma \ref{acc_lemma} gives rise to a region $F_0$, which works as a fundamental domain for $\mathcal{P}_0$ restricted to a certain region. See Figure \ref{fund_fig}.

\begin{figure}[h!]
	\centering
	\bigskip
	\begin{overpic}[width=6cm]{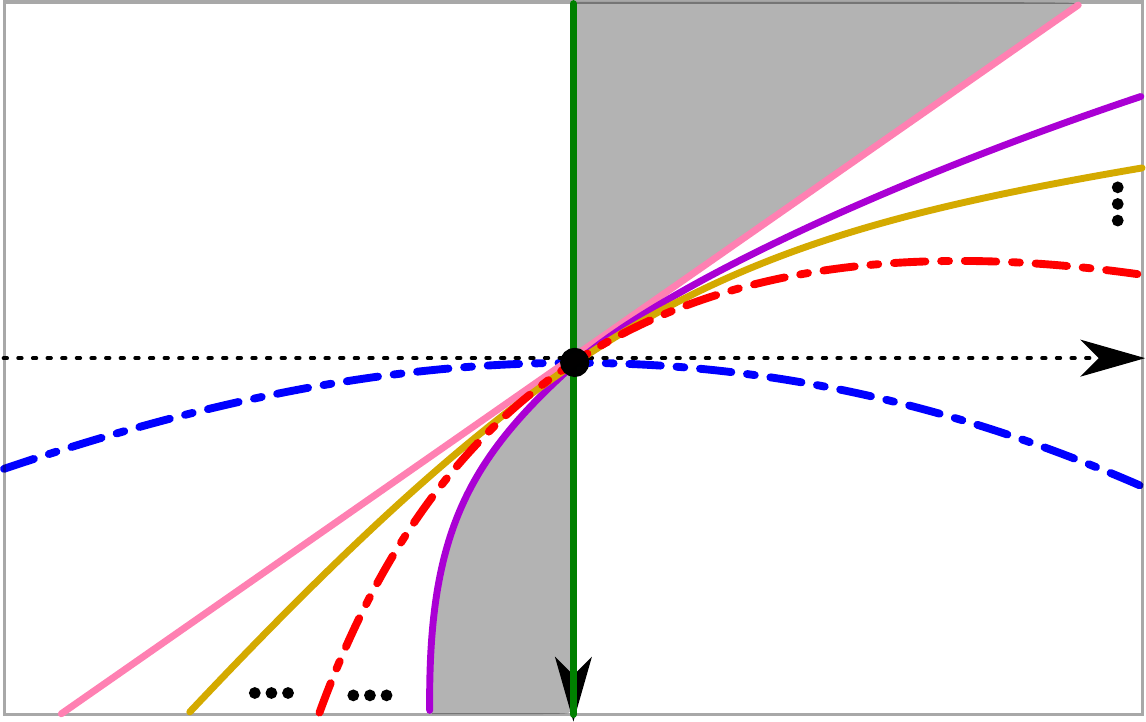}
		%	\begin{overpic}[grid,tics=5,width=8cm]{Figures/fundamental.pdf}								
		
		\put(45,59){{\footnotesize $S_0$}}	
		\put(94,59){{\footnotesize $S_1$}}	
		\put(100.5,54){{\footnotesize $S_2$}}		
		\put(100.5,47){{\footnotesize $S_3$}}		
		\put(100.5,36){{\small $W^u_{\ag}$}}		
		\put(100.5,19){{\small $W^s_{\ag}$}}

		\put(51,-3){{\footnotesize $S_0$}}	
		\put(4,-3){{\footnotesize $S_1$}}	
		\put(16,-3){{\footnotesize $S_3$}}		
		\put(37,-3){{\footnotesize $S_2$}}		
		\put(26.5,-4){{\small $W^u_{\ag}$}}		
		\put(-8,20){{\small $W^s_{\ag}$}}			
		
		\put(60,50){$F_0$}									
	\end{overpic}
	\bigskip
	\caption{Region $F_0$.}	\label{fund_fig}
\end{figure}
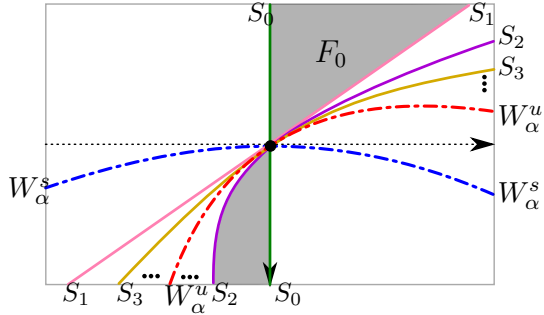 

Finally, we are able to prove Theorem \ref{modthm}. Let $\mathcal{P}_0$  and $\widetilde{\mathcal{P}_0}$ be the full first return maps associated to $Z_0$ and $\widetilde{Z}_0$, respectively, and assume that $h$ is a weak equivalence between $Z_0$ and $\widetilde{Z}_0$.
Using Proposition \ref{firstreturntermos}, we can see that there exist coordinate systems $(x,y,z)$ and $(\widetilde{x},\widetilde{y},\widetilde{z})$ at $p_0$ and $\widetilde{p_0}$, respectively, such that $\mathcal{P}_0$  and $\widetilde{\mathcal{P}_0}$ are given by
$$\mathcal{P}_0(x,y)=(\ag x-c\ag y^2+cx^2+\er_3(x,y),x),$$
and
$$\widetilde{\mathcal{P}_0}(\widetilde{x},\widetilde{y})=(\widetilde{\ag} x-\widetilde{c}\widetilde{\ag} \widetilde{y}^2+\widetilde{c}\widetilde{x}^2+\er_3(\widetilde{x},\widetilde{y}),\widetilde{x}),$$
respectively. Also, the fold lines $S_{Z_0}$ and $S_{\widetilde{Z_0}}$ are given by the $x$-axis and the $\widetilde{x}$-axis, respectively.
In this case, $\mathcal{W}(Z_0)=\ag$ and $\mathcal{W}(\widetilde{Z_0})=\widetilde{\ag}$.

Consider the same notation used in the proof of Lemma \ref{acc_lemma}. Let $\delta>0$ sufficiently small, and consider the map $\mathcal{P}_0$. There exists a unique point $w\neq (0,0)$ of $W^u_{\ag}\cap\{ y=\dg \}$, and, for each $n\in\N$, take $y_n$ as the unique point contained in $S_n\cap \{ y=\dg \}$. Therefore, from the construction above, there exists a sequence $(x_n,0)\in S_0$ such that 
\begin{enumerate}
	\item $(x_n,0)\rightarrow (0,0)$ as $n\rightarrow\infty$;
	\item $y_n=\mathcal{P}_0^n(x_n,0)$, for each $n\in\N$;
	\item $y_n\rightarrow w$ as $n\rightarrow\infty$.
\end{enumerate}

Now, for the map $\widetilde{\mathcal{P}_0}$, consider $\widetilde{w}=h(w)$, $\widetilde{x_n}=h(x_n)$ and $\widetilde{y_n}=h(y_n)$, for each $n\in\N$. Since $h$ is a weak-equivalence and $w\neq (0,0)$, it follows that
\begin{enumerate}
	\item $\widetilde{w}\neq (0,0)$;
	\item $(\widetilde{x_n},0)\rightarrow (0,0)$ as $n\rightarrow\infty$;
	\item $\widetilde{y_n}=\widetilde{\mathcal{P}_0}^n(\widetilde{x_n})$, for each $n\in\N$;
	\item $\widetilde{y_n}\rightarrow \widetilde{w}$ as $n\rightarrow\infty$.
\end{enumerate}

Notice that, since $\ag,\widetilde{\ag}\neq 0$, it follows that the dynamics of the systems near the invariant manifolds $W^u_\ag$ and  $W^u_{\widetilde{\ag}}$ of $\mathcal{P}_0$ and $\widetilde{\mathcal{P}_0}$ have the same behavior of the dynamics obtained from their linear approximations. Therefore, without loss of generality, consider that
$$\mathcal{P}_0(x,y)=(\ag x,x) \quad and \quad \widetilde{\mathcal{P}_0}(\widetilde{x},\widetilde{y})=(\widetilde{\ag} \widetilde{x},\widetilde{x}).$$

Hence, $y_n=(\ag^n x_n,\ag^{n-1}x_n)$ and $\widetilde{y_n}=(\widetilde{\ag}^n \widetilde{x_n},\widetilde{\ag}^{n-1}\widetilde{x_n}),$ for $n$ sufficiently big. Now, since $h: S_{Z_0}\rightarrow S_{\widetilde{Z_0}}$ is a diffeomorphism, it follows that $\widetilde{x_n}=K x_n+\er_2(x_n)$, for some $K\neq 0$. It follows that $\ag^n x_n\rightarrow \pi_1(w)\neq 0$ and $\widetilde{\ag}^n x_n\rightarrow \pi_1(\widetilde{w})/K \neq 0$, as $n\rightarrow 0$.

Now, if $\ag\neq \widetilde{\ag}$, then it follows that either $\ag^n x_n\rightarrow 0$ or $\widetilde{\ag}^n x_n\rightarrow 0$, which contradicts the fact that $\ag^n x_n\rightarrow \pi_1(w)$ and $\widetilde{\ag}^n x_n\rightarrow \pi_1(\widetilde{w})/K$. Therefore, it follows that $\ag=\widetilde{\ag}$, and the proof is complete.

\section{Conclusion and Further Directions}

In this paper, we have studied Filippov systems $Z_0=(X_0,Y_0)$ around a homoclinic-like loop $\Gamma_0$ at a fold-regular singularity under some generic conditions and we have proved that such loops are generic in one-parameter families. 

Also, we have seen that the fold line $S_{Z_0}$ of $Z_0$ connects to itself through orbits of $X_0$, $Y_0$ and $F_{Z_0}$ as a topological cylinder or a Möbius strip, giving rise to two classes of loops, $\Lambda_1^C$ and $\Lambda_1^M$, respectively. For simplicity, we considered only the class $\Lambda_1^C$ to avoid technicalities, nevertheless, we believe that similar results hold in the class $\Lambda_1^M$.

In the class $\Lambda_1^C$, we have seen that the first return map of $Z_0$ has a hyperbolic fixed point of either saddle (condition $(S)$) or nodal type (condition $(N)$). We have completely described the bifurcation diagram of $Z_0$ around $\Gamma_0$, provided that $Z_0$ satisfies $(N)$. If $Z_0$ satisfies $(S)$, we found all the bifurcating elements of $\Gamma_0$, nevertheless, the description of the bifurcation diagram remains as an open problem for this case. We conjecture that $Z_0$ has equivalent bifurcation diagram around $\Gamma_0$ for the cases $(N)$ and $(S)$, as can be seen in the Example \ref{examplebif}.

A natural extension of this work is to obtain bifurcation diagrams of Filippov systems around homoclinic-like loops passing through other kinds of $\s$-singularities (e.g. cusp-regular and fold-fold singularities). We highlight that the connection studied herein appears in the unfolding of loops passing through  a cusp-regular singularity. We hope that this study will guide us towards the comprehension of polycycles in $3D$ Filippov systems (see the planar version provided in \cite{AGN19}).

Also, if we relax the generic conditions imposed in the quasi-generic loops, one can certainly obtain interesting global behavior for Filippov systems $Z$ near $Z_0$. In fact, such a degeneracy of homoclinic-like loops at a fold-regular singularity might originate other bifurcating cycles.

% 
% 
%considerar contatos mais degenerados e estudar bifurcaçoes como no capitulo 2. Ex. cuspide é generica em dimensao 3.
%
%estudar as cascatas que podem ocorrer.
\section*{Acknowledgements}

OMLG is partially supported by the Brazilian FAPESP grant 2015/22762-5 and by the Brazilian CNPq grant 438975/2018-9. MAT is partially supported by the Brazilian CNPq grant 301275/2017-3.

	  \bibliographystyle{abbrv}
	  \bibliography{references_otavio.bib}

\end{document}